\documentclass[12pt,notitlepage]{amsart}
\usepackage{latexsym,amsfonts,amssymb,amsmath,amsthm,hyperref,xypic,mathrsfs}
\usepackage{graphicx}
\usepackage{color}
\usepackage[normalem]{ulem}

\usepackage{datetime2}

\let\turc\c
\usepackage{etoolbox}
\robustify\turc %new thing to get the c in Topacogullari
\renewcommand{\c}{\mathfrak{c}}
\newcommand{\bpm}{\begin{pmatrix}}
\newcommand{\epm}{\end{pmatrix}}
\newcommand{\Z}{\ensuremath{\mathbb{Z}}}
\newcommand{\Q}{\ensuremath{\mathbb Q}}
\newcommand{\R}{\ensuremath{\mathbb R}}
\newcommand{\eps}{\varepsilon}

\DeclareFontFamily{U}{mathx}{}
\DeclareFontShape{U}{mathx}{m}{n}{<-> mathx10}{}
\DeclareSymbolFont{mathx}{U}{mathx}{m}{n}
\DeclareMathAccent{\widehat}{0}{mathx}{"70}
\DeclareMathAccent{\widecheck}{0}{mathx}{"71}

\usepackage[inner=1.0in,outer=1.0in,bottom=1.0in, top=1.0in]{geometry}

\newcommand{\mz}{\ensuremath{\mathbb Z}}
\newcommand{\mr}{\ensuremath{\mathbb R}}

\newcommand{\mq}{\ensuremath{\mathbb Q}}
\newcommand{\mc}{\ensuremath{\mathbb C}}
\newcommand{\mn}{\ensuremath{\mathbb N}}
\newcommand{\mf}{\ensuremath{\mathbb F}}
\newcommand{\F}{\ensuremath{\mathbb F}}

\newcommand{\fp}{\ensuremath{\mathfrak p}}
\newcommand{\shortmod}{\ensuremath{\negthickspace \negthickspace \negthickspace \pmod}}

\newcommand{\sumstar}{\sideset{}{^*}\sum}
\newcommand{\sumprime}{\sideset{}{'}\sum}
\newcommand{\sumpm}{\sideset{}{^{\pm}}\sum}

\DeclareMathOperator{\lcm}{lcm}
\DeclareMathOperator{\Nm}{Nm}
\DeclareMathOperator{\Tr}{Tr}
\DeclareMathOperator{\Fr}{Fr}

\DeclareMathOperator{\Gal}{Gal}

\DeclareMathOperator{\Spec}{Spec}
\DeclareMathOperator{\GL}{GL}
\DeclareMathOperator{\PGL}{PGL}
\DeclareMathOperator{\St}{St}
\DeclareMathOperator{\Ind}{Ind}
\DeclareMathOperator{\disc}{disc}
\DeclareMathOperator{\Imag}{Im}
\DeclareMathOperator{\real}{Re}
\newcommand{\addcharmulti}{\theta}
\newcommand{\addchar}{\phi}
\renewcommand{\P}{\mathbb{P}}
\newcommand{\G}{\mathbb{G}}
\newcommand{\A}{\mathbb{A}}
\DeclareMathOperator{\Kl}{Kl}
\newcommand{\C}{\mathbb{C}}
\DeclareMathOperator{\Swan}{Swan}
\newcommand{\cK}{\mathcal{K}}
\newcommand{\cC}{\mathcal{C}}
\newcommand{\cO}{\mathcal{O}}
\renewcommand{\cL}{\mathcal{L}}
\newcommand{\cKl}{\mathcal{K}\ell}
\renewcommand{\cH}{\mathcal{H}}
\newcommand{\cF}{\mathcal{F}}
\newcommand{\sF}{\mathscr{F}}

\newcommand{\cG}{\mathcal{G}}

\theoremstyle{plain}		
	\newtheorem{mytheo}{Theorem} [section]
	
	\newtheorem{myprop}[mytheo]{Proposition}
	\newtheorem{mycoro}[mytheo]{Corollary}
     \newtheorem{mylemma}[mytheo]{Lemma}
	\newtheorem{mydefi}[mytheo]{Definition}
	
	\newtheorem{myremark}[mytheo]{Remark}

\theoremstyle{remark}

\numberwithin{equation}{section}
\numberwithin{figure}{section}

\begin{document}
\title{The cubic moment of $L$-functions for specified local component families}

\author{Yueke Hu} 
 \address{Yau Mathematical Sciences Center\\ Tsinghua University\\
	Beijing 100084\\
	China}
\email{yhumath@tsinghua.edu.cn}
 \author{Ian Petrow}
\address{  Department of Mathematics \\
 University College London \\
 25 Gordon Street \\
  London WC1H 0AY \\
   United Kingdom}
\email{i.petrow@ucl.ac.uk}

 \author{Matthew P. Young}
 
 \address{Department of Mathematics \\
 	  Rutgers University \\
 	 Piscataway \\
 	  NJ 08854-8019 \\
 		U.S.A.}		
 \email{mpy4@rutgers.edu}
 
 \thanks{Y.H. is supported by the National Key Research and Development Program of China (No. 2021YFA1000700). This work was supported by the Engineering and Physical Sciences Research Council award EP/W009838/1 (I.P.).
This material is based upon work supported by the National Science Foundation under agreement No.
DMS-2302210 (M.Y.). Any opinions, findings and conclusions or recommendations expressed in this material
are those of the authors and do not necessarily reflect the views of the National Science Foundation. 
 }

  \begin{abstract}
We prove Lindel\"of-on-average upper bounds on the cubic moment of central values of $L$-functions over certain families of $\operatorname{PGL}_2/\mathbb{Q}$ automorphic representations $\pi$ given by specifying the local representation $\pi_p$ of $\pi$ at finitely many primes. Such bounds were previously known in the case that $\pi_p$ belongs to the principal series or is a ramified quadratic twist of the Steinberg representation; here we handle the supercuspidal case.  
Crucially, we use new Petersson/Bruggeman-Kuznetsov forumulas for supercuspidal local component families recently developed by the authors. As corollaries, we derive Weyl-strength subconvex bounds for central values of $\operatorname{PGL}_2$ $L$-functions in the square-full aspect, and in the depth aspect, or in a hybrid of these two situations. A special case of our results is the Weyl-subconvex bound for all cusp forms of level $p^2$. Previously, such a bound was only known for forms that are twists from level $p$, which cover roughly half of the level $p^2$ forms.
  \end{abstract}

 \maketitle
\section{Introduction} 
\subsection{Statement of results}
Let $\sF$ be a collection  of arithmetic objects with associated $L$-functions $L(1/2,f)$ having analytic conductors $C(f)$.
A central program in the analytic theory of $L$-functions is the subconvexity problem, which is the
problem of showing
 that there exists $\delta>0$ such that for all $f\in \sF$ one has
$$L(1/2,f) \ll C(f)^{1/4-\delta}.$$ 
The Generalized Riemann Hypothesis (GRH) for $\sF$ implies the Lindel\"of hypothesis for $\sF$, i.e.\ subconvexity for $\sF$ with the best possible exponent, namely that any $\delta<1/4$ is admissible. A Weyl-subconvex estimate for a collection $\sF$ is the statement that any $\delta<1/12$ is an admissible saving, and represents progress on the subconvexity problem one-third of the way between convexity and Lindel\"of. 

A classical approach to the subconvexity problem is to estimate moments of $L$-functions. Suppose that $(\cF_i)$ is a sequence of finite subsets of $\sF$ with $\bigcup_i \cF_i = \sF$, and suppose for simplicity that $1 \ll C(f)/C(g) \ll 1$ for all $f, g \in \cF_i$, and all $i$. 
We write $\cF$ for a generic term in $(\cF_i)$, and we informally refer to $\cF$ as a family. Write $C(\cF)= \max_{f \in \mathcal{F}} C(f) \asymp \min_{f \in \mathcal{F}} C(f)$.
 By a moment of $L$-functions we mean (abusing terminology) a sum 
$$M_k(\cF):= \sum_{ f \in \cF} |L(1/2,f)|^k,$$ and the goal is to show asymptotic estimates for $M_k(\cF)$ as $|\cF|\to \infty$. The Lindel\"of hypothesis for $\sF$ implies for any $\cF \subseteq \sF$ that $M_k(\cF) \ll_\eps |\cF| C(\cF)^{\eps}$, and an estimate of this shape is called a Lindel\"of-on-average upper bound for $M_k(\cF)$. Sometimes, one can achieve a Lindel\"of-on-average upper bound for $M_k(\cF)$ without using GRH. If so, then for each $f_0 \in \cF$ 
 we recover a bound $L(1/2,f_{0})\ll_\eps |\cF|^{\frac{1}{k}} C(\cF)^{\eps}$ by positivity. 
 In fact, such a bound is subconvex  for $\sF$ if 
\begin{equation}
\label{heuristic}
\limsup_{\cF} \frac{1}{k} \frac{\log |\cF|}{\log C(\cF) } <\frac{1}{4}.
\end{equation}
So, for the purposes of the subconvexity problem, one should either try to study moments with larger $k$, or with smaller $|\cF|$ relative to $C(\cF)$.

In this paper we prove Weyl-strength subconvexity for many new cases of $\PGL_2/\Q$ automorphic forms: first, where $\sF$ consists of forms with finite conductor $Q$ an odd perfect square, and second, where $\sF$ consists of forms with $Q$ a power of a fixed prime $p$. Our method also uniformly treats hybrid versions of these two cases. We establish these estimates 
by studying  certain thin families of $\PGL_2/\Q$ automorphic forms that are defined in terms of the local representation theory at primes dividing the conductor.  Let us call such families $\cF$ ``local component families'' 
 (see Section \ref{sec:familydef} for a precise definition). Indeed, we show Lindel\"of-on-average upper bounds for the cubic moment $M_3(\cF)$, 
 and use as our main technical tool   
new versions of the Petersson/Bruggeman-Kuznetsov (PBK) formula for local component families that were recently developed in \cite{Hu, HPY}.

We can state our main results without making reference to local representation theory, however. 
To every local component family $\cF$, there is associated a set of spectral parameters $\Pi_\infty$ (specified in Section \ref{neighborhoods_of_reps}), and a square-full integer $Q$, i.e.\ such that $p| Q \Rightarrow p^2 | Q$, 
with the following properties: 
\begin{itemize}
\item the family $\cF$ consists of trivial central character newforms for $\GL_2$ over $\Q$, 
\item every $f \in \cF$ has conductor $Q$,
\item every $f\in \cF$ has spectral parameters lying in  $\Pi_\infty$, 
\item the size of $\cF$ satisfies $|\cF| \asymp Q^{o(1)}|\Pi_\infty|^{1+o(1)}\prod_p p^{\lceil \frac{\alpha_p}{2}\rceil}$ as $|\Pi_\infty|Q \to \infty$, where $Q= \prod_pp^{\alpha_p}$ and $|\Pi_\infty|$ is the Plancherel measure of $\Pi_\infty$. 
\end{itemize}
The last point above is a weak version of the Weyl law for $\cF$, which may be derived from \cite[Thm.\ 1.8, Props.\ 6.56, 6.58]{HPY} and \cite[Thm.\ 9.3]{knightly_kuznetsovs_2013} along the lines of loc.\ cit.\ Corollary 1.14.  
See also \cite{Palm, Knightly}  for other approaches to the Weyl law for $\cF$.

For a local component family $\cF$, define the cubic moment \begin{equation}
 \mathcal{M}(\mathcal{F})
 = 
 \sum_{\substack{f \in \mathcal{F}, \text{ cuspidal}}}
 L(1/2, f)^3
 + \mathop{\sum \int}_{\substack{f \in \mathcal{F}, \text{ Eisenstein}}}
 L(1/2, f)^3,
\end{equation}
where the sum/integral on the Eisenstein contribution indicates a continuous family of Eisenstein series for each cusp/pair of inducing characters.
\begin{mytheo}\label{MT_cubic}
There exists $B>0$ such that for any local component family $\cF$ as in Section \ref{sec:familydef}, we have  $\mathcal{M}(\cF) \ll_\eps |\Pi_\infty|^{B} \prod_p p^{\lceil \frac{\alpha_p}{2}\rceil(1+\eps)}$. If there exists $\delta>0$ such that $|s| \geq Q^\delta$ for all $s\in \Pi_\infty$, then the Lindel\"of-on-average bound $\mathcal{M}(\cF) \ll_\eps |\cF|^{1+\eps}$ holds.
\end{mytheo}
For a precise version of Theorem \ref{MT_cubic}, see Theorem \ref{MT_precise}. 

Appealing to the classification of smooth irreducible unitary representations of $\GL_2(\Q_p)$, every trivial central character cuspidal automorphic form $f$ with square-full conductor $Q$ lies in a local component family $\cF$. By the non-negativity of the central values, proved by Guo \cite{Guo}, building on Waldspurger \cite{Waldspurger}, we therefore derive the following Corollary.
\begin{mycoro}\label{MainCor}
Suppose $f$ is a cuspidal newform of square-full level $Q =\prod_p p^{\alpha_p}$, where $\alpha_2 = 0$ or $\alpha_2 \geq 1000$, with archimedean analytic conductor $Q_\infty$, and trivial central character. 
Then
\begin{equation*}
L(1/2, f) \ll_\eps Q_\infty^{\frac{1}{6}} \Big(\prod_{p} p^{\frac{1}{3}\lceil \frac{\alpha_p}{2}\rceil} \Big) (Q_\infty Q)^\eps.
\end{equation*}
\end{mycoro}
Note that the bound in Corollary \ref{MainCor} is subconvex for all $Q$ under consideration, and is sub-Burgess if $\alpha_p\geq 8$ for all $p\mid Q$. 
Let us point out two special cases of Corollary \ref{MainCor}. First, if $Q$ is an odd perfect square, then the Weyl bound \begin{equation}\label{WeylBoundSquareAspect}L(1/2,f)\ll_\eps (Q_\infty Q)^{\frac{1}{6}+\eps}\end{equation} holds. When $v_p(Q)=2$, our proof relies crucially on bounds on exponential sums over finite fields of characteristic $p$ coming from $\ell$-adic constructible sheaves (as developed by Grothendieck, Deligne, Katz, Laumon, and others). Our results are based on the Deligne-Laumon geometric Fourier transform \cite{LaumonFT},  Katz's theory of convolution sheaves \cite{Katz1988}, and  ultimately on Deligne's proof of the Weil conjectures \cite{DeligneWeil2}.

Second, if $Q=p^\alpha$ is a prime power, then the Weyl bound in depth aspect holds, namely \begin{equation}\label{WeylBoundDepthAspect}L(1/2,f)\ll_{\eps, p} (Q_\infty p^\alpha)^{\frac{1}{6}+\eps}.\end{equation} 
In particular, the bound \eqref{WeylBoundDepthAspect}  holds when $p=2$. Through the new PBK formula in \cite{HPY}, it relies on the Local Langlands Correspondence (LLC) for $\GL_2(\Q_p)$, and more specifically on Kutzko's compact induction theory (including the wild $2$-adic case).

The results in this paper generalize results in a sequence of papers \cite{ConreyIwaniec, Ivic, PetrowTwistedMotohashi, YoungHybrid, PetrowYoungWeyl, PetrowYoungCoset} as follows. 
Let $\chi$ be a primitive Dirichlet character of conductor $q$, and let $\cF_\chi$ be the family 
$$\cF_\chi =  \bigcup_{*\in \Pi_\infty} \bigcup_{m\mid q} \cH_*(m, \overline{\chi}^2) \otimes \chi,$$
where $\cH_*(m, \overline{\chi}^2)$ is the set of normalized newforms of level $m$, central character $\overline{\chi}^2$ and spectral parameter $*$. 
Note that every $f\otimes \chi \in \cF_\chi$ is a level $q^2$ newform with trivial central character that is \emph{not} twist-minimal at any prime dividing $q$.  Conrey and Iwaniec \cite{ConreyIwaniec} first studied the cubic moment $M_3(\cF_\chi)$ and achieved a Lindel\"of-on-average upper bound for it in the case that $q$ is odd square-free and $\chi$ is quadratic. Their results on $M_3(\cF_\chi)$ were extended to general $\chi$ in \cite{PetrowYoungWeyl, PetrowYoungCoset}. 

If $Q=p^\alpha$ is an odd prime power, then indeed every trivial central character 
newform that is not twist-minimal belongs to $\cF_\chi$ for some $\chi$. However, in this setting the non-twist-minimal forms only account for approximately half of the forms when $\alpha$ is even, and if $\alpha$ is odd then \emph{every} trivial central character newform is twist-minimal. The present paper addresses  
the twist-minimal cases, and in this sense now completes the work of \cite{PetrowYoungWeyl, PetrowYoungCoset}, 
 provided $\alpha\geq 2$. 

The results in this paper, however, are best stated in terms of local representation theory. Given a holomorphic cusp form, Maass cusp form, or a unitary Eisenstein series $f$, 
let $\pi$ be the associated automorphic representation and $\pi \simeq \pi_\infty \otimes \bigotimes'_p \pi_p$ its factorization as a restricted tensor product, where each $\pi_p$ is a smooth irreducible unitary generic representation of $\GL_2(\Q_p)$ (a \emph{local representation} for short). The well-known classification of local representations is that $\pi_p$ is one of:
\begin{itemize}
 \item the principal series, 
 \item the special representations,
 \item supercuspidal.
\end{itemize}
The above discussion on twist-minimality can be recast in terms of local representations. Suppose that the representation $\pi_p$ has trivial central character and $p \neq 2$. Then, $\pi_p$ is twist-minimal if and only if $\pi_p$ is supercuspidal or has conductor exponent $c(\pi_p)\leq 1$ \cite[Prop.\ 3.4]{TunnellLLC}. The same is nearly true when $p=2$ as well, see e.g.\ \cite[Prop.\ 6.37]{HPY}. So, the results of \cite{ConreyIwaniec, PetrowYoungWeyl, PetrowYoungCoset} are that a Weyl-subconvex bound for $L(1/2,\pi)$ holds if $\pi$ has trivial central character and $\pi_p$ is neither supercuspidal nor has $c(\pi_p)=1$ for any $p$.  The current paper deals with the case that $\pi$ does have supercuspidal local components. 

The reader will have noticed that the only remaining case is when for some $p$ the representation $\pi_p$ has $c(\pi_p)=1$, i.e.\ is an unramified twist of the Steinberg representation, i.e.\ when $p \| Q$ for some $p$.  Indeed, in the indicative case of the family of level $p$ newforms, there are no evident thin subfamilies (such as e.g.\ $\cF_\chi$, or a supercuspidal local component family) by which the left hand side of \eqref{heuristic} could be made small. Thus, instead of the cubic moment, one should instead pursue higher moments of $L$-functions, see e.g. \cite{KMVdK, KiralYoung, BlomerKhan} for results in this direction.

Actually, we could allow $\pi$ with $c(\pi_p)=1$ at some primes  as well as $1 \leq c(\pi_2) \leq 999$ with some extra technical work. Indeed, the paper \cite{PetrowYounghybrid} treats a hybrid family where each ramified $\pi_p$ either has $c(\pi_p)=1$, or is a quadratic twist of an unramified principal series or the Steinberg representation (these latter two cases as in the original Conrey-Iwaniec \cite{ConreyIwaniec} family). However, the bound in \cite[Cor.\ 1]{PetrowYounghybrid} fails to 
be ``locally" subconvex at the primes with $c(\pi_p) = 1$. 
We leave this aside in the present work.

\subsection{Related works}
There are a large number of important recent works on different types of generalizations of the cubic moment, including \cite{Frolenkov, NelsonCubic,  BFW, KwanI, WuXi, KwanEisenstein,  KwanIII,   GHLN}. 
The main distinguishing feature of the present work is our comprehensive treatment of specified local component families.

\subsection{Overview of the proof}\label{section:sketch}
The structure of the proof is the same as in \cite{PetrowYoungWeyl, PetrowYoungCoset}, though crucially using the new versions of the PBK formula as developed in \cite{Hu, HPY}.  These formulas associate to each local component family $\cF$ a generalized Kloosterman sum $H(m,n;c)$. See Section \ref{section:traceformula2} for explicit formulas for $H(m,n;c)$. We begin by expanding the $L$-functions in $\mathcal{M}(\cF)$ via an approximate functional equation, applying the PBK formula, and then using Poisson summation in three variables. The resulting sum has a natural local-to-global structure, so for the sake of exposition we now assume that $Q=p^\alpha$ is an odd prime power. Set $q'=p^{\lceil \frac{\alpha}{2}\rceil}$. After many simplifications and with some finite Fourier analysis, we transform $\mathcal{M}(\cF)$ into a main term plus a dual moment essentially of the form
\begin{equation}\label{sketchmoto1}
\sum_{\psi \shortmod{q'}} L(1/2, \psi)^3L(1/2,\overline{\psi}) \widehat{H}(\psi),
\end{equation}
where $\widehat{H}(\psi)$ is given by \eqref{eq:HhatpsiDef} below, namely
$$  \widehat{H}(\psi) = (q')^{-2}
 \sum_{u, x_1, x_2, x_3 \shortmod{q'}} 
 \overline{\psi}(u)
 H(\overline{u} x_1 x_2 x_3, 1 ;q') 
e_{q'}(x_1 +  x_2 +  x_3 - u  ).$$
NB: Here $\widehat{H}(\psi)$ is a mixed multivariable Fourier/Mellin transform of $H$, and is not the same as a similarly-named quantity from \cite[(1.46)]{HPY}.

Now let us assume that $\sigma_p$ is a trivial central character dihedral supercuspidal representation, and that $\cF$ is the local component family defined by $\sigma_p$. The representation $\sigma_p$ corresponds under the LLC to a pair $(L/\Q_p,\xi)$ consisting of a quadratic extension $L/\Q_p$ and a multiplicative character $\xi$ of $L$. Then, up to a factor depending only on $L$, we have 
$$\widehat{H}(\psi)\approx (q')^{-2}\sum_{\substack{ x_1, x_2, x_3 \shortmod{q'}}} 
\sum_{t \in \mathcal{O}_L/q'\cO_L}
 \psi\Big(\frac{\mathrm{Nm}(t)}{x_1 x_2 x_3}\Big)
\xi(t) e_{q'}\Big(-\mathrm{Tr}(t)+
 x_1 +  x_2 +  x_3 - \frac{x_1 x_2 x_3  }{\mathrm{Nm}(t)} \Big).$$
See \eqref{lemma:HhatDecentFormulaSupercuspidalInert} below for the precise formula.
 The sum over $x_1$ evaluates as  $\tau(\overline{\psi}) \psi\left(1- \frac{x_2x_3}{\Nm(t)}\right)$, and we recognize the resulting Gauss sum to be the root number of $\psi$. Using this root number to change one $L(1/2,\psi)$ in \eqref{sketchmoto1} to its complex conjugate, we obtain a dual moment of the form 
 \begin{equation}\label{motohashidualmoment}\sum_{\psi \shortmod{q'}} |L(1/2,\psi)|^4 g(\xi, \psi),\end{equation} 
 with 
\begin{equation} g(\xi, \psi):=\sum_{\substack{ x_2, x_3 \shortmod{q'}}} 
\sum_{t \in \cO_L/q'\cO_L} 
 \psi\Big(\frac{\mathrm{Nm}(t)}{ x_2 x_3} -1\Big)
\xi(t) e_{q'}(-\mathrm{Tr}(t) 
   +   x_2 +   x_3 ).
   \end{equation}
  The ``Motohashi-type formula'' given roughly by $M_3(\cF) = (M.T.)+ \eqref{motohashidualmoment}$ is crucial for the results of this paper, as it was in the preceding cubic moment results. For a general discussion on the Motohashi formulas and their ramifications in the theory of $L$-functions, see \cite[\S 4.5.3, 4.5.4]{MichelVenkateshGL2}. 
   
 It is pleasing to note that if in the definition of $g(\xi,\psi)$ one makes the substitutions \begin{align*} t \in \cO_L/q'\cO_L & \to (u,v) \in (\Z/q'\Z)^2 \\ \xi(t) & \to \chi(u)\overline{\chi}(v),\\ \Nm(t) &\to uv, \text{ and } \\ \Tr(t) & \to u+v,\end{align*} then by changing variables and evaluating Gauss sums, one recovers the familiar principal series character sum $g(\chi,\psi)$ from the previous papers \cite{ConreyIwaniec,PetrowYoungWeyl,PetrowYoungCoset}.
   
   We also remark that in the principal series case, the modulus $q'$ is a prime number if and only if $Q=p^2$, i.e., if and only if $\chi$ is primitive modulo $p$. In the supercuspidal case, $q'$ is prime if and only if $c(\sigma_p)=2$, i.e., if and only if $L/\Q_p$ is unramified and $\xi$ is a non-trivial character modulo $p\cO_L$.  Thus we observe that algebro-geometric methods to bound $g(\xi,\psi)$ (see Section \ref{section:AG} for details) enter the proof only for depth-zero (equiv.\ conductor-two) supercuspidals.  We find that the $L/\Q_p$ unramified case is generally more difficult to handle than the ramified case. Indeed, suppose $\sigma_p$ has trivial central character. Then  the extension $L/\Q_p$ is ramified if and only if $c(\sigma_p)$ is odd, so one sees that in this case the size of family $\cF$ is a bit larger compared to the conductor $Q$. From this perspective it is not surprising that the ramified case poses fewer problems for our analysis.

The character sum $g(\xi,\psi)$ generically enjoys square-root cancellation. For such generic arrangements of $\xi$ and $\psi$, an easy large-sieve type upper bound on $M_4( \{\psi \pmod {q'}\})$  
leads to a bound compatible with the Lindel\"of-on-average bound on $M_3(\cF)$.

On the other hand, for special arrangements of $\psi$ and $\xi$ it does happen that $g(\xi,\psi)$, equivalently $\widehat{H}(\psi)$, is large.  In particular, such special arrangements only occur when the extension $L/\Q_p$ is unramified. This phenomenon was already observed in the principal series case with $p^6\mid Q$, so it is perhaps unsurprising that the same issue appears for supercuspidal families. As in \cite{PetrowYoungCoset}, the strategy to circumvent this problem is to understand the set of $\psi$'s for which $\widehat{H}(\psi)$ does not have square-root cancellation, and to bound the fourth moment of Dirichlet $L$-functions on these sets of ``singular" characters $\psi$.

To describe the sets of singular characters $\psi$, recall that associated to a modulus $p^k$ and a character $\psi$ (resp.\ $\xi$) of $(\Z_p/p^k\Z_p)^\times$ (resp.\ $(\cO_L/p^k\cO_L)^\times$, there is an $\ell_\psi \in \Z_p$ (resp.\ $\ell_\xi \in \cO_L$) such that the Postnikov formula \eqref{eq:PostnikovThetaVersion} holds. In the principal series case, the singular characters lie in the two cosets of $\{\psi \pmod{q'}\}$ given by $\{ \psi: \ell_\psi^2 \equiv -\ell_\chi^2/4 \pmod{p^2}\}$ when $k\geq 3$, see \cite[Thms.\ 3.3, 3.4]{PetrowYoungCoset}. In particular, these cosets are non-empty only when $p\equiv 1 \pmod{4}$. In the supercuspidal case, the singular characters lie in the two cosets of $\{\psi \pmod{q'}\}$ given by $\{\psi: \ell_\psi^2 \equiv \Nm(\ell_\xi)/4\pmod{p^2}\}$ when $k\geq 2$, see Lemma \ref{lemma:HhatBoundInertRhoBound} below. In particular, these cosets are non-empty only when $\Nm(\ell_\xi)$ is a square modulo $p$. Recall (Lemma \ref{lemma:traceRisZero}) that $\Tr(\ell_\xi)=0$ when $\sigma_p$ has trivial central character. 
Thus, $\Nm(\ell_\xi)$ is a square modulo $p$ if and only if $p\equiv 3 \pmod{4}$.

 The heart of this paper is the elaborate bounds on the new supercuspidal versions of the character sum $\widehat{H}(\psi)$ in Section \ref{section:HhatCalculations}.  
In the end, the final bounds are similar in quality to those obtained in \cite{PetrowYoungCoset}, and the bound on the fourth moment of Dirichlet $L$-functions along a coset proved in \cite{PetrowYoungCoset}
is satisfactory here as well.

\subsection{Acknowledgements}
The results in this paper have been announced some years ago, including at the
Automorphic Forms Conference at the Erd\H os Center, Budapest \url{https://erdoscenter.renyi.hu/events/automorphic-forms-conference} on 5 September, 2022.  The present paper has been delayed in large part by an expansion in scope of the prerequisite paper \cite{HPY}.  We thank the Rényi Institute for its support and pleasant and productive working conditions.
We also thank the referee for a careful reading.

\section{Background tools}
\label{section:traceformula}

\subsection{The PBK formula}\label{sec:PBK}
\subsubsection{Local Fields}\label{sec:localfields}
Let $E/\Q_p$ be a finite extension. The non-archimedean local field $E$ has a ring of integers $\cO_E$ with maximal ideal $\fp=\fp_E$, for which we may choose a uniformizer $\pi_E \in \fp$. Let $\Nm, \Tr: E \to \Q_p$ be the field norm and trace, respectively. We denote the ramification index of $E/\Q_p$ by $e=e_E$  and the valuation of the discriminant by $d=d_E = v_p(\disc(E/\Q_p))$. Note that if $E/\Q_p$ is quadratic and $p\neq 2$, then $d=0$ or $1$, while when $p=2$ we must have $d=0,2,$ or $3$. Let $U_E(i)$ denote the standard filtration of $\cO_E^\times$, namely
$$ U_E(i) = \begin{cases} \cO_E^\times & \text{ if } i=0, \\ 1+\fp_E^i & \text{ if } i \geq 1.\end{cases}$$
If $E=\Q_p$, then we denote this filtration simply by $U(i)$. 

Let $\xi$ be a character of $E^\times$. The conductor exponent $c(\xi)$ of $\xi$ is the least integer $i \geq 0$  for which $\xi$ is trivial on $U_E(i)$. Likewise, if $\addchar$ is an additive character 
of $E$, then we denote by $c(\addchar)$ the least $i$ such that $\addchar$ is trivial on $\fp^i$. It will also be convenient for us to define $$c_0 = c_0(\xi) = \frac{c(\xi)}{e_E}.$$

Given a quadratic extension $L/\Q_p$, let  $\eta_L$ be the unique non-trivial quadratic character of $\Q_p^\times$  that is trivial on $\Nm(L^\times) \subseteq \Q_p^\times$. The following is standard.  
\begin{mylemma}
\label{lemma:thetarestrictedtoZp}
Suppose $L/\Q_p$ is a quadratic extension and $\eta: \mq_p^{\times} \rightarrow \mc^{\times}$ is the nontrivial quadratic character that is trivial on $\mathrm{Nm}(L^{\times}) \subseteq \mq_p^{\times}$. Then, the conductor exponent $c(\eta )=d$. 
\end{mylemma}

\begin{proof}
By local class field theory \cite[Ch.\ XIII \S4 Cor.]{SerreLocalFields}, the conductor exponent $c(\eta)$ is equal to the Artin conductor $f(\chi_1)$ of the unique nontrivial Galois character $\chi_1:\Gal(L/\Q_p) \to \{\pm 1\}$. Meanwhile, the conductor-discriminant formula (see e.g.\ loc.\ cit.\ Chapter VI \S3 Corollary 2) asserts that $\disc(L/\Q_p) = p^{f(\chi_1)+f(\chi_0)}$, where $\chi_0$ is the trivial Galois character of $\Gal(L/\Q_p)$. Since the Artin conductor $f(\chi_0)=0$, the Lemma follows.
\end{proof}
Recall that a character $\xi$ of $L^\times$ is called \emph{regular} if $\xi \neq \overline{\xi}$, where the bar denotes Galois conjugation.
\subsubsection{Neighborhood of a local representation}\label{neighborhoods_of_reps}
Let $G_p^\wedge$ be the unitary dual of $\PGL_2(\Q_p)$, i.e.\ the set of isomorphism classes of smooth irreducible unitary representations endowed with the Fell topology. Let $G_p^{\wedge,{\rm gen}}$ the subspace of generic representations. For each representation $\sigma \in G_p^{\wedge,{\rm gen}}$ we will describe a ``small'' neighborhood $\sigma[0]\subset G_p^{\wedge,{\rm gen}}$ surrounding it.

If $\sigma$ is not supercuspidal, we let $\sigma[0]$ be the connected component of $G_p^{\wedge,{\rm gen}}$ that contains $\sigma$. The non-supercuspidal connected components of $G_p^{\wedge,{\rm gen}}$ are parametrized by unitary characters $\chi$ of $\Q_p^\times$. Indeed, given $\chi$, one constructs a connected component $\cF_\chi\subset G_p^{\wedge,{\rm gen}}$ by  twisting $\chi$ by unramified quasicharacters, parabolic induction and taking irreducible sub/quotient representations. Every non-supercuspidal $\sigma \in  G_p^{\wedge,{\rm gen}}$ belongs to some $\cF_\chi$ and $\cF_{\chi_1}$ coincides with $\cF_{\chi_2}$ if and only if $\chi_1\vert_{\Z_p^\times} = \chi_2\vert_{\Z_p^\times}$ or $\chi_2^{-1}\vert_{\Z_p^\times}$.

Explicitly, if $\chi\vert_{\Z_p^\times}$ is not quadratic, then
\begin{equation}\label{sigma[0]_PS}
\cF_\chi= \{\pi(\nu \chi,\nu^{-1}\chi^{-1}): \nu \text{ unramified unitary}\},
\end{equation} 
and if $\chi\vert_{\Z_p^\times}$ is quadratic, then
\begin{multline}\label{sigma[0]_CI}
\cF_\chi= \{ \St \times \chi\} \sqcup \{\pi(|\cdot|^\sigma\chi,|\cdot|^{-\sigma}\chi) : \sigma \in (0,1/2)\} \sqcup \{\pi(|\cdot|^{it/\log p}\chi, |\cdot|^{-it/\log p}\chi): t\in [0,\pi] \} \\ \sqcup \{\pi(|\cdot|^\sigma\chi\eta,|\cdot|^{-\sigma}\chi\eta) : \sigma \in (0,1/2)\} \sqcup \{ \St \times \chi\eta\},
\end{multline} 
where $\St$ is the Steinberg representation and $\eta$ is the unramified quadratic character of $\Q_p^\times$. When $\chi\vert_{\Z_p^\times}$ is ramified quadratic we call $\cF_\chi$ the ``Conrey-Iwaniec local family'' since their important paper can be interpreted as studying the corresponding local component family.

If $\sigma \in G_p^{\wedge}$ is supercuspidal, then we let $\sigma[0]=\{\sigma , \sigma \times \eta\}$ modulo isomorphism. We also need the following explicit parametrization of supercuspidal representations. 

Recall that if $\sigma\in G_p^\wedge$ is  supercuspidal and has conductor exponent $c(\sigma)\geq 8$  when $p=2$, then $\sigma$ is dihedral \cite[\S6]{Rio}. That is, there exists a pair $(L/\Q_p,\xi)$ consisting of a quadratic field extension $L/\Q_p$ and a regular character $\xi$ of $L^\times$  such that the Weil group representation $\Ind_L^{\Q_p} \xi$ corresponds to $\sigma$ under the Local Langlands Correspondence (LLC). Let us write $\pi_\xi$ for the dihedral supercuspidal representation arising from $(L/\Q_p,\xi)$ and recall \cite[\S 1.2]{Schmidt:02a} that $\pi_\xi$ has central character $\xi\vert_{\Q_p^\times} \eta_{L}$ and conductor exponent $$c(\pi_\xi) = 2\frac{c(\xi)}{e_{L}}+d= 2c_0 + d.$$

When $p\neq 2$ the set of (isomorphism classes of) supercuspidal representations of $\GL_2(\Q_p)$ is in bijection with a subset of regular pairs $(L/\Q_p,\xi)$ called \emph{admissible} pairs (up to $\Q_p$-equivalence) \cite[\S 18.2, \S34]{BushnellHenniart:06a}.
 Thus, when $p \neq 2$ we assume throughout this paper that $(L/\Q_p,\xi)$ is an admissible pair.  When $p = 2$ the set of regular pairs $(L/\Q_2,\xi)$ (up to $\Q_2$-equivalence) with  $\xi \vert_{\Q_2^\times}= \eta_L$ and (furthermore) $2c_0+d\geq9$  is in bijection with the set of supercuspidal representations of $\PGL_2(\Q_2)$ and conductor exponent $\geq 9$, up to isomorphism \cite[Cor.\ 6.8]{HPY}.   
 When $p=2$ we assume throughout this paper that $c(\xi)$, equivalently $c(\sigma)$, is sufficiently large. Thus, in this paper we can unambiguously refer to the pair $(L/\Q_p,\xi)$ corresponding to a supercuspidal representation $\sigma$.

If $\sigma$ is supercuspidal with $\sigma \simeq \pi_\xi$ for some regular pair $(L/\Q_p,\xi)$ we define (following \cite[\S 7.4]{HPY}) for $0\leq n <c(\xi)$ the neighborhood
\begin{equation}\label{sigma[n]_def} 
\sigma[n]=\{\pi_{\xi_1}: \xi_1 \in {L^\times}^\wedge,\, c(\xi \xi_1^{-1})\leq n, \, \xi \vert_{\Q_p^\times} = \xi_1 \vert_{\Q_p^\times}\}
\end{equation}
of $\sigma$, up to isomorphism. Note when $n=0$ that \eqref{sigma[n]_def} is consistent with the previous definition of $\sigma[0]$, and that $$\#\sigma[0]= \begin{cases} 1 & \text{ if } L/\Q_p \text{ is unramified, } \\ 2 & \text{ if } L/\Q_p \text{ is ramified}. \end{cases} $$

Now let us consider the archimedean setting and give an ad-hoc definition of the neighborhood of a local representation. If $f$ is a Maass cusp form or unitary Eisenstein series of eigenvalue $\frac{1}{4}+t^2$ (resp.\ a holomorphic cusp form of weight $\kappa$), then we call $it$ (resp.\ $\frac{\kappa-1}{2}$) the spectral parameter of $f$, where we choose $t$ so that $\real(t)$ and $\Imag(t)\geq 0$.   We say that $\Pi_\infty \subseteq \C$ is a standard set of (archimedean) spectral parameters if there exists $T>1$ and $0<\eps<1$ so that $\Pi_\infty$ equals one of
\begin{itemize}
\item $(T-T^\eps,T+T^\eps)$, or 
\item $(T-T^\eps,T+T^\eps)i$, or 
\item $[\frac{1}{2},T)$, or 
\item $[0,\theta] \cup [0,T)i$,
\end{itemize}
where $\theta<1/2$ is a bound towards the Ramanujan-Selberg conjecture. Currently $\theta=7/64$ is admissible.
These sets $\Pi_\infty$ have a well-known interpretation in terms of $(\mathfrak{g},K)$-modules with $\mathfrak{g}=\mathfrak{gl}(2,\R)$  
 and $K=\mathrm{O}_2(\mr)$, making them subsets of the unitary dual of $\GL_2(\R)$, up to infinitesimal equivalence. 
Let $|\Pi_\infty|$ be the Plancherel volume of $\Pi_\infty$, in particular, we have $|\Pi_\infty| \asymp T^{1+\eps}$ in the first two cases and $|\Pi_\infty|\asymp T^2$ in the second two cases.

\subsubsection{Specified local component family}\label{sec:familydef}
Let $S$ be a finite set of finite primes. At  each $p \in S$ choose a representation $\sigma_p \in G_p^{\wedge,{\rm gen}}$, and at $\infty$, choose $\Pi_\infty$ to be one of the sets of archimedean spectral parameters from Section \ref{neighborhoods_of_reps}.

Define the specified local component family $\mathcal{F}$ associated to $\Pi_\infty, (\sigma_p)_{p\in S}$ to be 
\begin{equation}\label{eq:familydef}
\mathcal{F} = \{\pi \text{ cusp.\ aut.} : \omega_\pi =1,\, \pi_{\infty} \in \Pi_\infty,\, \pi_{p} \in \sigma_p[0] \text{ for all } p \in S, \, \pi_p \text{ unramified for all } p \not \in S\}.
\end{equation}

\subsubsection{The PBK formula for specified local components}\label{sec:pbk}
In order to derive a pleasant formula, we need to define a slightly larger family
 $\overline{\mathcal{F}} \supseteq \mathcal{F}$, as follows.  

Consider first the case that $\Pi_\infty$ equals $(T-T^\eps,T+T^\eps)i$ or  $[0,\theta] \cup [0,T)i$ for some $T,\eps$, i.e.\ corresponds to Maass forms.  
 In this case, we choose a smooth test function $h_\infty$ as in the classical Bruggeman-Kuznetsov formula (see \cite[(1.5)]{HPY}) on the spectrum
with the property that $h_\infty(t) \gg 1$ for $t \in [0,T]$ or $t \in [T-T^\eps, T  +T^{\varepsilon}]$ depending on the choice of $\Pi_\infty$, rapid decay outside this interval, and such that $h_\infty$ is nonnegative on $\mr \cup i[-1/4,1/4]$.   For example, we may take $h_\infty$ to be either
\begin{equation}
\label{eq:hdefInitialSegmentVersion}
h_\infty(t) = \frac{t^2+\frac14}{T^2} \exp\Big(-\Big(\frac{t}{T}\Big)^2\Big),
\end{equation}
or for $1\leq \Delta =T^\eps<T/100$ 
\begin{equation}
\label{eq:hdefWindowVersion}
h_\infty(t) = \frac{1}{\cosh\left(\frac{t-T}{\Delta}\right)} + \frac{1}{\cosh\left(\frac{t+T}{\Delta}\right)}.
\end{equation}

Next consider the case that $\Pi_\infty$ equals $(T-T^\eps,T+T^\eps)$ or $[\frac{1}{2},T)$ for some $T,\eps$, i.e., corresponds to holomorphic forms.  In this  section, we work with a single fixed weight $\kappa$. When applying the PBK formula in the holomorphic case, we will  work with a weighted sum over $\kappa$ (following  \cite[\S 13]{PetrowYoungWeyl}, \cite{YoungHybrid}), leading to a pleasant archimedean weight on the geometric (Kloosterman sum) side of the formula.

We also need to  enlarge our family if $2 \in S$ and $\sigma_2$ is supercuspidal. Since we are only interested in asymptotic upper bounds in this paper, we are free to enlarge our family at $2$ by any bounded size. So, we now suppose that  $\sigma_2$ has $c(\sigma_2)\geq 1000$ and trivial central character, and enlarge our family to $\sigma_2[200]\supset \sigma_2[0]$. Note that $\#\sigma_2[200]\leq 2^{201}$ (see e.g.\ \cite[Rem.\ 6.10]{HPY}), which is large but bounded as $c(\sigma_2)\to \infty$. 

Let $S_{it_j}$ be the set of trivial central character normalized cuspidal Maass newforms with spectral parameter $t_j$ (and any level). Similarly, let $S_{\kappa}$ be the set of trivial central character normalized cuspidal holomorphic weight $\kappa$ newforms (of any level). For $* = it_j$ or $\kappa$, following \cite{PetrowYoungWeyl} if $2 \in S$ and $\sigma_2$ is supercuspidal, let 
\begin{multline}\label{def:H*sigmap}\mathcal{H}_{*} ((\sigma_p)_{p \in S})= \\ \{ f  \in S_{*} :  \pi_{2,f} \in \sigma_2[200],\, \pi_{p,f} \in \sigma_p[0] \text{ for all } p \in S \text{ odd},\, \pi_{p,f} \text{ unramified for all } p \not \in S\},\end{multline}
and similarly if $2 \not \in S$ or $\sigma_2$ is not supercuspidal without the condition on $\pi_{2,f}$, and dropping the word ``odd''. Note that all forms in $\mathcal{H}_{*} ((\sigma_p)_{p \in S})$ have the same conductor, namely $Q = \prod_{p\in S} p^{c(\sigma_p)}$. 

\begin{mytheo}[PBK formula for specified local components]\label{PBKformula}
Suppose $h_\infty$, $S$, and $(\sigma_p)_{p \in S}$ are as above. Assume in addition that $c(\sigma_p)\geq 2$ for all $p\in S$, at least one of the $\sigma_p$ is supercuspidal, and if $2 \in S$ that $c(\sigma_2)\geq 1000$. Then, for all $m,n$ with $mn>0$ and $(mn,p)=1$ for all $p \in S$
we have 
\begin{equation}\label{eq:PBKformula1}
\sum_{t_j} h_\infty(t_j) \sum_{f \in \mathcal{H}_{it_j}((\sigma_p)_{p \in S})} w(f) \lambda_f(m)\overline{\lambda_f(n)} = \delta \cdot 1_{m=n} + \delta_{\rm fin} \sum_{c \equiv 0 \shortmod{q'}} \frac{H(m,n;c)}{c}H_\infty^+\Big( 4 \pi \frac{\sqrt{mn}}{c}\Big),
\end{equation}
and similarly for $\kappa \geq 2$ 
\begin{equation}
 \sum_{f \in \mathcal{H}_{\kappa}((\sigma_p)_{p \in S})} w(f) \lambda_f(m)\overline{\lambda_f(n)} = \delta \cdot \Big(1_{m=n} + 2 \pi i^{-\kappa}\sum_{c \equiv 0 \shortmod{q'}} \frac{H(m,n;c)}{c}J_{\kappa-1}\Big( 4 \pi \frac{\sqrt{mn}}{c}\Big)\Big),
\end{equation}
where 
\begin{itemize}
\item the harmonic weights $w(f)$ satisfy $w(f) = C(f)^{o(1)}$, 
\item the Hecke eigenvalues $\lambda_f(n)$ of $f$ are normalized by $\lambda_f(1)=1$, 
\item the diagonal weight $\delta = \prod_v\delta_v = \delta_\infty \delta_{\rm fin}$, is given locally by
$$\delta_{\infty} = \begin{cases} \frac{1}{4\pi} \int_{-\infty}^\infty h_\infty(t) \tanh(\pi t) t\,dt & \text{ Maass form case}, \\
\frac{\kappa-1}{4 \pi} &   \text{ weight } \kappa \text{ holomorphic form case},
\end{cases}$$
$$\delta_p = \begin{cases} 
\frac{1+p^{-1}}{1-p^{-1}} p^{c(\sigma_p)/2} & \text{ if } \sigma_p \text{ is special or principal series}, \\
p^{c_0} & \text{ if } \sigma_p \text{ is supercuspidal, } p\neq 2 \text{ and } L/\Q_p \text{ is unramified}, \\
(1+p^{-1})p^{c_0+1} & \text{ if } \sigma_p \text{ is supercuspidal, } p\neq 2 \text{ and } L/\Q_p \text{ is ramified},
\end{cases}$$
$$\delta_p = \begin{cases} 
p^{199}(1+p^{-1})p^{c_0+1} & \text{ if } \sigma_p \text{ is supercuspidal, } p= 2 \text{ and } d= 0, \\
p^{100}(1+p^{-1})p^{c_0+1} & \text{ if } \sigma_p \text{ is supercuspidal, } p= 2 \text{ and } d= 2, \\
p^{100}(1+p^{-1})p^{c_0+2} & \text{ if } \sigma_p \text{ is supercuspidal, } p= 2 \text{ and } d=3, \\
\end{cases}$$
\item the geometric conductor $q'= \prod_p p^{k_p}$ is given locally by $k_p = \lceil \frac{c(\sigma_p)}{2}\rceil$, except if $p=2$ and $\sigma_2$ is supercuspidal, in which case $k_2 = \lceil\frac{c(\sigma_2)}{2}\rceil + 199$ for $L/\Q_2$  unramified, and $k_2 = \lceil\frac{c(\sigma_2)}{2}\rceil + 100$ for $L/\Q_2$  ramified.
\item the generalized Kloosterman sum $H(m,n;c)$ is explicated in Section \ref{section:traceformula2}, and
\item the integral transform $H_\infty^+$ is as in the classical Bruggeman-Kuzentsov formula, explicitly, see \cite[(1.1)]{HPY}.
\end{itemize}
If $m,n$ are such that $(mn,p)=1$ for all $p \in S$ and $mn<0$, then an identical formula to \eqref{eq:PBKformula1} holds but with $H_\infty^+$ replaced by the function $H_\infty^-$ defined in \cite[(1.23)]{HPY}.  
\end{mytheo}
Remark: The assumption that at least one of the $\sigma_p$ is supercuspidal is merely for purposes of exposition, as this causes the Eisenstein spectrum to vanish. The more general formula can be found in \cite[(4.22)]{HPY}. PBK formulas with much narrower families of supercuspidal representations at $p=2$ are also available, see \cite[\S 7.3]{HPY}, but the statement in Theorem \ref{PBKformula} is all that is needed for this paper. 
\begin{proof}
Let us write $S= S_1 \sqcup S_2$ with $S_2$ consisting of primes $p$ at which $\sigma_p$ belongs to the Conrey-Iwaniec local family (i.e.\ $\cF_\chi$ for $\chi \vert_{\Z_p^\times}$ ramified quadratic) and $S_1= S_2^c$.

Let $f= \bigotimes_p f_p \in \mathcal{H}_{\rm fin}$ be the pure tensor with $f_p$ at $p \in S$ given by \begin{itemize}
\item $f_\chi$ as in \cite[(7.7)]{HPY} if $p \in S_1$ and $\sigma_p$ is principal series, 
\item $f_\xi$ as in \cite[(7.18)]{HPY} if $p$ is odd and $\sigma_p$ is supercuspidal,
\item $f_{\xi,200}$ as in \cite[(7.27)]{HPY} if $p=2$ and $\sigma_p$ is supercuspidal,
\item $(p+1)1_{ZK_0(p)}$ if $p \in S_2$. 
\end{itemize}
Let $h_\infty$ be given by one of \eqref{eq:hdefInitialSegmentVersion} or \eqref{eq:hdefWindowVersion}. Then \cite[Thm.\ 1.8]{HPY} applies with harmonic weights given by (1.12) of loc.\ cit.\ with $$N= \prod_{p \in S_1} p^{c(\sigma_p)} \prod_{p \in S_2}p.$$ Multiplying through the resulting formula by $\prod_{p \in S_2} \chi_p(mn)(1-p^{-1})^{-1},$ where $\chi_p$ is the Legendre symbol modulo $p$ (cf.\ loc.\ cit.\ Remark 7.7), we obtain the formula stated in Theorem \ref{PBKformula}. The claimed explicit formulas for the diagonal weights $\delta_v$, the local geometric conductors $k_p$, and the generalized Kloosterman sums $H_p(m,n;c)$ may be read off from the ``Examples'' of Section 1.2 of loc.\ cit., where we also used Remark 6.10 of loc.\ cit.\ in the $p=2$ case to compute $[\xi[200]:\xi[2-e_L]]$ numerically. 
\end{proof}

\subsection{Kloosterman sum properties}
\label{section:traceformula2}
Here we state some of the properties of the Kloosterman sums and the geometric conductor.  Let $S$ be a finite set of primes, and $\sigma_p$ be a local representation for each $p\in S$ as in Section \ref{sec:PBK}, namely, of trivial central character and conductor exponent $c(\sigma_p)\geq 2$ for all $p \in S$ and $c(\sigma_2)\geq 11$ if $2 \in S$. 

We first give some basic properties of the generalized Kloosterman sums, which hold in much greater generality (\cite[Thm.\ 3.8]{HPY}). If $c = c_1 c_2$ with $(c_1, c_2) = 1$, then
\begin{equation}
\label{eq:HCRT}
 H(m,n;c_1 c_2) = H(m \overline{c_2}, n \overline{c_2};c_1) H(m \overline{c_1}, n \overline{c_1};c_2).
\end{equation}
If $(c,p)=1$ for all $p \in S$, then 
\begin{equation*}
H(m,n;c) = S(m,n;c) = \sumstar_{x \shortmod{c}} e_c(xm + \overline{x} n), 
\end{equation*}
the standard Kloosterman sum.  On the other hand, if $c=p^j$ with $p \in S$, then
$H(m,n;p^j) = H_{\sigma_p}(m,n;p^j)$ depends on the local representation $\sigma_p$. 
The Kloosterman sums are periodic, namely 
\begin{equation}
\label{eq:KloostermanPeriodicModuloc}
H(m+ac, n+bc ; c) = H(m,n;c),
\end{equation}
for  any $m,n,a,b\in \Z$.  We also have that if $(n,c) = 1$ then
\begin{equation}
H(m,n;c) = H(mn,1;c).
\end{equation}

Next we give explicit formulas for $H_{\sigma_p}(m,n;p^k)$ in the specific cases of interest here.
If $\sigma_p$ is isomorphic to a principal series representation  $ \pi(\chi,\chi^{-1})$ or a ramified twist of the Steinberg representation $\St \times \chi$, then the geometric conductor exponent $k_p=v_p(q')$ is given by $k_p= \lceil c(\sigma_p)/2\rceil = c(\chi)$. We have 
\begin{equation}
\label{eq:KloostermanDefPScase}
 H_{\sigma_p}(m,n;p^k) = 
 \chi(m) \overline{\chi}(n)
 \sumstar_{t \shortmod{p^k}} \chi^2(t) e_{p^k}(tm + \overline{t} n)
\end{equation}
when $1\leq c(\chi)\leq k$ and $p \nmid mn$, and $H_{\sigma_p}(m,n;p^k) = 0$ when $k<c(\chi)$ or $p\mid mn$. 

Now suppose that $\sigma_p$ is a supercuspidal representation. By our assumptions, there exists a unique admissible pair $(L/\Q_p, \xi)$ with $\overline{\xi} \neq \xi$ 
and $\xi \vert_{\Q_p^\times} = \eta_{L}$ such that $\sigma_p \simeq \pi_\xi$ corresponds to $\Ind_L^{\Q_p} \xi$ under the LLC.  Let $\gamma=\lambda(L,\addchar)$ be the Langlands lambda function associated to $L$ and the additive character $\addchar = e(\{\cdot\}_p)$ of $\Q_p$ as in \cite[Lem.\ 1.2(iv)]{JacquetLanglands}. For this paper, we only need to know that $\gamma$ is independent of $\xi$ and satisfies $|\gamma|=1$. We have that
 \begin{equation}
\label{eq:KloostermanDefSCinertCaseGeneral}
H_{\sigma_p}(m,n;p^{k}) = 
 \overline{\gamma} 
p^{-\frac{d}{2}}
\sum_{\substack{t \in (\mathcal{O}_L/(p^k))^{\times} \\ \mathrm{Nm}(t) \equiv mn \shortmod{p^{k}}}} \xi(t) e_{p^{k}}(- \mathrm{Tr}(t)),
\end{equation}
when $k\geq v_p(q')$ and $p \nmid mn$, and that $H_{\sigma_p}(m,n;p^{k}) = 0$ otherwise. 

For the reader's convenience, we present the following table of data around trivial central character dihedral supercuspidal representations, for $p$ odd:
\begin{equation}
\label{eq:supercuspidalTable}
\begin{tabular}{c|c|c|c|c}
$e_L$ & $c(\pi_{\xi})$ & $c(\xi)$ & $c_0$ & $v_p(q')$  \\
\hline
$2$ & $2n+1$ & $2n$ & $n$ & $c_0+1$  \\
$1$ & $4n$ & $2n$ & $2n$ & $c_0$  \\
$1$ & $4n+2$ & $2n+1$ & $2n+1$ & $c_0$
\end{tabular}
\end{equation}
and  for our enlarged families when $p=2$ and $c(\pi)\geq 1000$:
\begin{equation}
\label{eq:supercuspidalTablepEquals2}
\begin{tabular}{c|c|c|c|c|c}
$e_L$ & $d$ & $c(\pi_{\xi})$ & $c(\xi)$ & $c_0$ & $v_p(q')$  \\
\hline
$1$ & 0 & $2n+2$ & $n+1$ & $n+1$ & $c_0+ 199$  \\
$2$ & 2 & $2n+2$ & $2n$ & $n$ & $c_0 + 101$  \\
$2$ & 3 & $2n+1$ & $2n-2$ & $n-1$ & $c_0 + 102$. 
\end{tabular}
\end{equation}

\subsection{Postnikov formula}
For a $p$-adic field $L$, let us recall the $p$-adic logarithm \cite[(5.4), (5.5) Props.]{Neukirch}.
\begin{mylemma}\label{lemma:padiclog}
For $1+x \in U_L(1)$, the series $$\log (1+x ) = x -\frac{x^2}{2} + \frac{x^3}{3} - \cdots$$ converges and defines a continuous homomorphism called the $p$-adic logarithm $\log: U_L(1) \to L$. If moreover $i > e_{L/\mq_p}/(p-1)$, then $\log$ restricts to an isomorphism of topological groups $\log : U_L(i) \to \fp_L^i.$ 
\end{mylemma}

We have the following formula, quoting from \cite[Lemma 6.1]{HPY}.
\begin{mylemma}[Postnikov]
\label{lem:postnikov}
  Let $\xi$ and $\addchar$ be continuous characters of $L^{\times}$ and $L$, respectively.  If $i > e/(p-1)$ and  $c(\xi) \geq \max(i,2)$, then there exists a unique $\alpha_{\xi} \in \pi_{L}^{-c(\xi) + c(\addchar)} (\mathcal{O}_L/ \mathfrak{p}^{c(\xi )-i}  )^{\times}$ such that
\begin{equation}
\label{eq:postnikov}
\xi(1+u) = \addchar(\alpha_{\xi} \log(1+u)), \qquad \text{for all $u \in \mathfrak{p}^i$}.
\end{equation}
\end{mylemma}
If $\addchar$ is an additive character of $\Q_p$, then $\addchar_L = \addchar \circ \Tr$ is an additive character of $L$, with conductor given by $$c(\addchar_L) = e_L \cdot c(\addchar) - \frac{d}{f},$$ where $f= f_{L/\Q_p}$ is the degree of the residue field extension, see e.g.\ \cite[2.3.1 Lem.]{Schmidt:02a}. 

We will typically apply Lemma \ref{lem:postnikov} as follows.  Suppose $[L:\Q_p]=2$ and $c_0 = c(\xi)/e_L$.  For any integer $k \geq c_0$, the character $\xi \vert_{\cO_L^\times}$ factors through $\left( \cO_L/ (p^k)\right)^\times$. It will be convenient to also use an additive character modulo $p^k$, so suppose now that $\addchar(x)=e(\{x\}_p)$ is the standard additive character, take $\alpha_\xi$ as in Lemma \ref{lem:postnikov} with respect to $\addchar_L$, and then define 
\begin{equation}
\label{eq:Rdef} \ell_\xi = p^k \alpha_\xi .\end{equation} Now Lemma \ref{lem:postnikov} gives that 
\begin{equation}
\label{eq:PostnikovThetaVersion}
\xi(1+u) = e_{p^k}\left(\mathrm{Tr} (\ell_\xi \log(1+u)\right) \quad \text{for all } u \in \mathfrak{p}^i,
\end{equation}
with $e_{p^k}(x) := \phi(x/p^k)$. Note that 
\begin{equation}
\label{eq:valuationNormR}
v_L(\ell_\xi) = e_L \cdot (k-c_0)-d \quad \text{ and } \quad v_p(\mathrm{Nm}(\ell_\xi)) = 2(k-c_0)  - d.
\end{equation}

\begin{mylemma}
\label{lemma:traceRisZero}
Suppose $[L:\Q_p]=2$ and that $\xi$ is a character of $L^{\times}$ that is trivial on $\mathrm{Nm}(L^{\times}) \subseteq \mq_p^{\times}\subseteq L^\times $. 
Let $i$ be the minimal positive integer such that $i > e/(p-1)$, and suppose $c(\xi) \geq \max(i,2)$.
There exists $\ell_\xi \in L$ such that 
  \eqref{eq:PostnikovThetaVersion} holds for all $u \in \mathfrak{p}^i$ and such that $\mathrm{Tr}(\ell_\xi) = 0 $ if $p$ odd or $L/\Q_2$ is unramified, while $\mathrm{Tr}(\ell_\xi) \equiv  0 \pmod{2^{k-d}}$ if $L/\Q_2$ is ramified.
\end{mylemma}
\begin{proof}
Using that $\xi$ is trivial on $\mathrm{Nm}(L^{\times})$, we have
\begin{equation}
\label{eq:1equalsthetathetabar}
1= \xi(1+u) \xi(1 + \overline{u}) = 
e_{p^k}(\mathrm{Tr}(\ell_\xi \log(1+u))+\mathrm{Tr}(\ell_\xi \log(1+\overline{u}))),
\end{equation}
where $\overline{u}$ is the Galois conjugate of $u$.
Note $\mathrm{Tr}(\ell_\xi \log(1+u))+\mathrm{Tr}(\ell_\xi \log(1+\overline{u})) = \mathrm{Tr}(\ell_\xi) \cdot \mathrm{Tr} \log(1+u)$, so 
\eqref{eq:1equalsthetathetabar} implies
\begin{equation}
\label{eq:traceinclusionproperty}
\mathrm{Tr}(\ell_\xi) \cdot \mathrm{Tr}(\log(1+u)) 
\in p^k \mz_p, \qquad \text{for all $u \in \mathfrak{p}^i$}.
\end{equation}
By Lemma \ref{lemma:padiclog}, $\log(1+ \mathfrak{p}^i) = \mathfrak{p}^i$ for any $i>e/(p-1)$, the minimal such $i$ being, explicitly, 
$$ 
\begin{cases} 
1 & \text{ if } (d=0 \text{ and } p>2) \text{ or }  (d=1 \text{ and }  p>3), \\ 
2 & \text{ if }(d=1 \text{ and } p=3) \text{ or } (d=0 \text{ and } p=2), \\
3 & \text{ if }d=2,3 \text{ and } p=2.\end{cases}$$
 By \cite[Ch.\ III (2.9) Thm.]{Neukirch}
the inverse different
is given by 
 $$\mathcal{D}^{-1}_{L/\Q_p}: = \{ x \in L : \Tr (xy) \in \Z_p, \, \forall y \in \cO_L\}= \fp_L^{-d},$$ 
 so that 
\begin{equation}
\label{eq:TraceImage}
\mathrm{Tr}(\mathfrak{p}^n)
=
\begin{cases}
p^n \mz_p, \qquad &\text{$L$ unramified}, \\
p^{\lfloor\frac{n+d}{2} \rfloor} \mz_p, \qquad &\text{$L$ ramified}.
\end{cases}
\end{equation}
Hence \eqref{eq:traceinclusionproperty} is equivalent to
\begin{equation*}
\mathrm{Tr}(\ell_\xi)
\in 
\begin{cases}
p^{k-i} \mz_p,  &\text{($L$ unramified}), \\
p^{k-\lfloor\frac{i+d}{2} \rfloor} \mz_p,  &\text{($L$ ramified)}
\end{cases} =  
\begin{cases} (p)^{k-1} & p\neq 2, \\ (p)^{k-2}, & d=0,2, \thinspace p=2, \\ (p)^{k-3} & d=3, \thinspace p=2. 
\end{cases}
\end{equation*}
Meanwhile, $\ell_\xi$ is only defined up to $\fp^{ek-d-i}$ by Lemma \ref{lem:postnikov}, i.e.\ by \eqref{eq:TraceImage} again, $\Tr(\ell_\xi)$ is only defined up to $(p)^{\lfloor k-\frac{i}{e}\rfloor}$, explicitly, up to $(p)^{k-1}$ if $p \neq 2$, and up to $(p)^{k-2}$ if $p=2$. Thus, we can choose $\ell_\xi \in L$ as claimed.   
\end{proof}

\begin{myremark}\label{rem:fixedL}
For the rest of the paper, we assume that $\mathrm{Tr}(\ell_\xi) = 0$ when $p \neq 2$ and $\xi$ is as in Lemma \ref{lemma:traceRisZero}, as this will simplify numerous computations.
\end{myremark}
\begin{myremark}\label{rm:normchar}
Let $\chi_L$ be defined by $\chi_L = \chi \circ \Nm$ for $\chi$ a character of $\Q_p^{\times}$. 
Since $\log \Nm(x) = \Tr(\log x)$, we have that 
$\alpha_\chi = \alpha_{\chi_L} \in \Q_p \subseteq L$ with respect to the additive character $\addchar_L = \addchar \circ \mathrm{Tr}$. 
\end{myremark}

\begin{mylemma}
\label{lemma:vpNormOfEllTheta}
Suppose $p$ is odd and let $(L/\Q_p, \xi)$ be an admissible pair. Let $k\geq v_p(q')$, where $q'$ is the geometric conductor of the supercuspidal representation corresponding to $(L/\Q_p, \xi)$ under the LLC. 
\begin{enumerate}
\item If $L/\Q_p$ is ramified, then $v_p(\mathrm{Nm}(\ell_\xi)) = 2(k-c_0) - 1 > 0$, and 
\item if $L/\Q_p$ is unramified, then $v_p(\mathrm{Nm}(\ell_\xi)) = 2(k-c_0)\geq 0$, i.e., $p \nmid \Nm(\ell_\xi)$ for $k=c_0$.
\end{enumerate}
\end{mylemma}
\begin{proof}
These assertions follow directly from the table  \eqref{eq:supercuspidalTable}.
\end{proof}

\subsection{The multivariable Postnikov formula over $\Q_p$}

We quote the following material from \cite{PetrowYoungCoset}, making simplifications when possible. 

A rational function $f \in \Z(t)$ is an equivalence class of pairs of polynomials $f_1/f_2$ with integer coefficients and $f_2$ not identically zero. An integer $t_0$ is said to be in the domain of $f$ if $f_2(t_0)\neq 0$ with $f=f_1/f_2$ and $f_1$ and $f_2$ coprime. Meanwhile, a rational function $f \in (\Z/p^\beta \Z)(t)$ is an equivalence class of pairs of polynomials $f_1/f_2$ with coefficients in $\Z/p^\beta\Z$ and with $p$ not dividing all of the coefficients of $f_2$. Similarly, $t_0 \in \Z/p^\beta\Z$ is said to be in the domain of $f$ if $p \nmid f_2(t_0)$ with $f=f_1/f_2$ in lowest terms. 
 The above notions also extend naturally to several variables. Lastly, in a character sum  of the form $\sumstar_t \chi(f(t)) \addchar(g(t))$, the $*$ is always taken to mean that we sum over those $t$ lying in the intersection of the domains of $f$ and $g$.

 Let $g \in \Z(t)$ be a rational function whose reduction $\overline{g}$ modulo $p^\beta$ exists.  Let $t_0$ be an integer whose reduction modulo $p^\beta$ lies in the domain of $\overline{g}$. Then it is easy to see that $v_p\big(\frac{g^{(n)}(t_0)}{n!}\big) \geq 0$ for all $n \geq 0$.
 In particular, this shows that 
\begin{equation}
\label{eq:rationalfunctionTaylorExpansion}
g(x_0 + p^\beta x_1) \equiv g(x_0) + p^\beta  g'(x_0) x_1  \pmod{p^{2\beta}},
\end{equation}
for any integer $x_0$ reducing to the domain of $\overline{g}$.
More generally, if 
$p$ does not divide {all the coefficients of the denominator of $g \in \Z(t_1,\ldots,t_n)$ and $x_0 \in \Z^n$ reduces modulo $p^\beta$ to lie in the domain of $\overline{g}$, then
 we have 
\begin{equation}
\label{eq:rationalfunctionTaylorExpansionMultiVariableVersion}
 g(x_0 + p^\beta x_1) \equiv g(x_0) + p^\beta g'(x_0) x_1 + p^{2\beta} \tfrac12 g''(x_0)[x_1] \pmod{p^{3\beta}},
\end{equation}
where $g'$ is the gradient of $g$ and $g''$ is the Hessian matrix, and $A[x]= x^\intercal A x$ is the quadratic form associated to a square matrix $A$ and evaluated at $x$.

  For $f_i \in \Z(t_1,\ldots, t_n)$, $i = 1, \ldots d$, let  $f = (f_1, \ldots, f_d) \in \Z(t_1,\ldots, t_n)^{d}$
 be the associated $d$-tuple of rational functions.  
 For such an $f$ we have the associated $d \times n$ Jacobian matrix, which we denote by $f' \in M_{d \times n} (\Z(t_1,\ldots, t_n))$. Similarly, we have the logarithmic Jacobian $(\log f)'$, where the $ij$ entry is given by 
 $\partial_j f_i/f_i$.

Define an additive character $\addcharmulti$ modulo ${\bf q} = (q, \dots, q) \in \mn^d$ as a group homomorphism $\mz^d/q\mz^d \rightarrow \mc^{\times}$, lifted to $\mz^d$ by periodicity.
Then there exists a unique $a_{\addcharmulti} \in \mz^d/q\mz^d$ such that $\addcharmulti(n) = e_q(a_{\addcharmulti} n)$ where $a_{\addcharmulti} n$ is the standard scalar product. 
Likewise, a Dirichlet character modulo ${\bf q} = (q, \dots, q) \in \mn^d$ is a map $(\mz^d/q\mz^d)^{\times} \rightarrow \mc^{\times}$ extended to $\mz^d$ in the natural way.  Again, $\chi$ may be expressed uniquely as 
  $  \chi((n_1,\ldots, n_d)) = \chi_1(n_1) \cdots \chi_d(n_d)$, where each
 $\chi_i$ is a Dirichlet character modulo $q$.  
 If $p$ is odd  and  $q=(p^{\beta}, \dots, p^{\beta})$ with  $\beta \geq 2$, we  define $\ell_{\chi} = (\ell_{\chi_1}, \ldots, \ell_{\chi_d})$ with $\ell_{\chi_i}$ as in Lemma \ref{lem:postnikov} with $L=\Q_p$.  Note that  the Postnikov formula generalizes to give for $n = (n_1, \dots, n_d)$ with each $n_i \equiv 1 \pmod{p}$ the formula
  $\chi(n) = e_{p^{\beta}}(\ell_{\chi} \log_p(n))$, with the standard scalar product and where $\log_p(n) = (\log_p(n_1), \dots, \log_p(n_d))$.
  
  \begin{mylemma}
  \label{lemma:characterSumPrimePowerEvenExponent}
  Let $p$ be an odd prime, 
$\chi$ be a Dirichlet character modulo $(p^{2 \alpha}, \dots, p^{2 \alpha})$, 
$\addcharmulti$ be an additive character modulo $(p^{2 \alpha}, \dots, p^{2 \alpha})$ and $f, g \in \Z(t_1,\ldots, t_n)^d$  as above. Then 
\begin{equation}
\label{eq:characterSumPrimePowerEvenExponent}
S:=\sumstar_{t \in \left( \Z/p^{2\alpha} \Z \right)^n } \chi(f(t)) \addcharmulti(g(t)) = p^{n\alpha}\sumstar_{\substack{t_0 \in \left( \Z/p^\alpha \Z \right)^n \\ \eqref{eq:characterSumCongruenceCondition} \text{ holds}}} \chi(f(t_0)) \addcharmulti(g(t_0)),
\end{equation}
where
\begin{equation}
\label{eq:characterSumCongruenceCondition}
\ell_{\chi} (\log f )'(t_0) + a_{\addcharmulti} g'(t_0) \equiv 0 \shortmod {p^\alpha}.
\end{equation}

Each summand on the right hand side of \eqref{eq:characterSumPrimePowerEvenExponent} does not depend on the choice of 
lift of $t_0$ to $\Z_p^d$.
  \end{mylemma}
\begin{proof}
 See \cite[Lem.\ 2.2]{PetrowYoungCoset}.
 \end{proof}
  Next we consider the odd exponent case.  To this end, we introduce multi-variable Gauss sums.  Let $\mathcal{L}: \mz^n \rightarrow \mz$ be a linear form with integer coefficients, and $\mathcal{Q}: \mz^{n} \rightarrow \mz$ be a quadratic form. 
 Define
 \begin{equation}
 \label{eq:GpDef1}
  G_p(\mathcal{Q},\mathcal{L}) = \sum_{t \in \F_p^n} e_p(\mathcal{Q}[t] + \mathcal{L} t).
 \end{equation}
  
 \begin{mylemma}
 \label{lemma:characterSumPrimePowerOddExponent}
  Let $p$ be an odd prime, 
$\chi$ be a Dirichlet character modulo $(p^{2 \alpha+1}, \dots, p^{2 \alpha+1})$, $\addcharmulti$ be an additive character modulo $(p^{2 \alpha+1}, \dots, p^{2 \alpha+1})$,  and $f, g \in \Z(t_1,\ldots, t_n)^d$ as above. Then 
\begin{equation}
\label{eq:characterSumPrimePowerOddExponent}
S=
\sumstar_{t \in \left( \Z/p^{2\alpha+1} \Z \right)^n } \chi(f(t)) \addcharmulti(g(t)) 
= 
p^{n\alpha}
\sumstar_{\substack{t_0 \in \left( \Z/p^\alpha \Z \right)^n \\ \eqref{eq:characterSumCongruenceCondition} \text{ holds}}}
\chi(f(t_0)) \addcharmulti(g(t_0)) G_p(\mathcal{Q},\mathcal{L}),
\end{equation}
where 
\begin{equation}
\label{eq:Mlinearformdef}
 \mathcal{L} = p^{-\alpha} (\ell_{\chi} (\log f)'(t_0) + a_{\addcharmulti} g'(t_0)),
\end{equation}
and $\mathcal{Q}$ is the quadratic form with associated matrix (in the standard basis for $\mz^n$) given by
\begin{equation}
\label{eq:QquadraticformDef}
 \mathcal{Q} = \tfrac12 \ell_{\chi} (\log f)''(t_0) + \tfrac12 g''(t_0).
\end{equation}
Each summand on the right hand side of \eqref{eq:characterSumPrimePowerOddExponent} does not depend on the choice of 
lift of $t_0$ to $\Z_p^d$.
  \end{mylemma}
  \begin{proof}
 See \cite[Lem.\ 2.3]{PetrowYoungCoset}.
 \end{proof}

\begin{myremark}
\label{rmk:PostnikovRemarkNumberField} 
 In Section \ref{sec:supercuspidal_simplifications} we will be faced with a situation in which one would like to apply Lemmas \ref{lemma:characterSumPrimePowerEvenExponent} and \ref{lemma:characterSumPrimePowerOddExponent} but where one variable is in $\cO_L$ instead of $\Z_p$. Rather than pursue a generalization of these lemmas, we instead opt to treat these sums ``by hand'' using additive orthogonality of characters for a quadratic extension $L/\Q_p$ of $p$-adic fields:
 for all $a \in \mathcal{O}_L$, we have
\begin{equation}\label{eq:addorth}
\sum_{t \in \mathcal{O}_L/(p^k)} 
e_{p^k}(\mathrm{Tr}(at))
= p^{2k} \cdot
\begin{cases}
\delta(p^k | a), \quad &\text{$L$ unramified}, \\
\delta(\mathfrak{p}^{2k-d} | a), \quad &\text{$L$ ramified}.
\end{cases}
\end{equation}
\end{myremark}

In view of Lemma \ref{lemma:characterSumPrimePowerOddExponent}, it will be useful to estimate quadratic Gauss sums.  

If $\mathcal{Q}$ is a finite-dimensional quadratic form over $\F_p$ with $p$ odd, then there exists a symmetric matrix $M$ with entries in $\F_p$ such that $\mathcal{Q}[x] = x M x^{\intercal}$. The kernel of $M$ is called the \emph{radical} of $\mathcal{Q}$.
\begin{mylemma}
\label{lemma:quadraticGaussSum}
 Let $p$ be an odd prime, let $\mathcal{Q}$ be a quadratic form over $\mf_p$, and let $\mathcal{L}$ be a linear form, as above.  Let $V$ be the  radical of $\mathcal{Q}$.  Let $r_{\mathcal{Q}}$ denote the rank of $\mathcal{Q}$.  Then $G_p(\mathcal{Q},\mathcal{L})$ vanishes unless $\mathcal{L} \vert_V=0$, in which case 
 \begin{equation*}
  |G_p(\mathcal{Q},\mathcal{L})|  = p^{n-\frac{r_{\mathcal{Q}}}{2}}. 
 \end{equation*}
\end{mylemma}
\begin{proof}
 See \cite[Lem.\ 2.4]{PetrowYoungCoset}.
 \end{proof}
\begin{mylemma}
\label{lemma:quadraticGaussSumEvaluation}
Let $p$ be an odd prime, and let $\mathcal{Q}$ be a non-degnenerate quadratic form over $\mf_p$ of rank $r$, and let $\mathcal{D}$ denote the discriminant of $\mathcal{Q}$.  Then
\begin{equation*}
G_p(\mathcal{Q}, 0) = \epsilon_p^r p^{r/2} \Big(\frac{\mathcal{D}}{p}\Big),
\qquad
\text{where }
\epsilon_p = 
\begin{cases}
1, \quad &p \equiv 1 \pmod{4} \\
i, \quad &p \equiv 3 \pmod{4}.
\end{cases}
\end{equation*}
\end{mylemma}
\begin{proof}
Suppose $\mathcal{Q}[x] = x M x^{\intercal}$.  Then $\mathcal{D} = \det(M)$, up to multiplication by a square in $\mf_p^{\times}$.  It is well-known that $\mathcal{Q}$ can be diagonalized; that is, there is a nonsingular matrix $U$ such that $U M U^{\intercal}$ is diagonal, say of the form $\mathrm{diag}(a_1, a_2, \dots, a_r)$ with each $a_i$ nonzero.  Note $a_1 \dots a_r = \det(U M U^{\intercal}) = \det(M) \det(U)^2$.  In these coordinates, we easily evaluate 
$G_p(\mathcal{Q}, 0) = \epsilon_p^r p^{r/2} (\frac{a_1 \dots a_r}{p})$, by say \cite[Section 3.5]{IK}.  This gives the claimed result.
\end{proof}

\subsection{Coset fourth moment}
We quote the following coset bound for the fourth moment of Dirichlet $L$-functions.
\begin{mytheo}[\cite{PetrowYoungCoset}, Theorem 1.4]
\label{thm:fourthcoset}
Let $q \geq 1$, and let $q_0$ be the least positive integer such that $q_0 | q$ and $q^2 | q_0^3$.  
 Let $\psi$ have conductor $q$.  Then for any $d|q$, $\sigma \geq 1/2$, and $X \geq 1$ we have
 \begin{equation*}
\int_{-X}^{X}  \sum_{\chi \shortmod{d}} |L(\sigma+it, \chi.\psi)|^4 dt
\ll X \lcm(d, q_0) (qX)^{\varepsilon}.
 \end{equation*}
\end{mytheo}

\section{Broad overview}
\subsection{Summary of \cite{PetrowYoungWeyl}}

For $T\geq 1, 1\leq \Delta <T/100$, $h_\infty$ given by either \eqref{eq:hdefInitialSegmentVersion} or \eqref{eq:hdefWindowVersion}, and a list of representations $(\sigma_p)_{p\in S}$, the goal of this paper is to estimate 
\begin{equation}\label{Mdef}\mathcal{M}^\pm((\sigma_p)_{p\in S}) := \sum_{t_j} h_\infty(t_j) \sumpm_{f \in \mathcal{H}_{it_j}((\sigma_p)_{p \in S})} w(f) L(1/2,f)^3 + (\text{Eis.}),\end{equation} 
where the $\pm$ indicates that the sum is over forms with even (resp.\ odd) parity and otherwise the notation is as in Theorem \ref{PBKformula}. 
The (harmonically weighted) size of the family in $\mathcal{M}^\pm((\sigma_p)_{p\in S})$ asymptotically equals $\frac{1}{2}\delta =\frac{1}{2}\delta_\infty\delta_{\rm fin}$, and $\delta_\infty$ satisfies 
$$\delta_\infty \asymp T^2 \quad \text{ or } \quad \delta_{\infty} \asymp \Delta T$$ 
when $h_\infty$ is given by \eqref{eq:hdefInitialSegmentVersion}, \eqref{eq:hdefWindowVersion}, respectively. We also have that $\delta_{\rm fin} = \prod_p \delta_p$ with $\delta_p \asymp p^{\lceil \frac{c(\sigma_p)}{2}\rceil}$ with absolute implied constants, where $Q = \prod_p p^{c(\sigma_p)}$ is the common (finite) conductor of the forms appearing in $\mathcal{M}^\pm((\sigma_p)_{p\in S})$. 
\begin{mytheo}\label{MT_precise}
Suppose that  $c(\sigma_p)\geq 2$ for all $p\in S$ and $c(\sigma_2)\geq 1000$ if $2 \in S$.
\begin{enumerate}
\item If $h_{\infty}$  is given by \eqref{eq:hdefInitialSegmentVersion}, then there exists $B>2$ such that $\mathcal{M}^\pm((\sigma_p)_{p\in S}) \ll_\eps T^B \delta_{\rm fin}^{1+\eps}.$
\item If $h_{\infty}$  is given by \eqref{eq:hdefWindowVersion} with $\Delta = T^\eps$  and there exists  $\delta >0$  such that  $T\gg Q^\delta$, then $\mathcal{M}^\pm((\sigma_p)_{p\in S}) \ll_\eps (T\delta_{\rm fin})^{1+\eps}.$
\end{enumerate}
\end{mytheo}
As a preliminary simplification, we may assume here that $\sigma_p$ is supercuspidal for at least one $p\in S$, as the complementary case was already covered in \cite{ConreyIwaniec, PetrowYoungWeyl, PetrowYoungCoset}.  Thus, the Eisenstein series term in \eqref{Mdef} vanishes and we omit it from all further discussion. 

\begin{mylemma}[Approximate Functional Equation]\label{AFE}
Let $f$ be a Maass form with trivial central character, conductor $Q$ and spectral parameter $t_f$. We have
\begin{equation}
\label{eq:LcubedAFE}
L(1/2, f)^3 = 
\sum_{(d, Q) = 1} \frac{4}{d}
\sum_{n_1, n_2, n_3} \frac{\lambda_f(n_1) \lambda_f(n_2 n_3)}{\sqrt{n_1 n_2 n_3}}
V_1\Big(\frac{n_1}{\sqrt{Q}}, t_f\Big)
V_2\Big(\frac{n_2 n_3 d^2}{Q}, t_f \Big),
\end{equation}
where 
$V_1$ and $V_2$ are the smooth functions given in \cite[(4.3)]{PetrowYoungWeyl}. 
\end{mylemma}
Remark. 
In the Maass form case, $V_1$ and $V_2$ depend on the parity $\delta$ of $f$.
If on the other hand, $f$ is a weight $\kappa$ holomorphic cusp form, the same formula \eqref{eq:LcubedAFE} holds with a different definition of $V_1$ and $V_2$, namely, in loc.\ cit.\ (4.3), instances of $\delta+it$ are replaced by $\frac{\kappa-1}{2}$ and instances of $\delta-it$ are replaced by $\frac{\kappa+1}{2}$.
\begin{proof}
The proof in \cite[Lem.\ 4.2]{PetrowYoungWeyl} goes through independently of value of the root number $\epsilon(f)$. Indeed, the standard approximate functional equation (AFE) for $L(1/2,f)$ says that $L(1/2,f) = (1+\epsilon(f))A(1/2,f)$, where $$A(1/2,f) =  \sum_n \frac{\lambda_f(n)}{\sqrt{n}} V_1\Big(\frac{n}{\sqrt{Q}},t_j\Big).$$ To give a formula for $L(1/2,f)^3$, we multiply $(1+\epsilon(f))A(1/2,f)$ by an AFE for $L(1/2,f)^2$, which does not depend on the value of $\epsilon(f)$ (the root number of $L(1/2,f)^2$ being $+1$). If $\epsilon(f)=+1$, the result is the claimed formula. If $\epsilon(f)=-1$, then $L(1/2,f)^2 = 0$ already, so we may multiply this by $2A(1/2,f)$ (which may not equal $L(1/2,f)$) to obtain a correct formula for $L(1/2,f)^3$. 
\end{proof}

Now we proceed to follow Section 4.2 of \cite{PetrowYoungWeyl}. We insert $\frac{1}{2}(1\pm \lambda_f(-1))$ to detect the parity condition in $\mathcal{M}^\pm((\sigma_p)_{p\in S})$ and apply Lemma \ref{AFE} and Theorem \ref{PBKformula} to obtain that $\mathcal{M}^\pm((\sigma_p)_{p\in S}) = \mathcal{D} + \mathcal{S}^{+} \pm \mathcal{S}^{-}$, where $\mathcal{D}$ is the diagonal term, and $\mathcal{S}^{\pm}$ is given by 
\begin{equation}\label{Spm_def}\mathcal{S}^{\pm}= \frac{1}{2}\delta_{\rm fin} \sum_{(d,Q)=1} \frac{4}{d} \sum_{n_1,n_2,n_3} \frac{1}{\sqrt{n_1n_2n_3}} \sum_{c \equiv 0 \shortmod{q'}} \frac{H(\pm n_1,n_2n_3 ;c)}{c} H_\infty^{\pm}\left( \frac{4 \pi \sqrt{n_1n_2n_3}}{c}\right),\end{equation}
with $q'$ the geometric conductor and $H^\pm_\infty$ here defined with respect to loc.\ cit.\ (4.7) divided by 4, cf.\ loc.\ cit.\ (4.9). By the trivial bound
 for supercuspidal Kloosterman sums \cite[Thm.\ 3.8(5)]{HPY} along with the classical Weil bound and estimates for $H_\infty^\pm$ found in \cite[\S 10, 13]{PetrowYoungWeyl}, the sum in \eqref{Spm_def} converges absolutely.

Inserting a partition of unity to \eqref{Spm_def} that localizes the variables to $n_j \asymp N_j$, $j=1,2,3$ and $c \asymp C$, trivial bounds show that the sum is very small unless $N_1 \ll (T\sqrt{Q})^{1+\varepsilon}$, $N_2 N_3 \ll d^{-2} (T\sqrt{Q})^{2+\varepsilon}$, and $C \ll (TQ)^{100}$, with $C \gg q'$, cf.\ \cite[Prop.\ 4.3]{PetrowYoungWeyl}. Both $\mathcal{S}^{+}$ and $\mathcal{S}^{-}$ are similar, so we only estimate $\mathcal{S}^{+}$ for simplicity of notation.

Next we follow Section 4.3 of \cite{PetrowYoungWeyl} and apply Poisson summation to each of $n_1, n_2, n_3$ modulo $c$ (using that the generalized Kloosterman sum $H(m,n;c)$ is periodic in $m,n$ modulo $c$; recall \eqref{eq:KloostermanPeriodicModuloc}). The results of Section 4.3 of loc.\ cit.\ go through with an extra factor of $\delta_{\rm fin}$, an identical treatment of the (archimedean) integral transforms, and where the character sum $G = G(m_1,m_2,m_3;c)$ appearing there should now be given by 
\begin{equation}\label{eq:Gdef}
G = c^{-3}
\sum_{x_1, x_2, x_3 \shortmod{c}} H(x_1, x_2 x_3 ;c) e_c(m_1 x_1 + m_2 x_2 + m_3 x_3).
\end{equation}
Thus, the only departure in this paper from the previous work of the last two authors is the different analysis of the character sum $G$. 

Following \cite[Section 5]{PetrowYoungWeyl} now, we 
define the Dirichlet series
\begin{equation}
\label{eq:Zdef}
Z(s_1, s_2, s_3, s_4) = \sum_{\epsilon_1 m_1, \epsilon_2 m_2, \epsilon_3 m_3 \geq 1}
\sum_{c \equiv 0 \shortmod{q'}}
\frac{c q' G(m_1, m_2, m_3;c) e_c(-m_1 m_2 m_3)}{m_1^{s_1} m_2^{s_2} m_3^{s_3} (c/q')^{s_4}}
\end{equation}
cf.\ loc.\ cit.\ (5.1). 
We claim the following bound on $Z$.
\begin{mylemma}
\label{lemma:Zlemma}
There exists a decomposition $Z = Z_0 + Z_1$ where 
$Z_0$ and $Z_1$ satisfy the following properties.  Firstly, $Z_0$ is meromorphic in $\mathrm{Re}(s_j) \geq \sigma > 1/2$ for all $j$ and analytic for $\mathrm{Re}(s_j) \geq \sigma > 1$ for all $j$.  It has a pole whenever some $s_j = 1$ and the other variables are held fixed.  In the region $\mathrm{Re}(s_j) \geq \sigma > 1$ it satisfies the bound
\begin{equation}
\label{eq:Z0bound}
Z_0(s_1, s_2, s_3, s_4) \ll_{\sigma, \varepsilon} \frac{(q')^2}{Q^{3/4}} Q^{\varepsilon}.
\end{equation}
Secondly, $Z_1$ is analytic for $\mathrm{Re}(s_j) \geq \sigma > 1/2$ for all $j$, wherein it satisfies the bound for $X \gg 1$
\begin{equation}
\label{eq:Z1bound}
\int_{-X}^{X} |Z_1(\sigma+it, \sigma+it, \sigma+it, \sigma-it)| dt
\ll_{\sigma, \varepsilon} (q')^{3/2+\varepsilon} X^{1+\varepsilon}. 
\end{equation}
The same bound \eqref{eq:Z1bound} also holds for $Z_0$, provided (say) $1/2 < \sigma \leq  \mathrm{Re}(s_j) \leq 0.99$.
\end{mylemma}

It is also necessary to bound the contribution from the terms with $m_1 m_2 m_3 = 0$, which are not included in $Z$.  For these terms, all that we need is the following estimate.  
\begin{mylemma}
\label{lemma:GboundZero}
If $m_1 m_2 m_3 = 0$ then
\begin{equation}
\label{eq:GboundZero}
|G(m_1, m_2, m_3;c)| \ll \frac{Q^{\varepsilon}}{c q'} \frac{(m_1, q') (m_2, q')(m_3,q')}{q'}.
\end{equation}
\end{mylemma}
We also need the following very minor fact about $G$, which is contained in Lemma \ref{lemma:Zformulaprelim}
\begin{mylemma}\label{lemma:veryeasylemma}
Suppose that $c = c_Q c'$ where $c_Q | Q^{\infty}$ and $(c', Q) = 1$. Then $G(m_1,m_2,m_3;c)$ vanishes unless $(m_1,c')=1$.
\end{mylemma}
Lemma \ref{lemma:veryeasylemma} is the appropriate substitute for the condition $(m_1,c/q')=1$ from \cite[(5.2)]{PetrowYoungWeyl}.
Inserting Lemma \ref{lemma:Zlemma} into the machinery developed in \cite{PetrowYoungWeyl} together with Lemmas \ref{lemma:veryeasylemma} and \ref{lemma:GboundZero} implies Theorem \ref{MT_precise}. 
 One can check this by following through the arguments of \cite[Section 12]{PetrowYoungWeyl}.  
There are a few key points to observe to aid in this (long) exercise:
\begin{enumerate}
 \item The integral transforms here are identical to \cite{PetrowYoungWeyl}, so the properties of the weight functions stated in \cite[\S 10, 11, 13]{PetrowYoungWeyl} hold verbatim here.  Thus, the Weyl bound holds in the archimedean aspect.
 \item The calculation of the contribution from the $Z$-function near the $1/2$-lines are essentially identical, using \eqref{eq:Z1bound} in place of \cite[(8.1)]{PetrowYoungWeyl}.
 \item On the other hand, in \cite[Lem.\ 8.1]{PetrowYoungWeyl} a stronger bound
$Z_0 \ll Q^{\varepsilon}$ was claimed in place of \eqref{eq:Z0bound}. These bounds are used for the estimation with $Z_0$ near the $1$-lines, and in the corresponding estimates in \cite[\S 12]{PetrowYoungWeyl} one finds that there is essentially $\sqrt{q'}$ to spare. Going over these calculations carefully, one finds that this extra slack is exactly what is needed in the present paper. 
 \item In checking the calculations for the estimates with $Z_0$ near the $1$-lines  as in \cite[\S 12]{PetrowYoungWeyl}, it is helpful to note by the explicit formula for $q'$ found in Theorem \ref{PBKformula} that
 \begin{equation}
 \label{eq:MvF}
  Q = \prod_p p^{c(\sigma_p)} \leq \prod_p p^{2 \lceil c(\sigma_p)/2 \rceil} =  (q')^2.
 \end{equation}
\item The contribution of the terms with $m_1 m_2 m_3 = 0$ in \cite{PetrowYoungWeyl} was estimated at the end of Section 12.  Going over these arguments, one sees that Lemma \ref{lemma:GboundZero} replaces \cite[(6.9)]{PetrowYoungWeyl} and Lemma \ref{lemma:veryeasylemma} replaces the condition that $(m_1,r)=1$. 
 \end{enumerate}

This ends our discussion on the extraction of work from \cite{PetrowYoungWeyl}.  The rest of the paper is devoted to proving Lemma \ref{lemma:Zlemma}. 
We will treat \eqref{eq:Zdef} in the case that $\epsilon_1 = \epsilon_2 = \epsilon_3 = +1$.  The other sign combinations are similar, so the benefit of this assumption is to simplify the notation, omit absolute value signs, and employ other similar conveniences.

\section{Simplifying $G$}
Our first steps in the analysis of $Z$ mirror \cite[Section 5]{PetrowYoungWeyl}, and use only the most elementary properties of the generalized Kloosterman sums.  

\subsection{Fourier decomposition}
Suppose that $c = c_Q c'$ where $c_Q | Q^{\infty}$ and $(c', Q) = 1$.  
Recall the character sum $G$ defined in \eqref{eq:Gdef} and let
$$G'(m_1, m_2, m_3;c) = c G(m_1, m_2, m_3;c) e_c(-m_1 m_2 m_3).$$
\begin{mylemma}
\label{lemma:Zformulaprelim}
We have that 
\begin{multline}
\label{eq:Bformulasimplified}
G'(m_1, m_2, m_3;c) 
\\
=
\frac{1_{(m_1,c')=1}}{c_Q^2}
\sum_{x_1, x_2, x_3 \shortmod{c_Q}} H(c' x_1 x_2 x_3, 1 ;c_Q) 
e_{c_Q}(m_1 x_1 + m_2 x_2 + m_3 x_3 - \overline{c'}m_1 m_2 m_3).
\end{multline}
\end{mylemma}
\begin{proof}
By the Chinese remainder theorem and \eqref{eq:HCRT}, we have
$G'(m_1, m_2, m_3;c) = A B$, where
\begin{equation*}
 A =\frac{1}{(c')^2}
\sum_{x_1, x_2, x_3 \shortmod{c'}} H(\overline{c_Q} x_1, \overline{c_Q} x_2 x_3 ;c') 
e_{c'}(\overline{c_Q}(m_1 x_1 + m_2 x_2 + m_3 x_3 - m_1 m_2 m_3)),
\end{equation*}
and
\begin{equation*}
 B=
\frac{1}{c_Q^2}
\sum_{x_1, x_2, x_3 \shortmod{c_Q}} H(\overline{c'} x_1, \overline{c'} x_2 x_3 ;c_Q) 
e_{c_Q}(\overline{c'}(m_1 x_1 + m_2 x_2 + m_3 x_3 - m_1 m_2 m_3)).
\end{equation*}
Since $H(m,n;c') = S(m,n;c')$ is the ordinary Kloosterman sum, we calculate by 
opening the definition and executing the $x_1$ and $x_2$ sums by
orthogonality of characters that
$A =   \delta((m_1, c') = 1)$. 

Next we turn to $B$.  We first observe that $H(m,n;c_Q) = H(mn,1;c_Q)$ holds for all $n$ (not just $(n, c_Q) = 1$), since it holds prime-by-prime by inspection of \eqref{eq:KloostermanDefPScase} and \eqref{eq:KloostermanDefSCinertCaseGeneral}.
Using this, and changing variables $x_i \rightarrow c' x_i$, we obtain the claimed formula.
\end{proof}

We may also write $m_i = m_i' m_{i,0}$ where $(m_i', Q) = 1$ and $m_{i,0} | Q^{\infty}$.

\begin{mylemma}
\label{lemma:Zformula}
For $\psi$ a Dirichlet character modulo $c_Q$,
define 
$\widehat{H}(\psi) = \widehat{H}(\psi, a_1, a_2, a_3)$ 
by
\begin{equation}
\label{eq:HhatpsiDef}
 \widehat{H}(\psi) = c_Q^{-2}
 \sum_{u, x_1, x_2, x_3 \shortmod{c_Q}} 
 \overline{\psi}(u)
 H(\overline{u} x_1 x_2 x_3, 1 ;c_Q) 
e_{c_Q}(a_1 x_1 + a_{2} x_2 + a_{3} x_3 - u a_{1}a_{2}a_{3} ).
\end{equation}
Then for $\mathrm{Re}(s_i) > 1$, we have that $Z= Z(s_1, s_2, s_3, s_4)$ equals
\begin{equation}\label{eq:Zformula}
\sum_{\substack{m_{i} | Q^{\infty} \\ (i=1,2,3)}}
 \sum_{\substack{c_Q | Q^{\infty} \\ c_Q \equiv 0 \shortmod {q'}}}
 \frac{ q'}{  \varphi(c_Q)}
 \sum_{\psi \shortmod{c_Q}} 
 \frac{\widehat{H}(\psi, m_{1}, m_{2}, m_{3})}{m_{1}^{s_1} m_{2}^{s_2} m_{3}^{s_3} (c_Q/q')^{s_4}}
\frac{ L(s_1, \psi) L(s_2, \psi) L(s_3, \psi) L(s_4, \overline{\psi})}{L(s_1 + s_4, \chi_0)},
 \end{equation}
 where $\chi_0$ is the trivial character modulo $q'$.
\end{mylemma}
\begin{proof}
Changing variables in \eqref{eq:Bformulasimplified} 
by $x_i \rightarrow \overline{m_i'} x_i$ gives
\begin{multline*}
 G'= 1_{(m_1,c')=1} B = 
\frac{1_{(m_1,c')=1}}{c_Q^2}
\sum_{x_1, x_2, x_3 \shortmod{c_Q}} H(c' \overline{(m_1' m_2' m_3')} x_1 x_2 x_3, 1 ;c_Q) 
\\
 e_{c_Q}((m_{1,0} x_1 + m_{2,0} x_2 + m_{3,0} x_3 - \overline{c'}m_1' m_2' m_3' m_{1,0}m_{2,0}m_{3,0} )).
\end{multline*}
Now we view $B$ as a function of $m_1' m_2' m_3' \overline{c'}$ on $(\mz/c_Q\mz)^{\times}$, and apply the (multiplicative) finite 
Fourier expansion, giving
\begin{equation}
\label{eq:BMellinFourierExpansion}
 G' = \frac{1_{(m_1,c')=1}}{\varphi(c_Q)}
 \sum_{\psi \shortmod{c_Q}} \widehat{H}(\psi) \psi(m_1' m_2' m_3' \overline{c'}),
\end{equation}
where recall $\widehat{H}(\psi) = \widehat{H}(\psi, m_{1,0}, m_{2,0}, m_{3,0})$ was defined by \eqref{eq:HhatpsiDef}.

Inserting the definition of $G'$ into the definition of $Z$ in \eqref{eq:Zdef} and using \eqref{eq:BMellinFourierExpansion}  gives
\begin{multline}
Z = \sum_{\substack{m_1, m_2, m_3 \geq 1 \\ c \equiv 0 \shortmod{q'}}}
 \frac{q' G'(m_1, m_2, m_3, c)}{m_1^{s_1} m_2^{s_2} m_3^{s_3} (c/q')^{s_4}} \\ 
= \sum_{m_1, m_2, m_3 \geq 1}
\sum_{\substack{q' | c_Q, \thinspace c_Q| Q^{\infty} \\ (c',  Q) = 1 }} \frac{\delta((m_1, c') = 1)}{m_1^{s_1} m_2^{s_2} m_3^{s_3} (c/q')^{s_4}}
\frac{q'}{\varphi(c_Q)}
 \sum_{\psi \shortmod{c_Q}} \widehat{H}(\psi) \psi(m_1' m_2' m_3' \overline{c'})
.
\end{multline}
Summing over the $m_i'$ and $c'$ forms the product of four Dirichlet $L$-functions, and the condition $(m_1, c') = 1$ leads to the factor $L(s_1 + s_4,\chi_0)^{-1}$ where $\chi_0$ is mod $c_Q$.
\end{proof}

Based on Lemma \ref{lemma:Zformula}, write
\begin{equation}
\label{eq:Zformula2}
Z = 
\sum_{\substack{c_Q | Q^{\infty} \\ q' | c_Q}}
\frac{q'}{\varphi(c_Q)} \sum_{\psi \shortmod{c_Q}} 
\frac{Z_{\mathrm{fin}, Q}}{(c_Q/q')^{s_4}} \cdot
\frac{ L(s_1, \psi) L(s_2, \psi) L(s_3, \psi) L(s_4, \overline{\psi})}{L(s_1 + s_4, \chi_0)},
\end{equation}
where $Z_{\mathrm{fin}, Q} = Z_{\mathrm{fin}, Q}(s_1, s_2, s_3, \psi, c_Q)$ is defined by
\begin{equation}
\label{eq:ZfinFormula2}
Z_{\mathrm{fin}, Q} = 
\sum_{\substack{m_{i} | Q^{\infty} \\ (i=1,2,3)}}
 \frac{\widehat{H}(\psi, m_{1}, m_{2}, m_{3})}{m_{1}^{s_1} m_{2}^{s_2} m_{3}^{s_3} }.
\end{equation}

\subsection{Multiplicativity of $\widehat{H}$}
Next we work out a form of multiplicativity of $\widehat{H}$. 
\begin{mylemma}\label{lemma:multiplicativity}
Suppose that $Q = Q_1 Q_2$ with $(Q_1, Q_2) = 1$, $c_Q = c_{Q_1} c_{Q_2}$,
$\psi = \psi_1 \psi_2$ with $\psi_i$ modulo $c_{Q_i}$, and $m_{j} = m_{j,1} m_{j,2}$, where $m_{j,i} | c_{Q_i}^{\infty}$, for $i=1,2$ and each $j=1,2,3$.  Then
\begin{equation*}
 \widehat{H}(\psi, m_{1}, m_{2}, m_{3})
 = \epsilon 
 \widehat{H}(\psi_1, m_{1,1}, m_{2,1}, m_{3,1})
 \widehat{H}(\psi_2, m_{1,2}, m_{2,2}, m_{3,2}),
\end{equation*}
where  $\epsilon$  is the function given by \eqref{eq:epsilonformula} below, which satisfies $|\epsilon| = 1$.  
In particular, $|\widehat{H}|$ is multiplicative.
\end{mylemma}
Remark.  This multiplicativity allows us to evaluate $\widehat{H}$ one prime at a time.

\begin{proof}
  The Chinese remainder theorem gives
\begin{multline*}
 \widehat{H}(\psi) = 
\Big(c_{Q_1}^{-2} \sum_{u, x_1, x_2, x_3 \shortmod{c_{Q_1}}} 
 \overline{\psi_1}(u)
 H(\overline{c_{Q_2}^2}\overline{u} x_1 x_2 x_3, 1 ;c_{Q_1}) 
 \\
e_{c_{Q_1}}(\overline{c_{Q_2}} (m_{1,1} m_{1,2} x_1 + m_{2,1} m_{2,2} x_2 + m_{3,1} m_{3,2} x_3 - u  m_{1,1} m_{1,2} m_{2,1} m_{2,2} m_{3,1} m_{3,2} ))
\Big)
\big(\dots \big),
\end{multline*}
with $(\dots)$ representing the symmetrically-defined term from the complementary modulus.  To simplify this, change variables
$u \rightarrow c_{Q_2} \overline{m_{1,2} m_{2,2} m_{3,2}} u$ and
$x_i \rightarrow c_{Q_2} \overline{m_{i,2}} x_i$ (for $i=1,2,3$), giving
\begin{multline*}
 \widehat{H}(\psi) = 
\Big(c_{Q_1}^{-2} \sum_{u, x_1, x_2, x_3 \shortmod{c_{Q_1}}} 
 \overline{\psi_1}(c_{Q_2} \overline{m_{1,2} m_{2,2} m_{3,2}} u)
 H( \overline{u} x_1 x_2 x_3, 1 ;c_{Q_1}) 
 \\
e_{c_{Q_1}}( m_{1,1}  x_1 + m_{2,1}  x_2 + m_{3,1} x_3 - 
u   m_{1,1}  m_{2,1} m_{3,1}  )
\Big)
\big(\dots \big).
\end{multline*}
This is the desired result, with the explicit value
\begin{equation}
\label{eq:epsilonformula}
 \epsilon = 
 \psi_1(\overline{c_{Q_2}} m_{1,2} m_{2,2} m_{3,2})
\psi_2(\overline{c_{Q_1}} m_{1,1} m_{2,1} m_{3,1}). \qedhere
\end{equation}
\end{proof}
In light of Lemma \ref{lemma:multiplicativity}, define
\begin{equation}
\label{eq:ZfinL1normDef}
\|Z_{\mathrm{fin}, Q}\| =  
\sum_{\substack{m_{i} | Q^{\infty} \\ (i=1,2,3)}}
 \frac{|\widehat{H}(\psi, m_{1}, m_{2}, m_{3})|}{m_{1}^{\sigma_1} m_{2}^{\sigma_2} m_{3}^{\sigma_3}},
\end{equation}
where $\sigma_i = \mathrm{Re}(s_i) > 1/2$.  Then Lemma \ref{lemma:multiplicativity} implies
\begin{equation}\label{Zfinfactorization}
\| Z_{\mathrm{fin}, Q} \| = \prod_{p \mid Q } \| Z_{\mathrm{fin}, p} \|.
\end{equation}

\subsection{Zero frequency}

\begin{proof}[Proof of Lemma \ref{lemma:GboundZero}]
We carry out the argument when $m_3=0$, the other cases being similar by symmetry.
 Then $G'(m_1, m_2, m_3;c) = c G(m_1, m_2, m_3;c)$, so the task is to show $|G'(m_1, m_2, m_3;c)| \ll \frac{Q^{\varepsilon}}{(q')^2} \prod_{i=1}^{3} (m_i, q')$.
By Lemma \ref{lemma:Zformulaprelim} we have that $G'= 1_{(m_1,c')=1}B,$ where 
$$B = \frac{1}{c_Q^2}
\sum_{x_1, x_2, x_3 \shortmod{c_Q}} H(c' x_1 x_2 x_3, 1 ;c_Q) 
e_{c_Q}(m_1 x_1 + m_2 x_2 + m_3 x_3 - \overline{c'}m_1 m_2 m_3).$$
Note that $B$ is multliplicative in $c_Q$; to see this
 simply change variables $x_3 \rightarrow \overline{c'} x_3$. 
It thus suffices to work one prime at a time. If $p \mid Q$ is such that the corresponding specified local representation $\sigma_p$ is not supercuspidal, then the required bound was already shown in \cite[(6.9)]{PetrowYoungWeyl}.

Assume now that  $p \mid Q$ is such that $\sigma_p$ is supercuspidal.
With $c_Q = p^k$, note
\begin{equation*}
B = p^{-2k}
\sum_{x_1, x_2, x_3 \shortmod{p^k}} H(x_1 x_2 x_3, 1;p^k) e_{p^k}(m_1 x_1 + m_2 x_2 ).
\end{equation*}
 Since $H(x_1 x_2 x_3, 1;p^k) = 0$ unless $(x_1 x_2 x_3, p) = 1$ (see around \eqref{eq:KloostermanDefSCinertCaseGeneral}), 
 we may restrict to $(x_1 x_2 x_3, p) = 1$ and then change variables $x_3 \rightarrow \overline{x_1 x_2} x_3$, giving
\begin{align*}
B &= p^{-2k}
\sumstar_{x_1, x_2, x_3 \shortmod{p^k}} 
H( x_3, 1;p^k) e_{p^k}(m_1 x_1 + m_2 x_2)
\\
&=  p^{-2k} S(m_1, 0;p^k) S(m_2, 0;p^k) 
\sumstar_{x_3 \shortmod{p^k}}
H( x_3, 1;p^k).
\end{align*}
It is well-known that the Ramanujan sum satisfies the bound $|S(n, 0;p^k)| \leq (n, p^k)$.  We also have by \cite[Prop.\ 7.8]{HPY} that the $x_3$-sum vanishes unless $c(\xi) = ek-d$, i.e., $c_0 = k - \frac{d}{e}$.
In particular, the sum vanishes unless $d=0$ or $2$ (see e.g.\ \eqref{eq:supercuspidalTable} \eqref{eq:supercuspidalTablepEquals2}). 
In addition, when $p=2$ each individual $H(x_3,1;p^k)$ vanishes if $k\leq c_0+100$, finishing the proof. Only the $p\neq 2$ unramified case remains to be treated. Here we have $q'=p^{c_0}$ (see \eqref{eq:supercuspidalTable}) and \cite[Prop.\ 7.8]{HPY} gives
\begin{equation*}
\big| \sumstar_{x_3 \shortmod{p^k}}
H( x_3, 1;p^k)\big| = \delta_p,
\end{equation*}
where we recall from Theorem \ref{PBKformula} that $\delta_p = p^{c_0}$.  From the above discussion, $k=c_0$, so
\begin{equation*}
|B| \leq p^{-k} (m_1, p^{k}) (m_2, p^{k}) \delta(c_0 = k) \leq (q')^{-2} \prod_{i=1}^{3} (m_i, q'). \qedhere
\end{equation*}
\end{proof}

\subsection{Evaluation of $\widehat{H}$: principal series and special}\label{sec:ps}
In light of the twisted multiplicativity of $\widehat{H}(\psi)$ (Lemma \ref{lemma:multiplicativity}), we proceed to investigate it one prime at a time. 
Thus, let us fix $p\in S$ now, and suppose that the specified $\sigma_p$ is either isomorphic to a principal series representation $\pi(\chi, \chi^{-1})$ or isomorphic to a trivial central character special representation $\St \times \chi$ with $\chi$ of conductor $p^{c_0}$, $c_0\geq 1$. In these cases, we have $c(\sigma_p)=2c_0$ and $v_p(q') = c_0$. Let us write $p^k$ with $k\geq v_p(q')$ for the $p$-part of $c_Q$, and $r= p^{k-c_0}$ for the $p$-part of $c_Q/q'$. 
\begin{mylemma}
\label{lemma:HhatEvalPScase}
Let assumptions be as directly above. 
\begin{enumerate}
\item If $c(\psi) > c_0$ then $\widehat{H}(\psi) = 0$.  
\item If $c(\psi) \leq c_0$, then
\begin{equation}
\label{eq:HhatHerevsHhatPY}
\widehat{H}(\psi, a_1, a_2, a_3)
= \frac{r}{p^{c_0}} \chi(-1) \widehat{H}_{\mathrm{PY}}(\psi, \chi, a_1, a_2, a_3, r),
\end{equation}
where $\widehat{H}_{\mathrm{PY}}$ corresponds to the definition in \cite[(5.12)]{PetrowYoungWeyl}.
\end{enumerate}
\end{mylemma}

\begin{proof}
Abusing notation, we now consider $\chi$ as a Dirichlet character modulo $p^k$ in lieu of a character of $\Q_p^\times$. We pick up the calculation at \eqref{eq:HhatpsiDef}, write $H(\overline{u} x_1 x_2 x_3, 1;p^k) = H(\overline{u} x_1, x_2 x_3;p^k)$, and insert the formula \eqref{eq:KloostermanDefPScase} for the Kloosterman sum to get 
\begin{multline*}
\widehat{H}(\psi, a_1, a_2, a_3) = 
p^{-2k}
\sum_{t,u, x_1, x_2, x_3 \shortmod{p^k}} 
 \overline{\psi}(u)  \chi^2(t)
\chi(\overline{u} x_1) \overline{\chi}(x_2 x_3)
\\
\times e_{p^k}(a_1 x_1 + a_{2} x_2 + a_{3} x_3 - u a_{1}a_{2}a_{3} )
  e_{p^k}(\overline{u} x_1  t + \overline{t} x_2 x_3).
\end{multline*}
Taking the sum over $x_1$ to the inside is beneficial because we can evaluate it as follows:
\begin{equation*}
\sum_{x_1 \shortmod{p^k}} \chi(x_1) e_{p^k}(x_1(a_1 + \overline{u} t))
= p^{k-c_0} \tau(\chi) \delta(p^{k-c_0} | t \overline{u} + a_1) \overline{\chi}\Big(\frac{t \overline{u}+a_1}{p^{k-c_0}}\Big),
\end{equation*}
where $\tau(\chi)$ is the Gauss sum of $\chi$ defined with respect to its conductor $p^{c_0}$.  To see this, write $x_1 = X + Y p^{c_0}$ where $X$ runs modulo $p^{c_0}$ and $Y$ runs modulo $p^{k-c_0}$.  The $Y$-sum is evaluated with orthogonality of additive characters, and then the $X$-sum is a Gauss sum.  Similarly, we evaluate $x_2$ with
\begin{equation*}
\sum_{x_2 \shortmod{p^k}} \overline{\chi}(x_2) e_{p^k}(x_2(a_2 + \overline{t} x_3))
= p^{k-c_0} \tau(\overline{\chi}) \delta(p^{k-c_0} | x_3 \overline{t} + a_2) \chi\Big(\frac{x_3 \overline{t}+a_2}{p^{k-c_0}}\Big).
\end{equation*}
Write $t = -a_1 u + rv$ and $x_3 = a_1 a_2 u + ry$, where $v,y$ run modulo $p^{c_0}$.
Note that
\begin{equation*}
\chi\Big(\frac{x_3 \overline{t}+a_2}{p^{k-c_0}}\Big)
= 
\overline{\chi}(t)
\chi\Big(\frac{x_3 +a_2 t}{p^{k-c_0}}\Big)
= \overline{\chi}(t)
\chi( y + a_2 v).
\end{equation*}
Similarly, $\overline{\chi}((t\overline{u} + a_1)/r) = \chi(u) \overline{\chi}( v )$.
Substituting and simplifying, we obtain that $\widehat{H}(\psi, a_1, a_2, a_3)$ equals
\begin{equation}
\label{eq:HhatPSsomewhatSimplified} 
\frac{\chi(-1)}{p^{c_0}}
\sum_{\substack{u \shortmod{p^k} \\ v,y \shortmod{p^{c_0}}}} 
 \overline{\psi}(u)  \chi(-a_1 u + rv)
\chi( y + a_2 v)
 \overline{\chi}(a_1 a_2 u + ry)
 \overline{\chi}( v )
e_{p^{c_0}}(a_3 y).
\end{equation}
As a function of $u$, the summands besides $\overline{\psi}(u)$ have period $p^{c_0}$, which shows that the sum vanishes if $c(\psi) > c_0$, giving the first item to be proved.

Now suppose that $c(\psi) \leq c_0$.  The summand in \eqref{eq:HhatPSsomewhatSimplified} is periodic in $u \pmod{p^{c_0}}$, so $\widehat{H}(\psi)$ equals
\begin{equation*}
\frac{\chi(-1)r}{p^{c_0}}
\sum_{\substack{u, v,y \shortmod{p^{c_0}}}} 
 \overline{\psi}(u)  \chi(-a_1 u + rv)
\chi( y + a_2 v)
 \overline{\chi}(a_1 a_2 u + ry)
 \overline{\chi}( v )
e_{p^{c_0}}(a_3 y).
\end{equation*}
Changing variables $v \rightarrow uv$ and $y \rightarrow uy$ and simplifying gives
\begin{equation*}
\frac{\chi(-1)r}{p^{c_0}}
\sum_{\substack{u, v,y \shortmod{p^{c_0}}}} 
 \overline{\psi}(u)  \chi(-a_1  + rv) 
\chi( y + a_2 v) 
 \overline{\chi}(a_1 a_2 + ry)
 \overline{\chi}(v )
e_{p^{c_0}}(a_3 yu).
\end{equation*}
The sum over $u,v,y$ above  is precisely \cite[(5.13)]{PetrowYoungWeyl}. 
\end{proof}

A difference between our approach here and that in \cite{PetrowYoungWeyl} is that here we expanded into Dirichlet characters modulo $c_Q$ while in \cite{PetrowYoungWeyl} the expansion was modulo $q'$.  
However,
Lemma \ref{lemma:HhatEvalPScase} shows that in the principal series case $\widehat{H}(\psi)$ vanishes if $v_p(c (\psi)) > v_p(q')$, so the expansion reduces into characters of modulus $q'$ after all.

\section{Bounds on $\widehat{H}$ in supercuspidal cases}
\label{section:HhatCalculations}
For the rest of this section, we
 fix $p\in S$ and suppose that the specified $\sigma_p$ is supercuspidal, and moreover that $c(\sigma_p) \geq 11$ if $p=2$. To such a $\sigma_p$ we associate by \cite[Thms.\ 6.6, 6.7]{HPY} a unique pair $(L/\Q_p,\xi)$ as in Section \ref{neighborhoods_of_reps}. In this section of the paper, we write $v_L$ for the valuation on $L$ and $v$ in lieu of $v_p$ for the $p$-adic valuation, as $p$ is fixed. Recall the Kloosterman sum $H(m,n;p^k)$ corresponding to $\sigma_p$ is given by  \eqref{eq:KloostermanDefSCinertCaseGeneral}.  For such $k$, which we may assume is $\geq v(q')$, define $\ell_\xi$ according to  \eqref{eq:Rdef}.  Also, recall that $\xi$ has least period $p^{c_0}$.

We also recall that the Postnikov formula (Lemma \ref{lem:postnikov}) for $\xi(1+u)$  holds provided either $p$ is large, or $|u|_L$ is small. 
Throughout this section, we use the condition
``$p^k$ is large" as a concise way to say that the Postnikov formula (precisely, in the form of \eqref{eq:PostnikovThetaVersion}) may be applied and that some sporadic higher-order terms in the $p$-adic logarithm may be discarded.  
 The definition ``We say `$p^k$ is large' if $k$ or $p$ is $\geq 10$.'' is more than sufficient.

We assume $p \neq 2$ in Sections 
\ref{section:HhatSimplifications}--\ref{section:AG},
and treat $p=2$ in
Section \ref{section:pequals2supercuspidal}.

\subsection{Warming up}
We begin with some computations that hold in generality. \begin{mylemma}
\label{lemma:HhatDecentFormulaSupercuspidalInert}
We have that  $\gamma p^{d/2} p^{2k} \widehat{H}(\psi, a_1, a_2, a_3)$ equals
\begin{equation}
\label{eq:HhatDecentFormulaSupercuspidalInert2}
\sum_{\substack{ x_1, x_2, x_3 \shortmod{p^k}}} 
\sum_{\substack{t \in \mathcal{O}_L/(p^k)}} 
 \psi\Big(\frac{\mathrm{Nm}(t)}{x_1 x_2 x_3}\Big)
\xi(t) e_{p^k}\Big(-\mathrm{Tr}(t)+
 a_1 x_1 + a_{2} x_2 + a_{3} x_3 - \frac{x_1 x_2 x_3  a_{1}a_{2}a_{3}}{\mathrm{Nm}(t)} \Big).
\end{equation}
\end{mylemma}
\begin{proof}
The definition, via \eqref{eq:HhatpsiDef} and \eqref{eq:KloostermanDefSCinertCaseGeneral}, implies that $\gamma p^{d/2} p^{2k} \widehat{H}(\psi, a_1, a_2, a_3)$ equals
\begin{equation}
\label{eq:HhatDecentFormulaSupercuspidalInert1}
\sumstar_{\substack{u, x_1, x_2, x_3 \shortmod{p^k}}} 
\sum_{\substack{t \in \mathcal{O}_L/(p^k) \\ \mathrm{Nm}(t) \equiv u^{-1} x_1 x_2 x_3 \shortmod{p^k}}} 
 \overline{\psi}(u)
\xi(t) e_{p^k}(-\mathrm{Tr}(t)+
 a_1 x_1 + a_{2} x_2 + a_{3} x_3 - u a_{1}a_{2}a_{3} ).
\end{equation}
Solving the inner congruence for $u$ gives \eqref{eq:HhatDecentFormulaSupercuspidalInert2}.
\end{proof} 

\begin{mylemma}
\label{lemma:HhatVanishesUnlessCoprimeSupercuspidalInert}
Suppose $p$ is odd and $k \geq c_0 +1$.  If $(a_1 a_2 a_3, p) \neq 1$, then $\widehat{H}(\psi, a_1, a_2, a_3) = 0$.
\end{mylemma} 
\begin{proof}
Suppose $p|a_1$, say.
Changing variables $x_1 \rightarrow \mathrm{Nm}(t) x_1$ in \eqref{eq:HhatDecentFormulaSupercuspidalInert2} gives
\begin{equation}
\label{eq:HhatDecentFormulaSupercuspidalInert3}
\sum_{\substack{ x_1, x_2, x_3 \shortmod{p^k}}} 
\sum_{\substack{t \in (\mathcal{O}_L/(p^k))^{\times}}} 
 \overline{\psi}(x_1 x_2 x_3)
\xi(t) e_{p^k}(-\mathrm{Tr}(t)+
 a_1 \mathrm{Nm}(t) x_1 + a_{2} x_2 + a_{3} x_3 - x_1 x_2 x_3  a_{1}a_{2}a_{3} ).
\end{equation}  
We change variables $t \rightarrow t + p^{k-1}$.  Since $k -1 \geq c_0$, this implies $\xi(t+p^{k-1}) = \xi(t)$, and also $e_{p^k}(a_1 \mathrm{Nm}(t+p^{k-1}) x_1) = e_{p^k}(a_1 \mathrm{Nm}(t) x_1)$.  Since $\mathrm{Tr}(t+p^{k-1}) = \mathrm{Tr}(t) + 2 p^{k-1}$, this shows the sum vanishes.  
\end{proof}

\subsection{Simplifications}\label{sec:supercuspidal_simplifications}
Recall that throughout this section, $p$ is an odd prime.
\label{section:HhatSimplifications}
\begin{mylemma}[First simplification, unramified case]
\label{lemma:HhatInitialEvaluation}
Suppose $L$ is unramified,
$p$ is odd,  $k \geq \max(c_0,2)$, and $p^k$ is large.  Let $n= \lfloor k/2 \rfloor$,
\begin{equation}
\label{eq:Adef}
A =A(x_1, x_2, x_3, t_0) := \frac{x_1 x_2 x_3 a_1 a_2 a_3}{\mathrm{Nm}(t_0)},
\end{equation}
and define the functions 
$$E = (E_1,E_2,E_3,E_4): ((\Z/p^n\Z)^\times)^3 \times (\cO_L/p^n\cO_L)^\times \to (\Z/p^n\Z)^3 \times \cO_L/p^n\cO_L$$ 
by
\begin{equation}
\label{eq:PostnikovCongruenceSystemSupercuspidalInert}
\begin{split}
E_i(x_1, x_2, x_3, t_0) &:=-\ell_{\psi} + a_i x_i - A, \qquad (i=1,2,3), \\
E_4(x_1, x_2, x_3, t_0) &:= A + \ell_{\psi} 
+\ell_\xi
- t_0.
\end{split}
\end{equation}
Let 
\begin{equation}
S_{\psi, \xi, \mathbf{a}} = \{ (x_1,x_2,x_3,t_0) \in ((\Z/p^n\Z)^\times)^3 \times (\cO_L/p^n\cO_L)^\times : E(x_1,x_2,x_3,t_0) = (0,0,0,0) \}.
\end{equation}
For $k$ even, with $k=2n$, we have  
\begin{equation}
\label{eq:HhatFormulaPostPosnikovEvenExponent}
\widehat{H}(\psi) = \frac{\overline{\gamma} p^{5n}}{p^{2k}}
\sumprime_{S_{\psi, \xi, \mathbf{a}}}
\psi\Big(\frac{\mathrm{Nm}(t_0)}{x_1 x_2 x_3}\Big)
\xi(t_0) e_{p^k}(-\mathrm{Tr}(t_0)+
 a_1 x_1 + a_{2} x_2 + a_{3} x_3 
 - 
 A
 ),
\end{equation}
where the $'$ on the summation here and below indicates that $(x_1, x_2, x_3, t_0)$ runs over any fixed choice of coset representatives for $[(\Z/p^k\Z)/ (p^n\Z/p^k\Z)]^3 \times (\cO_L/p^k\cO_L)/ (p^n\cO_L/p^k\cO_L)$ that project to  $S_{\psi, \xi, \mathbf{a}}$.
For $k$ odd, with $k=2n+1$, we have
\begin{equation}\label{eq:HhatInitialEvaluation2}
\widehat{H}(\psi) = \frac{ \overline{\gamma} p^{5n}}{p^{2k}}
\sumprime_{S_{\psi, \xi, \mathbf{a}}}
\psi\Big(\frac{\mathrm{Nm}(t_0)}{x_1 x_2 x_3}\Big)
\xi(t_0) e_{p^k}(-\mathrm{Tr}(t_0)+
 a_1 x_1 + a_{2} x_2 + a_{3} x_3 
 - 
 A
 )
\cdot G_p,
\end{equation}
 and where $G_p$ depending on $x_1, x_2, x_3, t_0$ is given by \eqref{eq:GpDef2} below. Each summand on the right-hand sides of \eqref{eq:HhatFormulaPostPosnikovEvenExponent} and \eqref{eq:HhatInitialEvaluation2} is independent of the choices of coset representatives.
\end{mylemma}
\begin{proof}
We continue the computation from \eqref{eq:HhatDecentFormulaSupercuspidalInert2}.  
For $i=1,2,3$, we replace $x_i$ by $x_i(1 + p^n y_i)$ and write $t=t_0(1+p^nt_1)$, where $(x_1, x_2, x_3, t_0)$ runs over a fixed choice $\mathcal{C}$ of coset representatives for $[(\Z/p^k\Z)/ (p^n\Z/p^k\Z)]^3 \times (\cO_L/p^k\cO_L)/ (p^n\cO_L/p^k\cO_L)$, and where 
 $y_i \in \Z/p^{k-n}\Z$ and $t_1 \in \cO_L/p^{k-n}\cO_L$.
 Note that $3^{-1} p^{3n} \equiv 0 \pmod{p^k}$, since 
when $p=3$ then $3n-1 \geq k$ for $k \geq 5$. 
 Thus
the assumption that $p^k$ is large implies
\begin{equation}
\log(1+ p^n y_i) \equiv p^{n} y_i - \tfrac12 p^{2n} y_i^2 \pmod{p^k}.
\end{equation}
Hence by \eqref{eq:PostnikovThetaVersion} with $L=\Q_p$, then 
\begin{equation}\label{eq:psiexp}
\overline{\psi}(1+ p^{n} y_i)
= e_{p^{k-n}}(-\ell_{\psi} y_i) e_{p^{k-2n}}(\tfrac12 \ell_{\psi} y_i^2).
\end{equation}

By \eqref{eq:PostnikovThetaVersion} and Remark \ref{rm:normchar}, we similarly have for $p^k$ large
\begin{equation}
\label{eq:psiPostnikovInert}
\psi(\mathrm{Nm}(t_0(1 + p^n t_1))
= \psi(\mathrm{Nm}(t_0)) e_{p^{k-n}}( \ell_{\psi} \mathrm{Tr}( t_1))
e_{p^{k-2n}}(-\tfrac12 \ell_{\psi} \mathrm{Tr}(t_1^2)),
\end{equation}
and
\begin{equation}
\label{eq:thetaPostnikovInert}
\xi(t_0(1+p^n t_1)) = \xi(t_0)  e_{p^{k-n}} (\mathrm{Tr}(\ell_\xi t_1))
e_{p^{k-2n}}(-\tfrac12 \mathrm{Tr}(\ell_\xi t_1^2))
.
\end{equation}
To understand the behavior of $e_{p^k}(-\frac{x_1 x_2 x_3 a_1 a_2 a_3}{\mathrm{Nm}(t)})$, note that
\begin{multline}
\label{eq:AformulaTaylorExpansion}
-\frac{x_1(1 + p^n y_1) x_2(1 + p^n y_2) x_3(1 + p^n y_3) a_1 a_2 a_3}{\mathrm{Nm}(t_0(1 + p^n t_1))}
\\
=
-A \cdot
[1 
+ p^n (y_1 + y_2 + y_3 - \mathrm{Tr}(t_1)) + c_2 p^{2n} ] \pmod{p^{3n}},
\end{multline}
where
\begin{equation}
\label{eq:c2def}
 c_2 = y_1 y_2 + y_1 y_3 + y_2 y_3
 - \Tr(t_1)(y_1 + y_2 + y_3)
 + (t_1^2 + t_1 \overline{t_1} + \overline{t_1}^2).
\end{equation}

Assembling all these formulas, we have
\begin{equation}
\label{eq:AformulaTaylorExpansion2}
\widehat{H}(\psi) = \frac{\overline{\gamma} }{p^{2k}}
\sumstar_{(x_1, x_2, x_3, t_0) \in \mathcal{C}} 
\psi\Big(\frac{\mathrm{Nm}(t_0)}{x_1 x_2 x_3}\Big)
\xi(t_0) e_{p^k}(-\mathrm{Tr}(t_0)+
 a_1 x_1 + a_{2} x_2 + a_{3} x_3 
 - 
 A)
 \cdot S_1,
\end{equation}
where
\begin{multline}
S_1 = 
\sum_{\substack{y_1, y_2, y_3 \shortmod{p^{k-n}} \\ t_1 \in \mathcal{O}_L/(p^{k-n})}}
e_{p^{k-n}}(-(\ell_{\psi}+A)(y_1 + y_2 + y_3) + a_1 x_1 y_1 + a_2 x_2 y_2 + a_3 x_3 y_3
)
\\
\times e_{p^{k-n}}(\mathrm{Tr}(t_1[\ell_{\psi} + \ell_\xi - t_0 + A]))
e_{p^{k-2n}}(r_2-Ac_2),
\end{multline}
and where
\begin{equation}
\label{eq:r2def}
 r_2 = \tfrac12 \ell_{\psi} (-\mathrm{Tr}(t_1^2) + y_1^2 + y_2^2 + y_3^2) 
 -\tfrac12 \mathrm{Tr}(\ell_\xi t_1^2).
\end{equation}

If $k=2n$, then by additive orthogonality of characters \eqref{eq:addorth}, $S_1$ evaluates as $p^{5n}$ times the indicator function of the system of equations \eqref{eq:PostnikovCongruenceSystemSupercuspidalInert}.
If $k=2n+1$, then $S_1$ evaluates as $p^{5n}$ times the indicator function of the system \eqref{eq:PostnikovCongruenceSystemSupercuspidalInert} times a quadratic exponential sum $G_p$ defined by
\begin{multline}
\label{eq:GpDef2}
G_p = 
\sum_{\substack{y_1, y_2, y_3 \shortmod{p^{ }} \\ t_1 \in \mathcal{O}_L/(p^{ })}}
e_{p^{}} \Big(\sum_{i=1}^{3} y_i \frac{a_i x_i - \ell_{\psi} - A}{p^n} + \frac{\mathrm{Tr}(t_1[\ell_{\psi} + \ell_\xi - t_0 + A])}{p^n} \Big)
e_{p^{}}(r_2-Ac_2).
\end{multline}
The equations \eqref{eq:PostnikovCongruenceSystemSupercuspidalInert} imply that the arguments of $e_p$  in \eqref{eq:GpDef2} are integers.

The final expression for $\widehat{H}(\psi)$ could \`a priori depend on the choice of $\cC$, which for sake of argument we temporarily write $\widehat{H}(\psi)_{\cC}$. However, the proof of course shows that $\widehat{H}(\psi)= \widehat{H}(\psi)_{\cC}$ is independent of $\cC$. Now, take two different choices of $\mathcal{C}$, $\mathcal{C}'$ that agree for all but one element $(x_{1,0}, x_{2,0}, x_{3,0},t_{0,0})$. Clearly, $\widehat{H}(\psi)_{\cC}- \widehat{H}(\psi)_{\cC'}=0$, but also note that all terms in the difference cancel except the one at $(x_{1,0}, x_{2,0}, x_{3,0},t_{0,0})$. Thus, the summands on the right hand sides of \eqref{eq:HhatFormulaPostPosnikovEvenExponent} and \eqref{eq:HhatInitialEvaluation2} are independent of the choice of coset representatives. 
\end{proof}

The following result further refines the evaluation from Lemma \ref{lemma:HhatInitialEvaluation}, when $a_1 = a_2 = a_3 = 1$.
\begin{mylemma}[Second simplification, unramified case]
\label{lemma:HhatInitialEvaluation2}
Suppose $L$ is unramified, $p$ is odd,  $k \geq \max(c_0, 2)$,
and $p^k$ is large.
Suppose $a_1 = a_2 = a_3 = 1$.   Let
\begin{equation}
\label{eq:PostnikovCongruenceSystemReducedSupercuspidalInert}
T_{\psi,\xi} := \{ s \in (\Z/p^n\Z)^\times: s^2 \ell_{\psi} - s \mathrm{Nm}(\ell_\xi) + \mathrm{Nm}(\ell_\xi) \ell_{\psi} =0\}.
\end{equation}
For $k=2n$, we have
\begin{equation}
\label{eq:HhatSimplifiedUnramifiedEvenk}
\widehat{H}(\psi) = \overline{\gamma} p^{k/2}
 \sumprime_{s \in T_{\psi,\xi}}
\psi\Big(\frac{s^2+\mathrm{Nm}(\ell_\xi)}{s^3}\Big)
\xi(s+\ell_\xi) e_{p^k}\Big(
 \frac{s\mathrm{Nm}(\ell_\xi)}{s^2+\mathrm{Nm}(\ell_\xi)}\Big).
\end{equation}
For $k=2n+1$, we have
\begin{equation}
\label{eq:HhatSimplifiedUnramifiedOddk}
\widehat{H}(\psi) = \frac{  \overline{\gamma} p^{5n}}{p^{2k}}
 \sumprime_{s \in T_{\psi,\xi}}
\psi\Big(\frac{s^2+\mathrm{Nm}(\ell_\xi)}{s^3}\Big)
\xi(s+\ell_\xi) e_{p^k}\Big(
 \frac{s\mathrm{Nm}(\ell_\xi)}{s^2+\mathrm{Nm}(\ell_\xi)} \Big) G_p',
\end{equation}
where  the $'$ on the summation indicates that $s$ runs over any fixed choice of coset representatives for $(\Z/p^k\Z)/ (p^n\Z/p^k\Z)$ that project to $T_{\psi, \xi}$ and where $G_p'$ depending on $s$ is given by \eqref{eq:Gp'Def} below. Each summand on the right-hand sides of \eqref{eq:HhatSimplifiedUnramifiedEvenk} and \eqref{eq:HhatSimplifiedUnramifiedOddk} is independent of the choices of coset representatives.
\end{mylemma}
\begin{proof}
We continue from Lemma \ref{lemma:HhatInitialEvaluation}.  Recall from Remark \ref{rem:fixedL} that throughout we have made a fixed a choice of $\ell_\xi \in \cO_L$ that satisfies $\Tr(\ell_\xi)=0$.
From $E_1=E_2=E_3=0$, we see that any $(x_1,x_2,x_3,t_0) \in S_{\psi,\xi,\mathbf{1}}$ satisfies $x_1 = x_2 = x_3$.  Say $s$ is this common value modulo $p^n$.
From $E_4=0$, we get that $t_0 = \ell_\xi + s \mod p^n\cO_L$, which completely determines $t_0 \in \cO_L/p^n\cO_L$. 
Fix any choice of lift of $s$ modulo $p^k$, and set $t_0=\ell_\xi+s \pmod{p^k}$, which determines a lift of $t_0$  modulo $p^k\cO_L$.
For these choices, we have  that $\mathrm{Nm}(t_0)\equiv s^2 + \mathrm{Nm}(\ell_\xi) \pmod{p^{k}}$ since $\Tr(\ell_\xi)=0$.  Note that $s^2 +\Nm(\ell_\xi) \neq 0$, since $\Tr(\ell_\xi)=0$ and $\pm \ell_\xi \not \in \Z_p^\times$.
Thus, for any choice of coset representatives modulo ${p^k}$ for 
\begin{equation}\label{eq:Tpsithetaprelim}
\{ s \in (\Z/p^n\Z)^\times: -\ell_{\psi} + s - \frac{s^3}{s^2 + \mathrm{Nm}(\ell_\xi)}=0 \},
\end{equation}
 there is a choice of coset representatives for the domain of summation in \eqref{eq:HhatFormulaPostPosnikovEvenExponent} or \eqref{eq:HhatInitialEvaluation2} so that these sets are in bijection.  The set in \eqref{eq:Tpsithetaprelim} is  in turn  in bijection with $T_{\psi,\xi}$.

  For $k=2n$, making this change of variables gives
\begin{equation}
\widehat{H}(\psi) = \frac{  \overline{\gamma} p^{5n}}{p^{2k}}
 \sumprime_{s \in T_{\psi,\xi}}
\psi\Big(\frac{s^2+\mathrm{Nm}(\ell_\xi)}{s^3}\Big)
\xi(s+\ell_\xi) e_{p^k}\Big(-\mathrm{Tr}(\ell_\xi+s)+
   3s 
 - 
 \frac{s^3}{s^2+\mathrm{Nm}(\ell_\xi)} \Big).
\end{equation}
which  simplifies as \eqref{eq:HhatSimplifiedUnramifiedEvenk}. 
The evaluation for $k=2n+1$ is the same, except for the simplification of $G_p$.  Define $m = m(s,\psi, \xi, n) \in \mz$ by
\begin{equation}
\label{eq:Mdef}
m = p^{-n} (\ell_{\psi} s^2 - s \mathrm{Nm}(\ell_\xi) + \ell_{\psi} \mathrm{Nm}(\ell_\xi)),
\end{equation}
which depends on the choice of coset representative modulo $p^k$ for $s$. 
We obtain
\begin{equation}
\label{eq:Gp'Def}
G_p' = 
\sum_{\substack{y_1, y_2, y_3 \shortmod{p^{ }} \\ t_1 \in \mathcal{O}_L/(p^{ })}}
e_{p^{}}\Big((-\sum_{i=1}^{3} y_i + \mathrm{Tr}(t_1)) (s^2 + \mathrm{Nm}(\ell_\xi))^{-1} m \Big)
e_{p^{}}(r_2-Ac_2), 
\end{equation}
which may depend on the choice of coset representatives, but the final expression \eqref{eq:HhatSimplifiedUnramifiedOddk} for $\widehat{H}(\psi)$ does not.
\end{proof}

Next we perform a similar analysis when $L$ is ramified.   
In the ramified case, we have $k \geq c_0 + 1$ (from the table \eqref{eq:supercuspidalTable}), and so by Lemma \ref{lemma:HhatVanishesUnlessCoprimeSupercuspidalInert} we may assume $a_1 = a_2 = a_3 = 1$.

\begin{mylemma}[Simplification, ramified case]
\label{lemma:HhatInitialEvaluationRamified}
Suppose $L$ is ramified,
$p$ is odd, $k \geq c_0+1$,  and $p^k$ is large.  
Suppose $a_1 = a_2 = a_3 = 1$.
For $k$ even, with $k=2n \geq 2$, we have  
\begin{equation}
\label{eq:HhatSimplifiedRamifiedEvenk}
\widehat{H}(\psi) = \epsilon p^{k/2}
\sumprime_{s \in T_{\psi,\xi}}
\Big(\frac{s}{p}\Big) 
\psi\Big(\frac{s^2+\mathrm{Nm}(\ell_\xi)}{s^3}\Big)
\xi(s+\ell_\xi) e_{p^k}\Big(
  \frac{s \mathrm{Nm}(\ell_\xi)}{s^2+\mathrm{Nm}(\ell_\xi)}\Big),
\end{equation}
for some $\epsilon$ of modulus $1$ depending only on the isomorphism class of $L$.

For $k$ odd, with $k=2n+1$, the sum $\widehat{H}(\psi)$ vanishes unless $v(\ell_\psi)>0$, in which case let
\begin{equation}
\label{eq:PostnikovCongruenceSystemReducedSupercuspidalRam2}
T_{\psi,\xi}' := \{ s \in (\Z/p^n\Z)^\times: s^2 \frac{\ell_{\psi}}{p} - s \frac{\mathrm{Nm}(\ell_\xi)}{p} + \frac{\mathrm{Nm}(\ell_\xi) \ell_{\psi}}{p} \equiv 0\pmod{p^n}\}.
\end{equation}
 For some $\epsilon$ of modulus $1$ depending only on the isomorphism class of $L$, we have
\begin{equation}
\label{eq:HhatSimplifiedRamifiedOddk}
\widehat{H}(\psi) = \epsilon p^{(k+1)/2}
\sumprime_{s \in T_{\psi,\xi}'}
\Big(\frac{s}{p}\Big) 
\psi\Big(\frac{s^2+\mathrm{Nm}(\ell_\xi)}{s^3}\Big)
\xi(s+\ell_\xi) e_{p^k}\Big(
  \frac{s \mathrm{Nm}(\ell_\xi)}{s^2+\mathrm{Nm}(\ell_\xi)}\Big).
\end{equation}
The $'$ on the summation in \eqref{eq:HhatSimplifiedRamifiedEvenk} and \eqref{eq:HhatSimplifiedRamifiedOddk} has the same meaning as in Lemma \ref{lemma:HhatInitialEvaluation2}.
\end{mylemma}
\begin{proof}
We roughly follow the strategy of proof of Lemma \ref{lemma:HhatInitialEvaluation}.   
Set $n =\lfloor k/2 \rfloor$.  We replace $x_i$ by $x_i(1 + p^n y_i)$ where the new $x_i $ runs over 
$(\Z/p^k\Z)/ (p^n\Z/p^k\Z)$, and $y_i \in \Z/p^{k-n}\Z$. However, for $t$ we choose a uniformizer $\pi$ of $L$ and write $t = t_0 ( 1+ \pi^{k-1} t_1)$, with $t_0 \in (\cO_L/p^k\cO_L)/(\fp^{k-1}/p^k\cO_L)$ and $t_1\in \cO_L/\fp^{2k-(k-1)}$. 
Precisely, we let $(x_1, x_2, x_3, t_0)$ run over a fixed choice $\mathcal{C}'$ of coset representatives for
$[(\Z/p^k\Z)/ (p^n\Z/p^k\Z)]^3 \times (\cO_L/p^k\cO_L)/(\fp^{k-1}/p^k\cO_L)$.
The Taylor expansions involving solely the $y_i$ variables are the same as in the proof of Lemma \ref{lemma:HhatInitialEvaluation}, but those involving $t_1$ need modification.  We have
$$
\psi(\Nm( 1+ \pi^{k-1} t_1)) = e_{p^k} ( \ell_\psi \Tr(\pi^{k-1}t_1 - \tfrac{1}{2}(\pi^{k-1}t_1)^2)),
$$
and 
$$
\xi( 1+ \pi^{k-1} t_1)) = e_{p^k} (\Tr(\ell_\xi \pi^{k-1}t_1 - \tfrac{\ell_\xi}{2}(\pi^{k-1}t_1)^2)).
$$ 
For one of the additive components, we have
\begin{equation}
e_{p^k}(-\mathrm{Tr}(t_0 (1 + \pi^{k-1} t_1))) 
= e_{p^k}(-\mathrm{Tr}(t_0)) e_{p^{k}}(-\mathrm{Tr}( \pi^{k-1}t_0 t_1)).
\end{equation}
Recalling \eqref{eq:TraceImage}, for the part involving $-\frac{x_1 x_2 x_3}{\mathrm{Nm}(t_0)}$, we have
\begin{multline}
- \frac{x_1 x_2 x_3}{\mathrm{Nm}(t_0)} 
\frac{(1 + p^n y_1)(1+p^n y_2)(1 + p^n y_3)}{(1 + \pi^{k-1} t_1)(1 + \overline{\pi^{k-1} t_1})}
= -A\Big[1 + p^n(y_1 + y_2 + y_3) -  \mathrm{Tr}(\pi^{k-1} t_1))
\\
+ p^{2n}(y_1 y_2 + y_1 y_3 + y_2 y_3) -  p^n\mathrm{Tr}(\pi^{k-1} t_1) (y_1 + y_2 + y_3) 
+  \Tr((\pi^{k-1}t_1)^2) + \Nm(\pi^{k-1}t_1)
\Big]+ O(p^{k-1+n}).
\end{multline}
Hence putting these together, along with the relevant Taylor expansions already computed in the proof of Lemma \ref{lemma:HhatInitialEvaluation}, we obtain (recalling the $p^{-d/2}$ factor in \eqref{eq:KloostermanDefSCinertCaseGeneral}) that 
\begin{equation}
\widehat{H}(\psi) =  \frac{\overline{\gamma}p^{-\frac12}}{p^{2k}}
\sumstar_{(x_1, x_2, x_3, t_0) \in \cC'}
\psi\Big(\frac{\mathrm{Nm}(t_0)}{x_1 x_2 x_3}\Big)
\xi(t_0) e_{p^k}(-\mathrm{Tr}(t_0)+
  x_1 +  x_2 +  x_3 
 - 
 A)
 \cdot S_1,
\end{equation}
where
\begin{multline}\label{eq:S1computation_ramified}
S_1 = 
\sum_{\substack{y_1, y_2, y_3 \shortmod{p^{k-n}} \\ t_1 \in \mathcal{O}_L/(\fp^{k+1})}}
e_{p^{k-n}}(-(\ell_{\psi}+A)(y_1 + y_2 + y_3) + x_1 y_1 + x_2 y_2 + x_3 y_3
)
\\
e_{p^{k}}(\mathrm{Tr}(\pi^{k-1} t_1[\ell_{\psi} + \ell_\xi - t_0 + A]))
e_{p^{k-2n}}(\tfrac12 \ell_{\psi} (y_1^2 + y_2^2 + y_3^2) - A (y_1 y_2 + y_1 y_3 + y_2 y_3))
\\
e_{p^{k-n}}(A (y_1 + y_2 + y_3) \mathrm{Tr}(\pi^{k-1} t_1))
e_{p^{k}}(-\tfrac{1}{2} \ell_{\psi} \mathrm{Tr}((\pi^{k-1}t_1)^2)-\tfrac{1}{2} \mathrm{Tr}(\ell_\xi (\pi^{k-1}t_1)^2))
\\
e_{p^{k}}(-A    \Tr((\pi^{k-1}t_1)^2) -A \Nm(\pi^{k-1}t_1)).
\end{multline}

Now we evaluate $S_1$.   First suppose $k=2n$ and recall that $\Tr(\fp^{k-1}) = (p^n)$. Therefore, in this case, all parts of $S_1$ that are degree 2 in terms of $y_i$, as well as the term $A (y_1 + y_2 + y_3) \mathrm{Tr}(\pi^{k-1} t_1)$, are trivial. Now, consider the change of variables $t_1 \to t_1 + \pi$. The terms of degree 2 in $t_1$ are unchanged by this transformation, so the sum vanishes unless $\ell_{\psi} + \ell_\xi -t_0 + A = 0 \in \mathcal{O}_L/\mathfrak{p}^{k-1}$.  The $y_i$-aspect is the same as in the unramified case. Similarly, define
$$E = (E_1,E_2,E_3,E_4): ((\Z/p^n\Z)^\times)^3 \times (\cO_L/\fp^{k-1})^\times \to (\Z/p^n\Z)^3 \times \cO_L/\fp^{k-1}$$
and set
\begin{equation}
S'_{\psi, \xi, \mathbf{a}} = \{ (x_1,x_2,x_3,t_0) \in ((\Z/p^n\Z)^\times)^3 \times (\cO_L/\fp^{k-1})^\times : E(x_1,x_2,x_3,t_0)=(0,0,0,0)\}.
\end{equation}
Then 
$$\widehat{H}(\psi) = \frac{\overline{\gamma} p^{5n-\frac{1}{2}}}{p^{2k}}
\sumprime_{S'_{\psi, \xi, \mathbf{a}}}
\psi\Big(\frac{\mathrm{Nm}(t_0)}{x_1 x_2 x_3}\Big)
\xi(t_0) e_{p^k}(-\mathrm{Tr}(t_0)+
  x_1 +  x_2 +  x_3 
 - A
 ) S'_1,
$$
with 
\begin{multline}
\label{eq:S'1}
S'_1 = \sum_{ t_1 \in \mathcal{O}_L/\fp}
e_{p^{k}}(\mathrm{Tr}(\pi^{k-1} t_1[\ell_{\psi} + \ell_\xi - t_0 + A]))
e_{p^{k}}(-\tfrac{1}{2} \ell_{\psi} \mathrm{Tr}((\pi^{k-1}t_1)^2))
\\
e_{p^k}(
-\tfrac{1}{2} \mathrm{Tr}(\ell_\xi (\pi^{k-1}t_1)^2))
e_{p^{k}}(-A    \Tr((\pi^{k-1}t_1)^2) -A \Nm(\pi^{k-1}t_1)).
\end{multline}
As in the unramified case, the summand of $\widehat{H}(\psi)$ is independent of the choice of coset representative. We choose a common coset representative $s$ modulo $p^k$ for $x_1=x_2=x_3$, and choose the coset representative for $t_0$ modulo $p^k \cO_L$ determined by $t_0=\ell_\xi+s$.  Thus, $\ell_{\psi} + \ell_\xi - t_0 + A \equiv \ell_{\psi} -s + A \pmod{p^k \cO_L},$ and $\ell_{\psi} -s + A \in \Z_p$ and $\equiv 0 \pmod{p^n}$. So, $v_L(\pi^{k-1} t_1[\ell_{\psi} + \ell_\xi - t_0 + A]) \geq k-1+2n=2k-1$, i.e.\ the corresponding term in \eqref{eq:S'1} vanishes.   

Now we assume in addition that $ \Tr \pi = 0$, which is possible since $p\neq 2$.  Using this, we see that  
$$S_1' = \sum_{t_1 \in \cO_L/\fp} e_p 
\Big( \frac{\pi^{2(k-1)}}{p^{k-1}} [-\tfrac{\ell_{\psi}}{2}  \mathrm{Tr}(t_1^2)-\tfrac{1}{2} \mathrm{Tr}(\ell_\xi t_1^2)
-A    (\Tr(t_1^2) - \Nm(t_1))]
\Big).
$$ 
Here we can identify $t_1$ as an element of $\Z/p\Z$, so that
\begin{equation}
S_1' = 
\sum_{x \shortmod{p}} e_{p}\Big( \frac{\pi^{2(k-1)}}{p^{k-1}} [ - \ell_{\psi}x^2 
-A x^2]\Big) = \epsilon_p  p^{1/2} \Big(\frac{-s (p/\pi^2)}{p}\Big)= \epsilon_p  p^{1/2} \Big(\frac{s (p/\Nm(\pi))}{p}\Big),
\end{equation}
where we used 
$\mathrm{Tr}(\ell_{\xi}) = 0$ and 
$\ell_{\psi} + A \equiv s \pmod{p}$.
Substituting the value of $t_0$ and simplifying gives \eqref{eq:HhatSimplifiedRamifiedEvenk}, using $A \equiv \frac{s^3}{s^2 + \mathrm{Nm}(\ell_{\xi})} \pmod{p^k}$ to aid in simplification. Note that the square class of $p/\Nm(\pi)$ does not depend on the choice of uniformizer $\pi$ (trace 0 or otherwise). Indeed, if $\pi'$ is another uniformizer, then $\pi'=u\pi$ for some $u \in \cO_L^\times$, but then $\Nm(\pi')/\Nm(\pi) = \Nm(u)$ is a square modulo $p$. 

Next consider $k$ odd of the form $k=2n+1$ and pick up the computation of $S_1$ at \eqref{eq:S1computation_ramified}. In this case, we make the change of variables $y_i \to y_i + p z_i$, with $y_i$ running modulo $p$ and $z_i$ running modulo $p^n$.  Likewise, we replace $t_1$ by $t_1 + \pi u_1$, where $t_1 \in \mathcal{O}_L/\fp$ and $u_1 \in \mathcal{O}_L/ \fp^k$.  The sums over the $z_i$ show $S_1$ vanishes unless $E_1$, $E_2$ and $E_3$  
vanish modulo $p^n$. Meanwhile, additive orthogonality of characters over $L$ (i.e., \eqref{eq:addorth}) shows that the $u_1$ sum vanishes unless $\ell_{\psi} + \ell_\xi - t_0 + A \equiv 0 \pmod{\fp^{k-1}}$. Thus, we have $$\widehat{H}(\psi) = \frac{\overline{\gamma} p^{5n+\frac{1}{2}}}{p^{2k}}
\sumprime_{S'_{\psi, \xi, \mathbf{a}}}
\psi\Big(\frac{\mathrm{Nm}(t_0)}{x_1 x_2 x_3}\Big)
\xi(t_0) e_{p^k}(-\mathrm{Tr}(t_0)+
  x_1 +  x_2 +  x_3 
 - A
 ) S'_1,
$$
with  
\begin{multline}
\label{eq:S1'computation_ramified2}
S_1' = 
\sum_{\substack{y_1, y_2, y_3 \shortmod{p} \\ t_1 \in \mathcal{O}_L/\fp}}
e_{p^{k-n}}(-(\ell_{\psi}+A)(y_1 + y_2 + y_3) + x_1 y_1 + x_2 y_2 + x_3 y_3
)
\\
e_{p^{k}}(\mathrm{Tr}(\pi^{k-1} t_1[\ell_{\psi} + \ell_\xi - t_0 + A]))
e_{p^{k-2n}}(\tfrac12 \ell_{\psi} (y_1^2 + y_2^2 + y_3^2) - A (y_1 y_2 + y_1 y_3 + y_2 y_3))
\\
e_{p^{k-n}}(A (y_1 + y_2 + y_3) \mathrm{Tr}(\pi^{k-1} t_1))
e_{p^{k}}(-\tfrac{1}{2} \ell_{\psi} \mathrm{Tr}((\pi^{k-1}t_1)^2)-\tfrac{1}{2} \mathrm{Tr}(\ell_\xi (\pi^{k-1}t_1)^2))
\\
e_{p^{k}}(-A    \Tr((\pi^{k-1}t_1)^2) -A \Nm(\pi^{k-1}t_1)).
\end{multline}
Again, the summand of $\widehat{H}(\psi)$ is independent of the choice of coset representative. We choose a common coset representative $s$ modulo $p^k$ for  $x_1=x_2=x_3$, and choose the coset representative for $t_0$ modulo $p^k \cO_L$ determined by $t_0=\ell_\xi+s$.  
We may also now  assume that $\Tr \pi = 0$, so that $\pi^2 = -\Nm(\pi) \in \Z_p$. We get 
\begin{multline}\label{eq:S1'computation_ramified3}
S_1' = 
\sum_{\substack{y_1, y_2, y_3 \shortmod{p} \\ t_1 \in \mathcal{O}_L/\fp}}
e_{p^{k-n}}(-(\ell_{\psi}+A-s)(y_1 + y_2 + y_3 - \frac{\pi^{2n}}{p^{n}}\mathrm{Tr}( t_1)) 
)
\\
\times
e_{p}(\tfrac12 \ell_{\psi} (y_1^2 + y_2^2 + y_3^2) - A (y_1 y_2 + y_1 y_3 + y_2 y_3)) e_{p}(\frac{\pi^{2n}}{p^n}A (y_1 + y_2 + y_3) \mathrm{Tr}( t_1))
\\
\times e_{p}(-\tfrac{1}{2}\frac{\pi^{2(k-1)}}{p^{k-1}} \ell_{\psi} \mathrm{Tr}((t_1)^2)-\tfrac{1}{2}\frac{\pi^{2(k-1)}}{p^{k-1}} \mathrm{Tr}(\ell_\xi (t_1)^2))e_{p}(-\frac{\pi^{2(k-1)}}{p^{k-1}} A(    \Tr((t_1)^2) + \Nm(t_1))).
\end{multline}
We can identify $t_1$ as an element $t$ of $\Z/p\Z$ and
make the change of variables $t \mapsto \frac{p^{n}}{\pi^{2n}} t$ to get
\begin{multline}\label{eq:S1'computation_ramified5}
S_1' = 
\sum_{y_1, y_2, y_3, t \shortmod{p} }
e_{p^{n+1}}(-(\ell_{\psi}+A-s)(y_1 + y_2 + y_3 -2 t) 
)
\\
e_{p}(\tfrac12 \ell_{\psi} (y_1^2 + y_2^2 + y_3^2) - A (y_1 y_2 + y_1 y_3 + y_2 y_3))
\\
e_{p}(2A (y_1 + y_2 + y_3) t)
e_{p}(-(3A+\ell_\psi)    t^2 ).
\end{multline}
Of course, this expression for $S_1'$ is independent of the choice of uniformizer. 
We can simplify this further.  Using $\mathrm{Nm}(\ell_\xi) \equiv 0 \pmod{p}$, \eqref{eq:PostnikovCongruenceSystemReducedSupercuspidalInert} implies that $\ell_{\psi} \equiv 0 \pmod{p}$.  
In addition, $A \equiv \frac{s^3}{s^2 + \mathrm{Nm}(\ell_\xi)} \equiv s \pmod{p}$.  Recall the definition of $m$ from \eqref{eq:Mdef}.
Then
\begin{equation}
p^{-n-1}[s - \ell_{\psi} - A] = p^{-n-1}[s - \ell_{\psi} - \frac{s^3}{s^2 + \mathrm{Nm}(\ell_\xi)}]
\equiv p^{-1} s^{-2} m \pmod{1}.
\end{equation}
Thus 
$S_1' = G_p(\mathcal{Q}, \mathcal{L})$, with
$$
\mathcal{Q}(y_1, y_2, y_3, t) = -s  (y_1 y_2 + y_1 y_3 + y_2 y_3)+2s(y_1 + y_2 + y_3) t - 3s t^2$$
and $$\mathcal{L}(y_1, y_2, y_3, t) = ms^{-2}(y_1 + y_2 + y_3 -2 t). $$

Using a computer algebra package, we calculate that
Using a computer algebra package shows $\mathcal{Q}$  has rank $ 3$ for $p$ odd.  
The $0$-eigenspace of the radical of $\mathcal{Q}$ is the span of $(1,1,1,1)$.  Lemma \ref{lemma:quadraticGaussSum} shows that $S_1'$ vanishes unless $\mathcal{L}$  vanishes at $(1,1,1,1)$, equivalently, unless $m \equiv 0 \pmod{p}$.  Then quotienting by the radical and applying Lemma \ref{lemma:quadraticGaussSumEvaluation} we get that $S_1' =  p^{5/2} \epsilon' (\frac{s}{p})$ 
where $\epsilon'$ has modulus $1$ and depends only on $p$.
This means
\begin{equation}
\widehat{H}(\psi) = \epsilon p^{\frac{k+1}{2}}
\sumprime_{s \in T_{\psi,\xi}'}
\Big(\frac{s}{p}\Big)
\psi\Big(\frac{s^2+\mathrm{Nm}(\ell_\xi)}{s^3}\Big)
\xi(s+\ell_\xi) e_{p^k}\Big(
  \frac{s \mathrm{Nm}(\ell_\xi)}{s^2+\mathrm{Nm}(\ell_\xi)}\Big),
\end{equation}
where
\begin{equation}\label{eq:Tprimedef2}
T_{\psi,\xi}' = \{ s \in (\Z/p^n\Z)^\times: s^2 \ell_{\psi} - s \mathrm{Nm}(\ell_\xi) + \mathrm{Nm}(\ell_\xi) \ell_{\psi} \equiv 0 \pmod{p^{n+1}}\}.
\end{equation} 
Since $p$ divides both $\ell_{\psi}$ and $\mathrm{Nm}(\ell_{\xi})$, we may clear a factor of $p$ to get the congruence stated in \eqref{eq:PostnikovCongruenceSystemReducedSupercuspidalRam2}. 
\end{proof}

\subsection{Bounds for $k \geq c_0 + 1$}
\label{section:kbiggish}
Recall from Lemma \ref{lemma:HhatVanishesUnlessCoprimeSupercuspidalInert} that we may assume $a_1 = a_2 =a_3 = 1$ when $k \geq c_0 + 1$.
The goal of this subsection is to prove the following.
\begin{myprop}
\label{prop:HhatBoundkbiggish}
Suppose $p$ is odd and $k \geq c_0 + 1$.  We have
\begin{equation}
\label{eq:HhatBoundkbiggish}
|\widehat{H}(\psi)| \ll \frac{p^{\lfloor 3k/2\rfloor}}{C(\psi)}.
\end{equation}
\end{myprop}
We need some lemmas before we can make this deduction.  One simple observation is that \eqref{eq:HhatBoundkbiggish} is trivial for $p^k \ll 1$, so in the proof we may assume that $p^k$ is large.  In what follows, we will impose this assumption without further mention.

\begin{mylemma}
\label{lemma:propHhatBoundramified}
Proposition \ref{prop:HhatBoundkbiggish} holds under the additional assumption that $k$ is even.
\end{mylemma}
\begin{proof}
Suppose that $L$ is unramified and $k=2n$.  If $v(\mathrm{Nm}(\ell_\xi)) \geq n$, then the congruence \eqref{eq:PostnikovCongruenceSystemReducedSupercuspidalInert} has solutions if and only if $v(\ell_{\psi}) \geq n$.  
In this case, every value of $s \pmod{p^n}$ is a solution.  Moreover, Lemma \ref{lem:postnikov} with $L=\Q_p$  shows that the condition $v(\ell_{\psi}) \geq n$ implies $C(\psi) \leq p^{k-n} = p^n$.  Therefore, applying the trivial bound to \eqref{eq:HhatSimplifiedUnramifiedEvenk} gives 
\begin{equation*}
|\widehat{H}(\psi)| \leq p^{k/2} p^n = \frac{p^{k/2} p^{2n}}{p^n} \leq \frac{p^{3k/2}}{C(\psi)}.
\end{equation*}
In the opposite case with $v(\mathrm{Nm}(\ell_\xi)) < n$, then \eqref{eq:PostnikovCongruenceSystemReducedSupercuspidalInert} implies $v(\ell_{\psi}) = v(\mathrm{Nm}(\ell_\xi))$.  By Lemma \ref{lem:postnikov} with $L=\Q_p$ , this implies $C(\psi) = p^{k-v(\mathrm{Nm}(\ell_\xi))}$.  Set $\ell_{\psi}' = p^{-v(\mathrm{Nm}(\ell_\xi))} \ell_{\psi}$  and $\ell_{\xi}'  = p^{-v(\mathrm{Nm}(\ell_\xi))} \Nm(\ell_\xi)$. 
Then \eqref{eq:PostnikovCongruenceSystemReducedSupercuspidalInert} is equivalent to
\begin{equation}
\label{eq:tcongruenceSupercuspdialRamified2}
\ell_{\psi}' s^2 - \ell_{\xi}'  s  +\mathrm{Nm}(\ell_\xi) \ell_{\psi}' \equiv 0 \pmod{p^{n-v(\mathrm{Nm}(\ell_\xi))}}.
\end{equation}
The discriminant of this polynomial is $\ell_\xi'^2 - 4 \ell_{\psi}' \mathrm{Nm}(\ell_\xi) \ell_{\psi}' \equiv {\ell_\xi'}^2 \pmod{p}$, which is nonzero modulo $p$.  Hence Hensel's lemma implies that the number of solutions to \eqref{eq:tcongruenceSupercuspdialRamified2}, for $s \pmod{p^{n-v(\mathrm{Nm}(\ell_\xi))}}$ is at most $2$.  Each solution to \eqref{eq:tcongruenceSupercuspdialRamified2} gives rise to $p^{v(\mathrm{Nm}(\ell_\xi))}$ solutions to \eqref{eq:PostnikovCongruenceSystemReducedSupercuspidalInert}.  Hence
\begin{equation*}
|\widehat{H}(\psi)| \leq
 2    p^{k/2} p^{v(\mathrm{Nm}(\ell_\xi))} 
=  \frac{2  p^{3k/2}}{p^{k-v(\mathrm{Nm}(\ell_\xi))}}
= \frac{2   p^{3k/2}}{C(\psi)}.
\end{equation*}

Now suppose that $L$ is ramified.  If $k=2n$ is even, then the trivial bound applied to \eqref{eq:HhatSimplifiedRamifiedEvenk} is the same as the trivial bound applied to \eqref{eq:HhatSimplifiedUnramifiedEvenk} as  in the just-treated unramified case.  Therefore, the desired bound holds here.  
\end{proof}
  Next we examine the case of  Proposition \ref{prop:HhatBoundkbiggish}  that   $L$ is unramified and $k$ is odd.  For this, we need to extract some more information, as follows.
\begin{mylemma}
\label{lemma:ExtraInfoUnramifiedkodd}
Suppose $p$ is odd, $L$ is unramified, $k \geq c_0+1$,  and $p^k$ is large.  
Suppose $a_1 = a_2 = a_3 = 1$, and $k=2n+1$.
The quadratic exponential sum $G_p'$ vanishes unless
\begin{equation}
  \label{eq:PostnikovCongruenceSystemReducedSupercuspidalInert2}
  \ell_{\psi} s^2  -  \mathrm{Nm}(\ell_\xi) s + \mathrm{Nm}(\ell_\xi) \ell_{\psi} \equiv 0 \pmod{p^{n+1}}.
 \end{equation}
Moreover, $G_p' = \epsilon p^3$, where $\epsilon$ has modulus $1$ and depends only on $p$.  
\end{mylemma}
\begin{proof}
First note 
$k \geq c_0 + 1$ implies
$\mathrm{Nm}(\ell_\xi) \equiv 0 \pmod{p}$, so \eqref{eq:PostnikovCongruenceSystemReducedSupercuspidalInert} implies $\ell_{\psi} \equiv 0 \pmod{p}$.  Thus $r_2 \equiv 0 \pmod{p}$ (recall \eqref{eq:r2def} for its definition).  Also, $A \equiv s \pmod{p}$.  Now we can view $G_p'$ as a sum of the form $G_p(\mathcal{Q}, \mathcal{L})$ as in \eqref{eq:GpDef1}, where $\mathcal{Q} = - s c_2$ (recall \eqref{eq:c2def} for its definition), and where $\mathcal{L}$ takes the form
\begin{equation}
\mathcal{L}(y_1, y_2, y_3, t) = p^{-n} s^{-2} m [y_1 + y_2 + y_3 - \mathrm{Tr}(t)],
\end{equation}
with $m = m(s,\psi, \xi, n) \in \mz$ defined by \eqref{eq:Mdef}.  

Choose a non-zero $\alpha \in \cO_L/(p)$ with $\Tr(\alpha)=0$ so that we may write $t_1 = u + v \alpha$ where $u,v \in \mf_p$.    Then $\mathrm{Tr}(t) = 2u$ and $t^2 + t \overline{t} + \overline{t}^2 = 3u^2 -\Nm(\alpha) v^2$.  In these coordinates,
\begin{equation}
c_2 = (y_1 y_2 + y_1 y_3 + y_2 y_3) - 2u(y_1 + y_2 + y_3) + (3u^2 -\Nm(\alpha) v^2).
\end{equation}
Note that $c_2$ takes the form $\frac12 (y_1, y_2, y_3, u, v) M (y_1, y_2, y_3, u, v)^{\intercal}$, where
\begin{equation}
\label{eq:M3matrixFormula}
 M = 
 \begin{pmatrix}
  0 & 1 & 1 & -2 & 0 \\
  1  & 0 & 1 & -2 & 0 \\
   1 &  1 & 0 & -2 & 0 \\
   -2 & -2  & -2  &  6 & 0 \\
   0 & 0  &  0 &  0  & -2\Nm(\alpha)
 \end{pmatrix}.  
\end{equation}
  Using a computer algebra package, we calculate that the eigenvalues of $M$ are 
  $8$, $-1$, $-1$, $0$, and $-2\Nm(\alpha)  $.
Moreover, the eigenspace corresponding to the $0$-eigenvalue is spanned by $(1,1,1,1,0)^{\intercal}$.

Now that the radical of $\mathcal{Q}$ has been determined, we next turn to the question of when the linear form $\mathcal{L}$ vanishes on it.
 Observe that the vanishing of $\mathcal{L}$ on $(1,1,1,1,0)^{\intercal}$ simply means that $m \equiv 0 \pmod{p}$, which is the same as \eqref{eq:PostnikovCongruenceSystemReducedSupercuspidalInert2}.

Finally, note that the quadratic form has rank $4$, and $\mathcal{L} = 0$ on the orthogonal complement of the span of $(1,1,1,1,0)^{\intercal}$.  The evaluation of $G_p'$ follows from
 Lemmas \ref{lemma:quadraticGaussSum} and \ref{lemma:quadraticGaussSumEvaluation}, and in particular does not depend on the choice of $\alpha$.
\end{proof}

\begin{mylemma}\label{lemma:k=2n+1}
Proposition \ref{prop:HhatBoundkbiggish} holds under the additional assumption that $k=2n+1$.
\end{mylemma}
Lemmas \ref{lemma:propHhatBoundramified} and \ref{lemma:k=2n+1} then imply Proposition \ref{prop:HhatBoundkbiggish}.
\begin{proof}
We follow the same method as in Lemma \ref{lemma:propHhatBoundramified}, treating the ramified and unramified cases simultaneously.  Looking at \eqref{eq:PostnikovCongruenceSystemReducedSupercuspidalInert2}, resp.\ \eqref{eq:PostnikovCongruenceSystemReducedSupercuspidalRam2}, if $v(\mathrm{Nm}(\ell_\xi)) \geq n+1$, then $v(\ell_{\psi}) \geq n+1$, so $C(\psi) \leq p^{k-n-1} = p^n$. From Lemma \ref{lemma:ExtraInfoUnramifiedkodd} 
and \eqref{eq:HhatSimplifiedUnramifiedOddk}, resp.\ Lemma \ref{lemma:HhatInitialEvaluationRamified}
 applying the triangle inequality implies
\begin{equation*}
|\widehat{H}(\psi)| \leq \frac{  p^{5n}}{p^{2k}} p^3  p^n =    p^{\frac{k+1}{2}}p^n \leq \frac{   p^{\frac{k+1}{2}} p^n p^{n}}{C(\psi)}
= \frac{p^{\frac{3k-1}{2}}}{C(\psi)}.
\end{equation*}
If $v(\mathrm{Nm}(\ell_\xi)) \leq n$, then the congruence \eqref{eq:PostnikovCongruenceSystemReducedSupercuspidalInert2}  resp.\ \eqref{eq:PostnikovCongruenceSystemReducedSupercuspidalRam2} implies that $v(\ell_{\psi}) = v(\mathrm{Nm}(\ell_\xi))$, and hence $C(\psi) = p^{k-v(\mathrm{Nm}(\ell_\xi))}$.  As in the proof of Lemma \ref{lemma:propHhatBoundramified}, let 
 $\ell_{\psi}' = p^{-v(\mathrm{Nm}(\ell_\xi))} \ell_{\psi}$ and $\ell_{\xi}'  = p^{-v(\mathrm{Nm}(\ell_\xi))} \Nm(\ell_\xi)$.  Then \eqref{eq:PostnikovCongruenceSystemReducedSupercuspidalInert2} resp.\ \eqref{eq:PostnikovCongruenceSystemReducedSupercuspidalRam2} is equivalent to
\begin{equation}
\label{eq:tcongruenceSupercuspdialRamified3}
\ell_{\psi}' s^2 - \ell_{\xi}'  s  +\mathrm{Nm}(\ell_\xi) \ell_{\psi}' \equiv 0 \pmod{p^{n+1-v(\mathrm{Nm}(\ell_\xi))}}.
\end{equation}
The discriminant of this polynomial is ${\ell_\xi'}^2 - 4 \ell_{\psi}' \mathrm{Nm}(\ell_\xi) \ell_{\psi}' \equiv {\ell_\xi'}^2 \pmod{p}$, which is nonzero modulo $p$.  Hence Hensel's lemma implies that the number of solutions to \eqref{eq:tcongruenceSupercuspdialRamified3}, for $s \pmod{p^{n+1-v(\mathrm{Nm}(\ell_\xi))}}$ is at most $2$.  Each solution to \eqref{eq:tcongruenceSupercuspdialRamified3} gives rise to $p^{v(\mathrm{Nm}(\ell_\xi))-1}$ solutions to \eqref{eq:PostnikovCongruenceSystemReducedSupercuspidalInert2}  resp.\ \eqref{eq:PostnikovCongruenceSystemReducedSupercuspidalRam2}.  Applying this information to \eqref{eq:HhatSimplifiedUnramifiedOddk} resp.\ \eqref{eq:HhatSimplifiedRamifiedOddk} leads to 
\begin{equation*}
|\widehat{H}(\psi)| \leq
 2  \frac{p^{5n}}{p^{2k}} p^{v(\mathrm{Nm}(\ell_\xi))-1} p^3 = 2   p^{\frac{k+1}{2}} p^{v(\mathrm{Nm}(\ell_\xi))-1}
=  2  p^{n + v(\mathrm{Nm}(\ell_\xi))}
= \frac{2  p^{k+n}}{C(\psi)} = \frac{2  p^{\frac{3k-1}{2}}}{C(\psi)}. \qedhere
\end{equation*}
\end{proof}

\subsection{Trivial character}
\label{section:HhatBoundTrivialCharacter}
The following lemma improves upon Proposition \ref{prop:HhatBoundkbiggish} for $\psi$ trivial.
\begin{mylemma}
\label{lemma:HhatpsiTrivialInertkLarger}
Suppose $p$ is odd, $k \geq c_0 +1$, $p^k$ is large, and
$\psi = \psi_0$ is trivial.  Then 
\begin{equation}
\label{eq:Hhatpsi0boundViaRamanujanSum}
|\widehat{H}(\psi_0, 1, 1, 1)| \leq  |S(\mathrm{Nm}(\ell_{\xi}), 0;p^k)|.
\end{equation}
\end{mylemma}
\begin{proof}
Suppose $L$ is unramified and $k=2n$.  The congruence \eqref{eq:PostnikovCongruenceSystemReducedSupercuspidalInert}, with $\ell_{\psi} = 0$, implies
$\widehat{H}(\psi_0) = 0$ unless 
 $\mathrm{Nm}(\ell_\xi) \equiv 0 \pmod{p^n}$ (if $\widehat{H}(\psi_0) = 0$ then \eqref{eq:Hhatpsi0boundViaRamanujanSum} holds trivially).
By Lemma \ref{lemma:HhatInitialEvaluation2}, we have
\begin{equation}
\widehat{H}(\psi_0) =  \overline{\gamma} p^{n}
\sumstar_{s \shortmod{p^n}} \xi(s+\ell_\xi) e_{p^k}\Big(\frac{s \mathrm{Nm}(\ell_\xi)}{s^2 + \mathrm{Nm}(\ell_\xi)} \Big).
\end{equation}
Using Lemma \ref{lemma:thetarestrictedtoZp}, we have $\xi(s)=1$, so 
using $\Tr \ell_\xi=0$ we have 
$$
\xi(s + \ell_{\xi}) = 
\xi(1+s^{-1} \ell_\xi) = e_{p^{k}}(\Tr \left( s^{-1}\ell_\xi^2 + s^{-3}\ell_\xi^4/3 + O(5^{-1}\ell_\xi^6)\right)).$$ 
Note that $\ell_\xi^4= (-\Nm(\ell_\xi))^2 = \Nm(\ell_\xi)^2$, and $v(\Nm(\ell_\xi))\geq n$. 
Assuming temporarily that $p \neq 3$, we then have $s^{-3} \ell_{\xi}^4/3 \equiv 0 \pmod{p^k}$.
Similarly,
$e_{p^k}(\frac{s \mathrm{Nm}(\ell_\xi)}{s^2 + \mathrm{Nm}(\ell_\xi)}) = e_{p^k}(s^{-1} \mathrm{Nm}(\ell_\xi))$, using that $\mathrm{Nm}(\ell_\xi)^2 \equiv 0 \pmod{p^k}$.  Therefore, for $p \neq 3$,
\begin{equation}
\label{eq:Hhatpsi0ViaRamanujan}
\widehat{H}(\psi_0) =   \overline{\gamma}p^{n} \sumstar_{s \shortmod{p^n}} e_{p^k}(-s^{-1} \mathrm{Nm}(\ell_\xi)) =   S(\mathrm{Nm}(\ell_\xi), 0;p^k).
\end{equation}
We next argue that \eqref{eq:Hhatpsi0ViaRamanujan} holds for $p=3$.  In this case, we have that $s^{-3} \ell_{\xi}^4/3 \equiv 0 \pmod{p^{k-1}}$.  Changing variables $s \rightarrow s (1+p)^{-1}$ shows the sum vanishes unless $v(\ell_{\xi}^2) \geq k-1$.   In turn, this implies $v(\ell_{\xi}^4/3) \geq 2(k-1)-1$ which for $k$ large (namely, $k \geq 3$) implies $\ell_{\xi}^4/3 \equiv 0 \pmod{p^k}$.  This discussion shows the term with $\ell_{\xi}^4/3$ may be dropped when $p=3$, in which case we obtain the claimed formula.

Now suppose $L$ is unramified and $k=2n+1$.  
We continue with \eqref{eq:HhatSimplifiedUnramifiedOddk}.
Lemma \ref{lemma:ExtraInfoUnramifiedkodd} implies that we may assume
$\mathrm{Nm}(\ell_\xi) \equiv 0 \pmod{p^{n+1}}$, in which case $G_p'$ evaluates as $\epsilon p^3$.  Mercifully, in this case we have $v(\ell_\xi^4/3) = 2v(\Nm(\ell_\xi))-v(3)\geq 2n+2-v(3) \geq k$, so there is no need to treat $p=3$ separately.
Using similar Taylor expansions as for $k=2n$, and absorbing $\overline{\gamma}$ into $\epsilon$ we get
\begin{equation}
\label{eq:HhatSimplifiedUnramifiedEvenk_trivialpsi}
\widehat{H}(\psi_0) = \epsilon \frac{p^{5n+3}}{p^{2k}}
\sumstar_{\substack{s \shortmod{p^n}}}
e_{p^k}(s^{-1} \mathrm{Nm}(\ell_\xi)) =  \epsilon S(\mathrm{Nm}(\ell_\xi), 0;p^k).
\end{equation}

Now suppose $L$ is ramified.  Recall from Lemma \ref{lemma:thetarestrictedtoZp} that $\xi \vert_{\mz_p^{\times}}$ is the Legendre symbol, so $\xi(s+\ell_\xi) = (\frac{s}{p}) \xi(1+ s^{-1} \ell_\xi)$.  
Suppose $k=2n$.
Then
\eqref{eq:HhatSimplifiedRamifiedEvenk} reduces to the same formula as in the unramified case, up to the multiplication by $\epsilon$.  Hence $\widehat{H}(\psi_0) = \epsilon S(\mathrm{Nm}(\ell_\xi), 0 ;p^k)$.  

Finally, suppose $k=2n+1$.  
We return to \eqref{eq:HhatSimplifiedRamifiedOddk}.  
The summation condition $s \in T_{\psi,\xi}'$ is empty unless $\mathrm{Nm}(\ell_{\xi}) \equiv 0 \pmod{p^{n+1}}$, in which case $s$ is free.  Note $\mathrm{Nm}(\ell_{\xi})^2 \equiv 0 \pmod{p^{k}}$ since $2n+2 \geq k$.  Therefore, similarly as in the previous cases, we have 
$$
\widehat{H}(\psi_0) = \epsilon p^{\frac{k+1}{2}} \sumstar_{s \shortmod{p^n}} e_{p^k}(s^{-1} \mathrm{Nm}(\ell_{\xi})) = \epsilon S(\mathrm{Nm}(\ell_{\xi}), 0 ;p^k). \qedhere
$$
\end{proof}

\subsection{The small $k$ range, $k=c_0$}
\label{section:ksmallish}
Our next goal is to consider $k=c_0$.  There are a variety of cases: when $c_0 = 1$, we obtain character sums over a finite field, which we mainly treat in Section \ref{section:AG} below, 
while if $c_0 \geq 2$, then the Postnikov formula is effective.  Moreover, the character sum also exhibits different behavior depending on whether $(p,a_1 a_2 a_3) = 1$ or not.

Recall from \eqref{eq:supercuspidalTable} that  if $p$ is odd, $k=c_0$, and $L/\Q_p$ is ramified, then $H_p$ vanishes identically. Therefore, we assume that $L/\Q_p$ is the unique unramified quadratic extension for the remainder of this section. It may also be helpful to recall that $v(\mathrm{Nm}(\ell_\xi)) = 0$ and $(s^2 + \mathrm{Nm}(\ell_\xi), p) = 1$ in this situation, see \eqref{eq:valuationNormR} and Lemma \ref{lemma:traceRisZero}.
 
\begin{mylemma}
\label{lemma:HhatSupercuspidalPrimeTrivialPsi}
Suppose $p \neq 2$ and $k = c_0 =1$.  Suppose $\psi = \psi_0$ is trivial.  If $p|a_1 a_2 a_3$ then
\begin{equation*}
|\widehat{H}(\psi_0,a_1,a_2,a_3)| \leq    p^{-1} (a_1, p)(a_2, p)(a_3,p).
\end{equation*}   
\end{mylemma}
\begin{proof}
From \eqref{eq:HhatDecentFormulaSupercuspidalInert2}, we have that $\widehat{H}$ equals
(for arbitrary $k \geq c_0$)
\begin{equation*}
 \overline{\gamma}p^{-2k} \sumstar_{\substack{ x_1, x_2, x_3 \shortmod{p^{k}}}} 
 \thinspace
\sumstar_{\substack{t \in \mathcal{O}_L/(p^{k})}} 
\xi(t) e_{p^{k}}(-\mathrm{Tr}(t)+
 a_1 x_1 + a_{2} x_2 + a_{3} x_3 - \frac{x_1 x_2 x_3 a_1 a_2 a_3}{\mathrm{Nm}(t)} ).
\end{equation*}
Now set $k=1$.
If $p|a_1 a_2 a_3$, then $e_p(-x_1 x_2 x_3 a_1 a_3 a_3/\mathrm{Nm}(t)) = 1$.  The $t$-sum is the Gauss sum for $\xi$.  Each of the $x_i$-sums creates a Ramanujan sum.  Indeed,
\begin{equation}
\widehat{H}(\psi_0, a_1,a_2,a_3)
= \frac{\overline{\gamma}\tau_L(\xi)}{p^2} S(a_1, 0;p) S(a_2, 0;p) S(a_3, 0;p).
\end{equation}
Using $|\tau_L(\xi)| = p$ and $|S(a,0;p)| \leq (a,p)$ completes the proof.
\end{proof}

\begin{mylemma}
\label{lemma:HhatpsiTrivialInertkSmallest}
Suppose $p \neq 2$,  and $k = c_0 \geq 2$.
Suppose $\psi = \psi_0$ is trivial, 
and $p^k$ is large.
If $p|a_1 a_2 a_3$ 
then $\widehat{H}(\psi_0) = 0$ unless
$p^k |a_1 a_2 a_3$, in which case 
\begin{equation}
\label{eq:HhatpsiTrivialInertkSmallest}
\widehat{H}(\psi_0, a_1, a_2, a_3) =   
\overline{\gamma}p^{-2k} \tau_{L}(\xi) S(a_1, 0;p^k) S(a_2, 0;p^k) S(a_3, 0;p^k).
\end{equation}
\end{mylemma}
\begin{proof}
If $p^k | a_1 a_2 a_3$, then 
\eqref{eq:HhatpsiTrivialInertkSmallest} holds by the same argument as in the proof of Lemma \ref{lemma:HhatSupercuspidalPrimeTrivialPsi}.

Now we show $\widehat{H}(\psi_0) = 0$ unless $p^k | a_1 a_2 a_3$.  
Since $\psi$ is trivial, we may take $\ell_{\psi} = 0$.
Lemma \ref{lemma:HhatInitialEvaluation} implies that $\widehat{H}(\psi_0) = 0$ if
there are no solutions to the system \eqref{eq:PostnikovCongruenceSystemSupercuspidalInert}.  Note that for $i=1,2,3$, $E_i=0$ implies that $a_i x_i \equiv \frac{a_1 a_2 a_3 x_1 x_2 x_3}{\mathrm{Nm}(t_0)} \pmod{p^n}$, where $(x_1 x_2 x_3 t_0, p) = 1$.  
Hence if there exists $1 \leq j \leq n$ such that $p^j$ divides $a_1$ (say) then $p^j$ also divides $a_2$ and $a_3$.  This then implies $p^{\min(3j,n)}$ divides $a_1 a_2 a_3$.  Feeding this information back into $E_i=0$ implies that $p^{\min(3j,n)}$ divides each $a_i$.  Repeating this argument shows that $p^n$ divides each $a_i$, which in turn implies $a_1 a_2 a_3 \equiv 0 \pmod{p^k}$ since $k \leq 3n$.
\end{proof}

Recall (cf.\ \eqref{eq:Tprimedef2}) the set 
\begin{equation}\label{eq:Tprimedef3}
T_{\psi,\xi}' = \{ s \in (\Z/p^n\Z)^\times: s^2 \ell_{\psi} - s \mathrm{Nm}(\ell_\xi) + \mathrm{Nm}(\ell_\xi) \ell_{\psi} \equiv 0 \pmod{p^{n+1}}\},
\end{equation}
by which we mean that $s\in T_{\psi,\xi}' $ if and only if there exists a lift of $s$ modulo $p^{n+1}$ for which the indicated equation holds (modulo $p^{n+1}$).  Recall also the set $T_{\psi, \xi}$ from \eqref{eq:PostnikovCongruenceSystemReducedSupercuspidalInert}, which is defined similarly but subject to the congruence modulo $p^{n}$.
\begin{mylemma}
\label{lemma:HhatBoundSupercuspidalInertkequalsi0IntermediateStep}
Suppose $p \neq 2$,  $k = c_0 \geq 2$, and $p^k$ is large. 
Suppose that $a_1 = a_2 = a_3 =1$. 
\begin{enumerate}
\item For $k=2n$, we have
\begin{equation}
|\widehat{H}(\psi, 1, 1, 1)| \leq   
p^{k/2} 
\sum_{s \in T_{\psi,\xi}}  1.
\end{equation}
\item For $k=2n+1 \geq 3$, we have
\begin{equation}
\label{eq:HhattBoundSupercuspidalInertOddkSplit}
|\widehat{H}(\psi, 1, 1, 1)| \leq  
  p^{k/2}  
\Big(
\sum_{s \in T_{\psi,\xi}} 1 
 +
\sum_{s \in T_{\psi,\xi}'}  p^{ 1/2}
 \Big).
\end{equation}
Moreover, the second term in \eqref{eq:HhattBoundSupercuspidalInertOddkSplit} may be dropped if the discriminant of
the polynomial in $s$ appearing in \eqref{eq:PostnikovCongruenceSystemReducedSupercuspidalInert2} (i.e., $\mathrm{Nm}(\ell_\xi)^2 - 4\mathrm{Nm}(\ell_\xi) \ell_{\psi}^2$) is nonzero modulo $p$.
\end{enumerate}
\end{mylemma}

\begin{proof}
For $k$ even, the claimed bound follows from Lemma \ref{lemma:HhatInitialEvaluation2}, so we assume $k$ odd. Continuing from \eqref{eq:HhatSimplifiedUnramifiedOddk}, to bound the sums it suffices  to understand the Gauss sum $G_p' = G_p(\mathcal{Q}, \mathcal{L})$, to which we apply Lemma \ref{lemma:quadraticGaussSum}.  

The linear form $\mathcal{L}$ may be inferred from \eqref{eq:Gp'Def}. Choose a non-zero $\alpha \in \cO_L/(p)$ with $\Tr(\alpha)=0$ so that we may write $t_1 = u + v \alpha$ where $u,v \in \mf_p$.  We will use coordinates ${\bf x}:=(y_1, y_2, y_3, u, v) \in \mf_p^5$. Then
$$
\mathcal{L}(\mathbf{x}) = 
\frac{m}{s^2 - \ell_\xi^2} (
2u-y_1 - y_2 -y_3), 
$$
where $m$ was defined in \eqref{eq:Mdef}. Next, 
we have $\mathcal{Q}[{\bf x}] = \frac12  {\bf x} M_2 {\bf x}^{\intercal}$, where
\begin{equation}
M_2 = 
\begin{pmatrix}
\ell_{\psi} & - A & - A  & 2  A & 0 \\
-A &  \ell_{\psi} & - A  & 2  A & 0 \\
-A &  - A &  \ell_{\psi}  & 2 A & 0 \\
2A & 2 A & 2 A & -6A -2  \ell_{\psi} & -2 \ell_\xi \alpha \\
0 & 0 & 0  & -2 \ell_\xi \alpha  & 2\Nm(\alpha) ( A+  \ell_{\psi})
\end{pmatrix}.
\end{equation}
The authors used a computer algebra package (Mathematica) to aid in the following calculations.  The determinant of $M_2$ is
\begin{equation}\label{eq:M2detcalc}
-4\Nm(\alpha) (A + \ell_{\psi})^2 \cdot q_3, 
\qquad \text{where}
\qquad 
q_3 = \ell_\psi(A+\ell_\psi)^2 + \ell_\xi^2(2A-\ell_\psi).
\end{equation}
Working in $\mf_p$, with $p \neq 2$, recall that $\mathrm{Tr}(\ell_\xi) \equiv 0 \pmod{p}$, so $\ell_\xi^2 = - \mathrm{Nm}(\ell_\xi) \in \mf_p$.
As mentioned earlier, $\ell_{\psi} + A \equiv s \not \equiv 0 \pmod{p}$.  Hence $\det(M_2) = 0$ if and only if $q_3 = 0$.
In $\mf_p$ we have $A = s - \ell_{\psi}$  and $\ell_\psi s^2 +\ell_\xi^2( s - \ell_\psi) =0$ (from \eqref{eq:HhatSimplifiedUnramifiedEvenk}), so
\begin{gather*}
q_3 =\ell_\psi s^2 + \ell_\xi^2(2s-3\ell_\psi) = -\ell_\xi^2 (s-\ell_\psi) + \ell_\xi^2(2s-3\ell_\psi)  
=  \ell_\xi^2(s-  2 \ell_{\psi}).  
\end{gather*}
Hence $q_3 = 0$ implies that $s= 2 \ell_{\psi}$, since $v(\ell_\xi^2) = 0$.
Note that this means that $A = s-\ell_{\psi} = \ell_{\psi}$.  With this substitution, we have
\begin{equation*}
q_3 
= \ell_{\psi} (4 \ell_{\psi}^2 + \ell_\xi^2).
\end{equation*}
Therefore, $4 \ell_{\psi}^2 = -\ell_\xi^2$ (and $A= \ell_{\psi}$) are necessary conditions for $\det(M_2) = 0$ (constrained by the $E_i$ congruences).

Next we argue that when $\det(M_2) = 0$, then $M_2$ has rank $4$.  To see this, we can simplify $M_2$ using $A = \ell_{\psi}$, giving that $M_2 = \ell_{\psi} \cdot M_2'$, where
\begin{equation}
M_2' = 
\begin{pmatrix}
1 & - 1 & - 1  & 2   & 0 \\
-1 &  1 & - 1  & 2  & 0 \\
-1 &  - 1 &  1  & 2   & 0 \\
2  & 2   & 2   & -8 & \frac{-2 \ell_\xi \alpha}{\ell_{\psi}} \\
0 & 0 & 0  & \frac{-2 \ell_\xi \alpha}{\ell_{\psi}}  & 4 \Nm(\alpha) 
\end{pmatrix}.
\end{equation}
One can manually perform row reductions on $M_2'$ to see it is row-equivalent to 
\begin{equation}
\begin{pmatrix}
1 & 0 & 0  & -2   & 0 \\
0 &  1 & 0  & -2  & 0 \\
0 &  0 &  1  & -2   & 0 \\
0  & 0   & 0   & 4 & \frac{-2 \ell_\xi \alpha}{\ell_{\psi}} \\
0 & 0 & 0  & \frac{-2 \ell_\xi \alpha}{\ell_{\psi}}  & 4 \Nm(\alpha) 
\end{pmatrix}.
\end{equation}
As a safety check, we observe that the fourth and fifth rows are linearly dependent under the assumption $4 \ell_{\psi}^2 = - \ell_\xi^2$.  The rank is thus $4$.

If the rank of $\mathcal{Q}$ is 5, then its radical is trivial and $\mathcal{L} \vert_V=0$ vacuously. On the other hand, if the rank of $\mathcal{Q}$ is 4, then its radical $V$ is the line through $(1,1,1,1/2,\frac{\ell_\xi \alpha}{4\Nm(\alpha)\ell_\psi})^\intercal$.
Converting back to $(y_1,y_2,y_3,t_1) \in \F_p^3 \times \F_{p^2}$ coordinates, one can check that $V$ does not depend on the choice of $\alpha$. Since the linear form $2u-y_1-y_2-y_3$ does not vanish on this vector, one sees that $\mathcal{L} \vert_V=0$ if and only if $m \equiv 0 \pmod p$.

Finally, for the last sentence of the lemma, the previous work derived that $\mathcal{Q}$ having rank $\leq 4$ implies that $4 \ell_{\psi}^2 + \ell_\xi^2 = 0 \in \mf_p$, as claimed.
\end{proof}

The following result refines Lemma \ref{lemma:HhatBoundSupercuspidalInertkequalsi0IntermediateStep}.
For this, define
\begin{equation}
\label{eq:rhoDef}
\rho(\Delta, p^m)
= \# \{ x \shortmod{p^m} : x^2 \equiv \Delta \pmod{p^m} \}.
\end{equation}
This notation agrees with \cite[Section 3.2]{PetrowYoungCoset}.
\begin{mylemma}
\label{lemma:HhatBoundInertRhoBound}
Suppose $p$ is odd, $k = c_0 \geq 2$, and $p^k$ is large. 
Suppose that $a_1 = a_2 = a_3 =1$.  Let 
\begin{equation}
\label{eq:DeltaDef}
\Delta = -\frac{\ell_{\psi}^2}{\mathrm{Nm}(\ell_\xi)} + \frac{1}{4} . 
\end{equation}
\begin{enumerate}
\item For $k=2n$, we have
\begin{equation}
|\widehat{H}(\psi)| \leq  
 p^{k/2} \rho(\Delta, p^n).
\end{equation}
\item For $k=2n+1 \geq 3$, we have
\begin{equation}
|\widehat{H}(\psi)| \leq  
  p^{k/2}
\Big(
   \rho(\Delta, p^n)
 +
  p^{1/2} \rho(p^{-2} \Delta, p^{n-1}) \cdot \delta(p^2|\Delta)
 \Big).
\end{equation}
\end{enumerate}
\end{mylemma}
\begin{proof}
This builds on Lemma \ref{lemma:HhatBoundSupercuspidalInertkequalsi0IntermediateStep}, which estimated $|\widehat{H}(\psi)|$ in terms of the number of solutions to the congruence $s^2 \ell_{\psi} - s \mathrm{Nm}(\ell_\xi) + \mathrm{Nm}(\ell_\xi) \ell_{\psi} \equiv 0 \pmod{p^m}$, for $m \geq 1$ and with $(s,p) = 1$.
Recall that $v(\Nm(\ell_\xi))=0$ since $k=c_0$ and that $(s^2+\mathrm{Nm}(\ell_{\xi}), p) = 1$.  Hence the existence of a solution to this congruence requires $v(\ell_{\psi}) = 0$.
Changing variables $s \rightarrow \mathrm{Nm}(\ell_\xi) s/\ell_{\psi}$ shows this congruence is equivalent to
\begin{equation}
\label{eq:scongruenceMonic}
s^2 - s  +\frac{\ell_{\psi}^2}{\mathrm{Nm}(\ell_\xi)} \equiv 0 \pmod{p^m}.
\end{equation} 
Completing the square shows the number of solutions to \eqref{eq:scongruenceMonic} equals $\rho(\Delta, p^m)$, with $\Delta$ defined by \eqref{eq:DeltaDef}.  
If $k$ is even, then applying the above observations with $m=n$ to \eqref{eq:HhatSimplifiedUnramifiedEvenk} finishes the proof.

The proof for $k$ odd is similar, but with one additional piece of information.  The last sentence of Lemma \ref{lemma:HhatBoundSupercuspidalInertkequalsi0IntermediateStep} gives that the second term in  \eqref{eq:HhattBoundSupercuspidalInertOddkSplit} may be dropped unless $p|\Delta$.  It is easy to check 
that if $p | \Delta$ then \eqref{eq:PostnikovCongruenceSystemReducedSupercuspidalInert2} has no solutions if $p^2 \nmid \Delta$. 
Then under the assumption $p^2|\Delta$, factoring out $p$ shows that the number of solutions to \eqref{eq:PostnikovCongruenceSystemReducedSupercuspidalInert2} is
$\rho(p^{-2} \Delta, p^{n-1})$.  This completes the proof for $k$ odd.
\end{proof}

\begin{mylemma}
\label{lemma:HhatSupercuspidalInertDegenerate}
Suppose $p \neq 2$, $k = c_0 \geq 2$, and $p^k$ is large.
Assume $\psi$ is nontrivial.
Suppose for some $i=1,2,3$ and some $1 \leq j \leq k$ that
$v(a_i)=j$.  Then $\widehat{H}(\psi, a_1, a_2, a_3) = 0$ unless $C(\psi) =p^{k-j}$ and $v(a_i)=j$ for all $i=1,2,3$.
\end{mylemma}

\begin{proof}
Recall that $\widehat{H}$ is given up to a scalar by \eqref{eq:HhatDecentFormulaSupercuspidalInert2}.
Suppose without loss of generality that $v(a_1)=j$, where $1 \leq j \leq k$.  The sum in $x_1$ is a Gauss sum and $\psi$ is non-trivial, so 
 the inner sum over $x_1$ is nonzero
if and only if 
\begin{equation}
\label{eq:lemma:HhatSupercuspidalInertDegenerate}
k-v\left(a_1 \left(1-\frac{x_2x_3a_2a_3}{\Nm(t)}\right)\right) = c(\overline{\psi}).
\end{equation}
Since $v \left(1-\frac{x_2x_3a_2a_3}{\Nm(t)}\right)\geq 0$, if $\widehat{H}(\psi, a_1, a_2, a_3) \neq  0$, then $c(\psi)\leq k-j$. The $x_2$-sum is also a Gauss sum, so that similarly $\widehat{H}(\psi, a_1, a_2, a_3) = 0$ unless $k-v\left(a_2 \left(1-\frac{x_1x_3a_1a_3}{\Nm(t)}\right)\right) = c(\overline{\psi})$. Since $v(a_1)>0$, we have $v\left(1-\frac{x_1x_3a_1a_3}{\Nm(t)}\right)=0$. Thus, $k-v(a_2)= c(\psi) \leq k-j$, which is to say $j\leq v(a_2)$. Since $j>0$, we now know by \eqref{eq:lemma:HhatSupercuspidalInertDegenerate} that $c(\psi) = k-v(a_1)$ as well. Thus, $v(a_1)=v(a_2)=j$ and $c(\psi)= k-j$. By symmetry, we also have $j = v(a_3)$. 
\end{proof}

\begin{mylemma}
\label{lemma:HhatBoundSupercuspidalInertDegenerate}
Suppose $p$ is odd, $k = c_0 \geq 2$, $p^k$ is large, and $\psi$ is nontrivial.
Suppose that there exists $1 \leq j \leq k-1$ such that $C(\psi)=p^{k-j}$ and $v(a_i)=j$ for $i=1,2,3$.
Then
\begin{equation*}
|\widehat{H}(\psi, a_1, a_2, a_3)| \leq  
  p^{k/2} p^{3j/2}.
\end{equation*}
\end{mylemma}
\begin{proof}
We work with
\eqref{eq:HhatDecentFormulaSupercuspidalInert2}.  The sum over $x_i$ has period $p^{k-j}$, so it is the same sum repeated $p^j$ times.  Therefore, $\gamma p^{2k} \widehat{H}(\psi,p^j,p^j,p^j)$ equals
\begin{equation}\label{eq:lemma:HhatBoundSupercuspidalInertDegenerate1}
p^{3j} 
\sum_{\substack{ x_1, x_2, x_3 \shortmod{p^{k-j}}}} 
\sum_{\substack{t \in \mathcal{O}_L/(p^k)}} 
 \psi\Big(\frac{\mathrm{Nm}(t)}{x_1 x_2 x_3}\Big)
\xi(t) e_{p^k}(-\mathrm{Tr}(t)) e_{p^{k-j}}(
 x_1 +  x_2 + x_3
  - \frac{x_1 x_2 x_3  p^{2j}}{\mathrm{Nm}(t)} ).
\end{equation}

We follow the same steps as in the proof of Lemma \ref{lemma:HhatInitialEvaluation} but applied to \eqref{eq:lemma:HhatBoundSupercuspidalInertDegenerate1}. We apply the Postnikov formula, writing $t = t_0(1+p^n t_1)$ where $n = \lfloor k/2 \rfloor$, and replace $x_i$ by $x_i(1+p^{\lfloor (k-j)/2 \rfloor} y_i)$, for $i=1,2,3$.  
Set $\alpha = \lfloor \frac{k-j}{2} \rfloor$ and note that $v(\ell_\psi)=j$.  We have 
\begin{align}
\psi(\mathrm{Nm}(1 + p^n t_1)) &  = \psi_L(1 + p^n t_1)= e_{p^{k-n}}( \ell_{\psi} \mathrm{Tr}(t_1)), \label{al:HhatBoundSupercuspidalInertDegenerate2} 
\\
\xi(1 + p^n t_1) & = e_{p^{k-n}}(\Tr(\ell_\xi t_1)) e_{p^{k-2n}}(-\tfrac12 \Tr(\ell_\xi t_1^2)), \label{al:HhatBoundSupercuspidalInertDegenerate3}
\\ 
\overline{\psi}(1+p^\alpha y_i) & = e_{p^{k-j-\alpha}}(- p^{-j}\ell_{\psi} y_i) e_{p^{k-j-2\alpha}} (\tfrac12 p^{-j}\ell_{\psi} y_i^2), \label{eq:psiPostnikovSemiDegenerateCase}
\\
e_{p^k}(-\mathrm{Tr}(t)) & = e_{p^k}(-\Tr(t_0)) e_{p^{k-n}}(-\Tr(t_0t_1)), 
\\
 e_{p^{k-j}}( x_1 +  x_2 + x_3) & \to e_{p^{k-j}}( x_1 +  x_2 + x_3) e_{p^{k-j-\alpha}} ( x_1y_1+x_2y_2 + x_3y_3.)\label{al:HhatBoundSupercuspidalInertDegenerate4}
\end{align}
In \eqref{al:HhatBoundSupercuspidalInertDegenerate2}, we used the fact that $2n \geq k-j$ to dispense with the quadratic and higher order terms. Lines \eqref{al:HhatBoundSupercuspidalInertDegenerate3} to \eqref{al:HhatBoundSupercuspidalInertDegenerate4} are essentially the same as in the proof of Lemma \ref{lemma:HhatInitialEvaluation}, see e.g.\ \eqref{eq:thetaPostnikovInert}. 

Let $A = \frac{x_1x_2x_3}{\Nm(t_0)}$, as in \eqref{eq:Adef}. Following \eqref{eq:AformulaTaylorExpansion}, the cubic term is replaced by $e_{p^{k-j}} \left(-p^{2j}\frac{x_1 x_2 x_3}{\mathrm{Nm}(t_0)}\right)$ times 
\begin{multline}
e_{p^{k-j}} \Big( -p^{2j} A
\Big[\frac{(1+p^{\alpha} y_1)(1+p^{\alpha} y_2)(1+p^{\alpha} y_3)}{(1+p^n t_1)(1+ p^n \overline{t_1})} -1 \Big] \Big)
\\
= 
e_{p^{k-j-\alpha}}\Big(-p^{2j} A (y_1 + y_2 + y_3) \Big)
e_{p^{k-n}}\Big(p^{3j} A \mathrm{Tr}(t_1) \Big).
\end{multline}
Here only the linear terms survived, since $2j + 2 \alpha \geq k-j$, $2j + \alpha +n \geq k-j$, and $3j + 2n \geq k$.

First,  suppose that $k$ and $k-j$ are both even so that we may drop the quadratic terms in \eqref{al:HhatBoundSupercuspidalInertDegenerate3} and \eqref{eq:psiPostnikovSemiDegenerateCase}. By additive orthogonality of characters, the sum $\widehat{H}(\psi, p^j, p^j, p^j)$ vanishes unless \begin{equation}
\label{eq:PostnikovCongruenceSystemSupercuspidalInertDegenerate}
\begin{split}
&-p^{-j}\ell_{\psi} +  x_i - A p^{2j}  = 0 \in \Z/p^\alpha\Z, \qquad i=1,2,3, \text{ and }\\
&-t_0 + \ell_\xi + \ell_{\psi}
+ p^{3j} A  = 0 \in \mathcal{O}_L/p^n\cO_L,
\end{split}
\end{equation} cf.\ \eqref{eq:PostnikovCongruenceSystemSupercuspidalInert}. Then 
\begin{equation*}
  \widehat{H}(\psi, p^j, p^j, p^j) = \overline{\gamma}\frac{p^{3j}}{p^{2k}} p^{3\alpha + 2n}
\sumstar_{\substack{x_1, x_2, x_3 \shortmod{p^\alpha} \\ t_0 \in (\mathcal{O}_L/(p^n))^{\times} \\ \eqref{eq:PostnikovCongruenceSystemSupercuspidalInertDegenerate} \text{ holds} }}
(\dots).
\end{equation*}
These congruences determine the $x_i$ modulo $p^{\min(\alpha, 2j)}$, and $t_0$ modulo $p^{\min(n,3j)}\cO_L$, which by Hensel's lemma, lift to a unique solution modulo $p^\alpha$ resp.\ modulo $p^n\cO_L$.  Hence in this case we obtain the claimed bound
\begin{equation*}
|\widehat{H}(\psi, p^j, p^j, p^j)| \leq  p^{3j-2k} p^{3 \frac{k-j}{2}} p^k 
=   p^{k/2} p^{3j/2}.
\end{equation*}

To treat the cases where the exponents may be odd, we point out that the relevant quadratic form from \eqref{al:HhatBoundSupercuspidalInertDegenerate3} and \eqref{eq:psiPostnikovSemiDegenerateCase}
is already diagonalized and is  nonsingular since $v(\mathrm{Nm}(\ell_\xi)) = v(p^{-j}\ell_{\psi}) = 0$.  Hence the quadratic exponential sum always has square-root cancellation.
The end result is the claimed bound.
\end{proof}

We close this section with two elementary preparatory lemmas in the case $k=c_0=1$. This case is the only one that gets no information from the Postnikov formula. 
\begin{mylemma}
\label{lemma:HhatSupercuspidalPrime}
Suppose $p \neq 2$, $k = c_0 =1$ and $\psi$ has conductor $p$.  If $p|a_1 a_2 a_3$, then $\widehat{H}(\psi, a_1, a_2, a_3) = 0$.   If $(p, a_1 a_2 a_3) = 1$, then $\widehat{H}(\psi, a_1, a_2, a_3)$ is given by \eqref{eq:HhatPrimeSupercuspidal} below.
\end{mylemma}

\begin{proof}
If $p|a_1$ (say), then the inner sum over $x_1$ in \eqref{eq:HhatDecentFormulaSupercuspidalInert2} vanishes by orthogonality of characters.  This gives the first claim, by symmetry.

Now suppose $a_1 = a_2 =a_3 = 1$.  The inner sum over $x_1$ in \eqref{eq:HhatDecentFormulaSupercuspidalInert2} evaluates as a Gauss sum, giving that $\widehat{H}(\psi)$ equals
\begin{equation}
\label{eq:HhatPrimeSupercuspidal}
\frac{\overline{\gamma}\tau(\overline{\psi})}{p^{2}}
\sum_{\substack{ x_2, x_3 \shortmod{p}}} 
\sum_{\substack{t \in \mathcal{O}_L/(p)}} 
 \psi\Big(\frac{\mathrm{Nm}(t)}{ x_2 x_3}\Big)
 \psi\Big(1 - \frac{x_2 x_3}{\mathrm{Nm}(t)}\Big)
\xi(t) e_{p}(-\mathrm{Tr}(t) 
   +   x_2 +   x_3 ). 
\end{equation}
\end{proof}
\begin{mylemma}\label{lemma:HhatSupercuspidalPrime_triv}
Suppose $p \neq 2$, $k = c_0 =1$, and $a_1 = a_2 = a_3 = 1$. For $\psi=\psi_0$ trivial we have 
\begin{equation}\label{eqn:kc01_triv}\gamma p^2 \widehat{H}(\psi_0,1,1,1) =p \sumstar_{y \shortmod{p}} \Kl_{\rm ns} (y) S(y,1;p)-  \tau_L(\xi) ,\end{equation}
where $$\Kl_{\rm ns} (y):= \sum_{\substack{t \in \cO_L/(p) \\ \Nm(t)=y}} \xi(t) e_p(-\Tr(t)), \quad \text{ and } \quad \tau_L(\xi) = \sumstar_{ t \in \cO_L/(p)}\xi(t) e_p(-\mathrm{Tr}(t)).$$ 
\end{mylemma}
\begin{proof}
 The inner sum over $x_1$ in \eqref{eq:HhatDecentFormulaSupercuspidalInert2} evaluates as a Ramanujan sum, giving
\begin{multline*}
\gamma p^2 \widehat{H}(\psi_0) = \sumstar_{x_2,x_3\shortmod{p}} \sum_{t \in \cO_L/(p)} S\left(1-\frac{x_2x_3}{\Nm(t)}, 0;p\right) \xi(t)  e_p(-\mathrm{Tr}(t) + x_2 + x_3) \\
 =  p\sumstar_{x_3 \shortmod{p}} \sum_{t \in \cO_L/(p)} \xi(t) e_p(-\mathrm{Tr}(t) + x_3 + \frac{\mathrm{Nm}(t)}{x_3}) -   \sumstar_{x_2,x_3 \shortmod{p}}  \sum_{t \in \cO_L/(p)}  \xi(t)  e_p(-\mathrm{Tr}(t) + x_2 + x_3),
\end{multline*}from which \eqref{eqn:kc01_triv} follows. 
\end{proof}

\subsection{Estimates from algebraic geometry.}
\label{section:AG}
Next we use the theory of $\ell$-adic trace functions to bound the sum $\widehat{H}(\psi)$ as it appears  in Lemmas \ref{lemma:HhatSupercuspidalPrime} and \ref{lemma:HhatSupercuspidalPrime_triv}. For background, see \cite{SGA4.5, Katz1980, Katz1988, FKMS}. 

In fact, we work in somewhat more generality than is strictly required. Let $k$ be a finite field of characteristic $p$ and cardinality $q$, and $K$ be a quadratic extension of $k$. Let $\addchar$ be a fixed non-trivial additive character of $k$. Let $\psi, \xi$ be non-trivial multiplicative characters of $k^\times,K^\times$, respectively, extended by $0$ to $\P^1(k), \P^1(K)$. Let 
\begin{equation}
g(\xi,\psi) = \sum_{x_2,x_3 \in k} \sum_{t \in K}  \psi\Big(\frac{\mathrm{Nm}(t)}{ x_2 x_3}-1\Big)
\xi(t) \addchar(-\mathrm{Tr}(t) 
   +   x_2 +   x_3 ). 
\end{equation}
 The main goal of this section is to prove the following.
\begin{myprop}\label{AGprop}
If $\xi$ is non-trivial, then 
$g(\xi,\psi) \ll q^2$ for all non-trivial $\psi$, with an absolute and effective implied constant. 
\end{myprop}
Proposition \ref{AGprop} implies that $H(\psi)\ll p^{1/2}$ in the case that $p$ is odd, $k=c_0=1$, $c(\psi)=1$ and $p \nmid a_1a_2a_3$ cf.\ Lemma \ref{lemma:HhatSupercuspidalPrime}.

\subsubsection{Notation and tools} 
Refreshing the notation, let $\addchar$ be a fixed non-trivial $\overline{\Q}_\ell^\times$-valued additive character of $k$, and define the additive character $\overline{\addchar}_K$ of $K$ by $\overline{\addchar}_K(x) =\addchar(-\Tr(x)).$ Fix an embedding $\overline{\Q}_\ell \hookrightarrow \C$ so that our trace functions may take complex values. Finally, we mention that in this paper by a ``sheaf on $V$'' we will always mean a constructible $\overline{\Q}_\ell$-sheaf on $V\subseteq \P^1$.

Fouvry, Kowalski and Michel have introduced a notion of conductor for sheaves on dense open sets of $\P^1_k$. 
In this paper, we use the definition \cite[Def.\ 4.3]{FKMS}, namely: 
\begin{mydefi}
For a sheaf $\cF$ on $V \subseteq \P_k^1$, lisse on $U\subseteq V$, the positive integer 
\begin{equation*}
{\rm c}(\cF): = {\rm Rk}(\cF) +  |\P^1(\overline{k}) - U(\overline{k})| + \sum_{x \in \P^1(\overline{k})} \Swan_x(\cF),
\end{equation*} is called the conductor of $\cF$. 
\end{mydefi}
Remark. More natural and advanced notions of conductor of sheaves on general varieties are now available, see \cite{FFKS_conductor}. 

Our main tool to prove Proposition \ref{AGprop} is the Deligne-Laumon geometric Fourier transform. We give a few preliminary definitions, following closely \cite[\S8.2]{Katz1988}. Let $\mathcal{L}_{\addchar}$ be the Artin-Schreier sheaf on $\A^1_k$ associated to $\addchar$, which we may pull back to $\A_{\overline{k}}^1$ for $\overline{k}$ an algebraic closure of $k$. For any $a \in k$ also define $\mathcal{L}_{\addchar(ax)} = (x\mapsto ax)^*\mathcal{L}_{\addchar}$, which is another Artin-Schreier sheaf on $\A^1_{\overline{k}}$. If $a=0$, the sheaf $\mathcal{L}_{\addchar(ax)}$ is just the constant sheaf. 
\begin{mydefi}
A sheaf $\cF$ on $\A^1_{\overline{k}}$ is called \emph{elementary} if it satisfies the following two conditions
\begin{enumerate}
\item If $U$ is a non-empty open set on which $\cF$ is lisse and $j: U \hookrightarrow \A^1_{\overline{k}}$ is the open immersion, then the map $\cF \to j_*j^* \cF$ is injective. Equivalently, for all $a \in k$, the cohomology group $H_c^0(\A^1_{\overline{k}}, \cF \otimes \cL_{\addchar(ax)})=0$, i.e.\ $\cF$ has no non-zero punctual sections. 
\item For all $a \in k$, the cohomology group $H_c^2(\A^1_{\overline{k}}, \cF \otimes \cL_{\addchar(ax)})=0$.
\end{enumerate}
\end{mydefi}
A sheaf $\cF$ on $\A_k^1$ is said to be elementary 
if for some (equiv.\ any) algebraic closure $\overline{k}$ of $k$, the sheaf $\cF$ pulled back to $\A_{\overline{k}}^1$
is elementary.

Given an elementary $\cF$ on $\A_k^1$, let $\widehat{\cF}$ be the geometric Fourier transform of $\cF$ with respect to $\addchar$, as defined e.g.\ in \cite[8.2.3 Def.]{Katz1988}. 
\begin{mytheo}[Deligne, Katz, Laumon]\label{ThmGeomFT_MT} The Deligne-Laumon geometric Fourier transform enjoys the following properties. 
\begin{enumerate}
\item If $\cF$ is elementary, then $\widehat{\cF}$ is elementary.
\item If $\cF$ is elementary, then we have have a canonical isomorphism 
$$ \widehat{\widehat{\cF}} \simeq \cF(-1),$$ where the outer transform is defined with respect to $\overline{\addchar}$. 
\item If $\cF$ is elementary and pure of weight $w$, then $\widehat{\cF}$ is pure of weight $w+1$. 
\item If $\cF$ is elementary, geometrically irreducible, and pure of weight $0$, then $\widehat{\cF}$ is geometrically irreducible. 
\end{enumerate}
\end{mytheo}
\begin{proof}
\begin{enumerate}
\item See e.g.\ \cite[8.2.5 Main Thm.(1)]{Katz1988}.
\item See e.g.\ \cite[8.2.5 Main Thm.(1)]{Katz1988}.
\item This is the  main theorem of the geometric Fourier transform, due to Katz and Laumon, see e.g.\ \cite[Thm.\ 7.3.8(5)]{Katz1990}.
\item The geometric irreducibility of $\widehat{\cF}$ follows from Katz's diophantine irreducibility criterion \cite[Lem.\ 7.0.3]{KatzRigidLocalSystems} and the Plancherel formula, as observed by Fouvry, Kowalski, and Michel, see \cite[Prop.\ 6.8]{FKMS}. 
\end{enumerate}
\end{proof}

Lastly, we recall a sheaf on $\G_{m,k}$ associated to a multiplicative character $\xi$ on $K^\times$. 
\begin{mylemma}[Katz]\label{sheafKlnsproperties}
There exists a lisse sheaf $\cKl_{\rm ns} = \cKl(K,\overline{\addchar}_K,\xi)$ on $\G_{m,k}$ with the following properties. The sheaf $\cKl_{\rm ns}$
\begin{enumerate}
\item  is of rank 2, pure of weight 1, with trace function $t_{\cKl_{\rm ns}}$ for $y \in k^\times$ given by $$t_{\cKl_{\rm ns}}(y) = - \sum_{\substack{t\in K \\ \Nm(t)=y}} \xi(t) \overline{\addchar}_K(t),$$ 
\item  is tamely ramified at 0, totally ramified at $\infty$, and has $\Swan_\infty (\cKl_{\rm ns})) = 1$, 
\item is geometrically irreducible, and 
\item has ${\rm c}(\cKl_{\rm ns})=5$. 
\end{enumerate}
\end{mylemma}
\begin{proof}
\begin{enumerate}
\item The sheaf $\cKl_{\rm ns} = \cKl(K,\overline{\addchar}_K,\xi)$ is constructed and these basic properties are proved in \cite[8.8.5 Thm.]{Katz1988}.
\item This is \cite[\S 8.8.6]{Katz1988}.
\item The continuous representation $V = \cKl_{{\rm ns},\overline{\eta}}$ of the \'etale fundamental group $\pi_1(\G_m, \overline{\eta})$ restricts to a continuous representation of the inertia subgroup $I_\infty$. The hypotheses of \cite[1.11 Lem.(2)]{Katz1988} are satisfied thanks to point (2) of the lemma, and so $V$ is absolutely irreducible as an $I_\infty$-representation, \`a fortiori $\cKl_{\rm ns}$ is geometrically irreducible. 
\item Direct from the definition of the conductor and points (1), (2). 
\end{enumerate}
\end{proof}

\subsubsection{Proof of Proposition \ref{AGprop} when $\psi$ is non-trivial}
We re-arrange $g(\xi,\psi)$ as
$$ -\sum_{y} \sum_{x_1,x_2} t_{\cKl_{\rm ns}}(y) \psi(\frac{y}{x_1x_2}-1) \addchar(x_1+x_2) = -\sum_y  t_{\cKl_{\rm ns}}(y) \sum_{x_1} \addchar(x_1) \sum_{x_2} \psi(\frac{y}{x_1x_2}-1) \addchar(x_2),$$
where each of $y,x_1,x_2$ run over $k^\times$. 
Changing variables $x_2 \rightarrow yx_2/x_1$ we get that
$$g(\xi,\psi)=- \sum_y  t_{\cKl_{\rm ns}}(y) \sum_{x_1} \addchar(x_1) \sum_{x_2} \psi(\frac{1}{x_2}-1) \addchar(\frac{y}{x_1}x_2).$$ 
Let $\gamma: \A^1_k \to \P^1_k$ be the fractional linear transformation defined by $\gamma(X) = \frac{1}{X}-1$ and let $\cK_\psi$ be the Kummer sheaf corresponding to $\psi$ (non-trivial). Then the sheaf $\gamma^{*}K_\psi$ on $\A^1_k$ is  lisse on $U=\A^1_{k} \smallsetminus \{0,1\}$ of rank 1, pure of weight 0, and geometrically irreducible with trace function $t_{\gamma^* \cK_\psi}(x) =\psi(\frac{1}{x}-1)$. Since $\gamma^{*}K_\psi$ is rank 1 and not isomorphic to any Artin-Schreier sheaf, it is a Fourier sheaf, \`a fortiori elementary. Precisely, $\gamma^{*}K_\psi$ is middle-extension, rank 1, and ramified at 1, so that \cite[8.3.1 Lem.(2), 8.2.1.3 Lem.(1)]{Katz1988} applies.    
\begin{mylemma}\label{sheafGproperties}
For $\psi$ non-trivial, the sheaf $\cG:= \widehat{\gamma^* \cK_\psi}$ on $\A_k^1$ 
\begin{enumerate}
\item is elementary, pure of weight 1, and geometrically irreducible,
\item is lisse on $\G_{m,k}$, ramified at $0$, of rank $2$,
\item has $\Swan_0 (\cG)=0$ and $\Swan_\infty(\cG)=1$ with $\infty$-breaks $0$ and $1$, and
\item has trace function $$t_{\cG}(u) =- \sum_{x \in k \smallsetminus \{0,1\}} \psi(x^{-1}-1) \addchar(ux)$$ for all $u \in k$.
\end{enumerate}
\end{mylemma}
\begin{proof}
\begin{enumerate}
\item These facts follow directly from Theorem \ref{ThmGeomFT_MT}(1), (3), (4). 
\item Since $\gamma^* \cK_\psi$ is tamely ramified (indeed, lisse) at $\infty$, the sheaf $\cG$ is lisse on $\G_{m, k}$ and is ramified at $0$, see \cite[8.5.8 Cor.]{Katz1988}.  The rank of $\cG$ may be computed using \cite[Lem.\ 7.3.9(2)]{Katz1990}, since $\gamma^*\cK_{\psi}$ is lisse at $\infty$, everywhere tame, rank 1, and ramified at $0,1$. 
\item It follows from Laumon's local geometric Fourier transform \cite{LaumonFT} that $$\Swan_0( \cG) \leq \Swan_\infty(\gamma^* \cK_\psi) = 0,$$ see e.g.\ \cite[Thm.\ 7.5.4(5)]{Katz1990}. For the second part, by Fourier inversion Theorem \ref{ThmGeomFT_MT}(2), we may apply \cite[8.5.8 Cor.]{Katz1988} ``backwards'': since $\gamma^*\cK_\psi$ is not lisse on $\G_m$, the sheaf $\cG$ has a break at 1, and since $ \gamma^*\cK_\psi$ is not lisse at 0, the sheaf $\cG$ has a break $<1$. As $\cG$ has generic rank $2$ and $\Swan_\infty (\cG) \in \Z_{\geq 0}$, we must have that $0,1$ are the only two $\infty$-breaks of $\cG$, so $\Swan_\infty (\cG)=1$. 
\item The formula for the trace function follows from proper base change and the Grothendieck-Lefshetz trace formula \cite[Arcata IV Thm.\ 5.4, Rapport Thm.\ 3.2]{SGA4.5}, since by (1) $\cG$ is elementary. 
\end{enumerate}
\end{proof}
Remark: The sheaf $\cG$ is not a convolution sheaf, as it is not totally wildly ramified at $\infty$. Therefore we cannot directly apply \cite[5.1 Convolution Thm.]{Katz1988}, which if it were permissible would quickly finish the proof of Proposition \ref{AGprop}. 

By Lemma \ref{sheafGproperties}(4), we have
$$g(\xi,\psi)= \sum_y  t_{\cKl_{\rm ns}}(y) \sum_{x} \addchar(x)t_{\cG}(y/x).$$ 
Let $i: \G_{m,k} \to \A^1_{k}$ be the fractional linear transformation defined by $i(X)=1/X$. Let $j: \G_{m,k} \hookrightarrow \A^1_{k}$ be the inclusion map, and consider the extension by $0$ sheaf $j_!i^*\cG$ on $\A^1_{k}$.
\begin{mylemma}
The sheaf  $j_!i^*\cG$ on $\A^1_{k}$ is elementary. 
\end{mylemma}
\begin{proof}
It is clear from basic properties of the maps $j_*$, $j_!$, $j^*$ \cite[\S 4.4 Relations entre les foncteurs]{Katz1980} that the map $j_!i^*\cG \hookrightarrow j_*j^*j_!i^*\cG$ is injective, i.e.\ $j_!i^*\cG$ has no non-zero punctual sections.

For the second condition, by excision \cite[Sommes Trig.\ (2.5.1)*]{SGA4.5} the sequence 
$$ H_c^2(\G_{m,{\overline{k}}}, j_!i^*\cG \otimes \mathcal{L}_{\addchar(ax)}) \to H_c^2(\A_{\overline{k}}^1, j_!i^*\cG \otimes \mathcal{L}_{\addchar(ax)}) \to H_c^2(\{0\}, j_!i^*\cG \otimes \mathcal{L}_{\addchar(ax)})$$
is exact for any $a \in k$. The right-hand group vanishes for dimension reasons, precisely \cite[Arcata IV Thm.\ 6.1]{SGA4.5} applied to the structure morphism $f:\{0\}\to \Spec (k)$.  So, to show that $j_!i^*\cG$ is elementary it suffices to prove the left-hand $H_c^2$ vanishes. 

Recall \cite[Sommes trig.\ Rem.\ 1.18(d)]{SGA4.5} or \cite[2.0.6]{Katz1988} that if $\cF$ is a sheaf lisse on a connected dense open set $U$, its $H_c^2$ admits a simple interpretation in terms of the Galois representation of the \'etale fundamental group $\pi^{\rm geom}_1=\pi_1(U_{\overline{k}}, \overline{\eta})$ on $\cF_{\overline{\eta}}$, where $\overline{\eta}$ is a geometric generic point of $U_{\overline{k}}:=U \times_k \overline{k}$. Precisely, we have
\begin{equation}\label{H2cExplicit}H_c^2(U_{\overline{k}}, \cF) \simeq (\cF_{\overline{\eta}})_{\pi^{\rm geom}_1}(-1),\end{equation}
i.e.\ the $\pi^{\rm geom}_1$-coinvariants twisted by $\overline{\Q}_\ell(-1)$. 

We claim that the sheaf $j_!i^*\cG \otimes \mathcal{L}_{\addchar(ax)}$ is geometrically irreducible. Indeed, the sheaf $j_!i^*\cG(1/2)$ is lisse on $\G_{m,k}$,  pure of weight 0, and geometrically irreducible by Lemma \ref{sheafGproperties}(1). Then, by Katz's diophantine irreducibility criterion \cite[Lem.\ 7.0.3]{KatzRigidLocalSystems}, we have for all $n\geq 1$ that 
$$\frac{1}{q^n}\sum_{x \in k_n^\times} |t_{j_!i^*\cG(1/2)}(x,k_n)|^2 = 1+ O(q^{-n/2}),$$
where $k_n/k$ is an extension of degree $n$ and $t_{\cF}(x,k_n) = \Tr(\Fr_x| \cF)$. Given two finite-dimensional representations $V$ and $W$ of a compact group $G$, the eigenvalues of $\sigma \in G$ on $V \otimes W$ are the pairwise product of the eigenvalues of $\sigma$ on $V$ and $W$. Therefore, the sheaf $j_!i^*\cG(1/2) \otimes  \mathcal{L}_{\addchar(ax)}$ is pure of weight 0 because $j_!i^*\cG(1/2)$ and $\mathcal{L}_{\addchar(ax)}$ are pure of weight 0. 
Then, since $|t_{j_!i^*\cG(1/2) \otimes  \mathcal{L}_{\addchar(ax)}}(x,k_n)|= |t_{j_!i^*\cG(1/2)}(x,k_n)\addchar(\Tr_{k_n/k}(ax))|= |t_{j_!i^*\cG(1/2)}(x,k_n)|$, it follows by a reverse application of Katz's diophantine irreducibility criterion that $j_!i^*\cG \otimes \mathcal{L}_{\addchar(ax)}$ is geometrically irreducible, as claimed. 

Finally, since $j_!i^*\cG \otimes \mathcal{L}_{\addchar(ax)}$ is geometrically irreducible and  not geometrically trivial (it has generic rank 2), we have from the explicit description \eqref{H2cExplicit} that $H_c^2(\G_{m,{\overline{k}}}, j_!i^*\cG \otimes \mathcal{L}_{\addchar(ax)})=0$ as claimed. 
\end{proof}

\begin{mylemma}\label{sheafHproperties}
For $\psi$ non-trivial, the elementary sheaf $\cH:= \widehat{j_!i^* \cG}$
\begin{enumerate}
\item is lisse on $\G_{m,k}$, tame at 0, and of rank 3 with trace function 
\begin{equation}\label{eq:t_Heqn}
t_{\cH}(u) = -\sum_{x \in k^{\times}} t_{\cG}(x^{-1})\addchar(ux),
\end{equation} for all $u \in k$
\item is geometrically irreducible and pure of weight 2, and
\item has $\Swan_\infty(\cH)=1$, and ${\rm c}(\cH) \leq 6.$
\end{enumerate}
\end{mylemma}
\begin{proof}
Since $\cG$ is lisse on $\G_{m,k}$ and tame at $0$ by Lemma \ref{sheafGproperties}(3) (i.e.\ $\cG \in \mathcal{T}\subset \mathcal{T}_1$ in Katz's notation), we have an isomorphism
\begin{equation}
(\mathcal{L}_{\addchar})_! \ast \cG \simeq j^* \cH
\end{equation}
of lisse sheaves on $\G_{m,k}$  \cite[8.6.2 Prop.]{Katz1988}, where $\ast$ is Katz's geometric convolution. We can quickly derive the rest of the claimed properties of $\cH$ by bringing to bear Katz's theory of convolutions. 
\begin{enumerate}
\item The first three facts may be read off from \cite[5.2.1 Prop.(2)(3), 8.0.2 Lem.(1)]{Katz1988}. The formula for the trace function follows from proper base change and the Grothendieck-Lefshetz trace formula \cite[Arcata IV Thm.\ 5.4, Rapport Thm.\ 3.2]{SGA4.5} since by Theorem \ref{ThmGeomFT_MT}(1) the sheaf $\cH$ is elementary, and where the sum in \eqref{eq:t_Heqn} runs over $k^\times$ in lieu of $k$ because we have chosen the extension by zero $j_!$ of $i^{*}\cG$. 
\item These two facts follow from the expression of the sheaf $\cH$ as a geometric Fourier transform, specifically Theorem \ref{ThmGeomFT_MT}(3)(4). 
\item By the Euler-Poincar\'e formula \cite[2.3.1]{Katz1988} with $U=\G_m$ and the fact that $\cH$ is elementary, we have 
$\Swan_\infty(\cH) = h_c^1(U_{\overline{k}}, \cH)$. 
We have 
$$ 
h^1_c(U_{\overline{k}}, \cH) = h^1_c(U_{\overline{k}}, \cL_{\addchar})h^1_c(U_{\overline{k}}, \cG),
$$ 
by \cite[5.2.2.1 Cor.]{Katz1988}. Thus, by the Euler-Poincar\'e formula again and Lemma \ref{sheafGproperties}(3), we have $$\Swan_\infty(\cH) =\Swan_\infty(\cL_{\addchar}) \Swan_\infty(\cG) = 1 \cdot 1 =1.$$ The bound on the conductor follows by definition now. 
\end{enumerate}
\end{proof}

Now we draw together the previous results to conclude the required bound on $g(\xi,\psi)$. By Lemma \ref{sheafHproperties}(1) we have that
$$ 
g(\xi,\psi)=- \sum_y  t_{\cKl_{\rm ns}}(y) t_{\cH}(y).
$$ 
 We argue by the orthogonality relation in \cite[Thm.\ 5.2]{FKMS}, which is essentially the Riemann Hypothesis of Deligne \cite[Thm.\ 3.3.1]{DeligneWeil2}. First, take the Tate twists $\cKl_{\rm ns}(1/2)$ and $\cH(1)$, so that these sheaves retain the properties of  Lemmas \ref{sheafKlnsproperties} and \ref{sheafHproperties} except that they are of weight $0$. Then
$$ 
 g(\xi,\psi)=-q^{3/2} \sum_y  t_{\cKl_{\rm ns}(1/2)}(y) t_{\cH(1)}(y).
$$
Since $\cKl_{\rm ns}(1/2)$ and $\cH(1)$ are both geometrically irreducible, and of different ranks, they cannot be geometrically isomorphic. Finally, by the orthogonality relation \cite[Thm.\ 5.2]{FKMS} and since $\cKl_{\rm ns}$ and $\cH$ have bounded conductors (Lemmas \ref{sheafKlnsproperties}(4) and \ref{sheafHproperties}(3)), we have that 
\begin{equation*}
g(\xi,\psi) \ll q^2,
\end{equation*}
with an absolute implied constant. This concludes the proof of Proposition \ref{AGprop} for $\psi$ non-trivial.

\subsubsection{The case that $\psi$ is trivial}
Define
\begin{equation}\label{gxidef}
g(\xi) = - \sum_{y \in k^\times} t_{\cKl_{\rm ns}}(y) \Kl_2(y),
\end{equation}
where $\Kl_2(y) = \sum_{\substack{x_1,x_2 \in k^\times \\ x_1x_2=y}} \addchar(x_1+x_2)$ and $\cKl_{\rm ns}$ is defined in Lemma \ref{sheafKlnsproperties} with respect to $\xi$.  
\begin{myprop}\label{AGprop_triv}
If $\xi$ is non-trivial, then the character sum $g(\xi)$ satisfies $g(\xi) \ll q^{3/2}$ with an absolute and effective implied constant.
\end{myprop}
Since $t_{\cKl_{\rm ns}}(y)=-\Kl_{\rm ns}(y)$ and $\Kl_2(y) = S(y,1,p)$ when $k=\F_p$ and $\addchar=e_p$, Lemma \ref{AGprop_triv} implies that $\widehat{H}(\psi_0) \ll p^{1/2}$ when $p \neq 2$, $k=c_0=1$, and $p \nmid a_1a_2a_3$ cf.\ Lemma \ref{lemma:HhatSupercuspidalPrime_triv}.
\begin{proof}
Recall that $\Kl_2(y)= t_{\cKl}(y)$, where $\cKl=\cKl(\addchar; 1,1;1,1)$ in Katz's notation \cite[4.1.1 Thm.]{Katz1988} is the standard rank 2 Kloosterman sheaf attached to $\addchar$. Recall that $\cKl_{\rm ns}$ and $\cKl$ are both geometrically irreducible with conductor $5$, see Lemma \ref{sheafKlnsproperties}(3), (4) and e.g.\ \cite[Thm.\ 4.4]{FKMS}. Then, by the orthogonality relation \cite[Thm.\ 5.2]{FKMS}, to obtain square-root cancellation in \eqref{gxidef}, it suffices to show that  $\cKl_{\rm ns}$ is \emph{not} geometrically isomorphic to $\cKl$ (note $\cKl \simeq D(\cKl)$, see \cite[4.1.4 Cor.(1)]{Katz1988} or \cite[Thm.\ 4.4]{FKMS}, again). 
\begin{mylemma}
When $\xi$ is non-trivial, the lisse sheaves $ \cKl$ and $\cKl_{\rm ns}$ are not geometrically isomorphic.
\end{mylemma}
\begin{proof}
Indeed, Katz  showed that these sheaves are not isomorphic as $I_0$-representations \cite[Ch.\ 7]{Katz1988}. To pin the reference down more precisely, we unpack some notation. 

Let $\pi:\G_{m,K}  \to \G_{m,k}$ be the canonical projection. The representations of $\pi_1^{\rm geom} = \pi_1(\G_{m, \overline{k}}, \overline{\eta})$ associated to $\cKl$ and $\cKl_{\rm ns}$ are unchanged by the pull-backs $\pi^*\cKl$ and $\pi^*\cKl_{\rm ns}$. For $\cKl_{\rm ns}$ we have \cite[8.8.7 Lem.]{Katz1988} that $$\pi^*\cKl_{\rm ns} \simeq \cKl(\addchar_K; \xi, \xi^{q }; 1,1 ).$$ Meanwhile, we have $\cKl= \cL_{\addchar} \ast \cL_{\addchar}$ where $\ast$ is convolution in the sense of \cite[Ch.\ 5]{Katz1988}. Then, since $\ast$ commutes with taking fibers \cite[5.1 Convolution Theorem(2)]{Katz1988}, we have $$\pi^*\cKl \simeq \pi^*\cL_{\addchar} \ast \pi^*\cL_{\addchar} \simeq \cL_{\addchar \circ \Tr} \ast \cL_{\addchar \circ \Tr},$$ where the latter isomorphism is by \cite[\S 4.3(2)]{Katz1988}. So, in the notation of \cite[4.1.1 Thm.]{Katz1988}, we have 
$$ \pi^*\cKl \simeq \cKl(\addchar_K; 1,1;1,1).$$ 
Now working over $K$ so that the order of $\xi$ divides $|K|-1= q^2-1$, when $\xi$ is non-trivial the sheaves $\pi^*\cKl$ and $\pi^*\cKl_{\rm ns}$ are not isomorphic as $I_0$-representations \cite[\S 7.4.1]{Katz1988}, \`a fortiori as $\pi_1^{\rm geom} = \pi_1(\G_{m, \overline{k}}, \overline{\eta})$-representations. 
\end{proof}
\end{proof}

\subsection{The case $p=2$}
\label{section:pequals2supercuspidal}
\begin{mylemma}
\label{lemma:HhatPropertiesp=2}
Suppose $p=2$ and $k \geq c_0 + 100$.
\begin{enumerate}
\item If $(a_1 a_2 a_3, 2) \neq 1$ then $\widehat{H}(\psi)=0$.
\item We have
\begin{equation*}
\widehat{H}(\psi) \ll \frac{2^{3k/2}}{C(\psi)}.
\end{equation*}
\item If $\psi = \psi_0$ is trivial, then
\begin{equation*}
\widehat{H}(\psi_0) \ll 2^k \delta(k \geq 2c_0 +O(1)).
\end{equation*}
\end{enumerate}
\end{mylemma}
\begin{proof}
We pick up the computation from \eqref{eq:HhatDecentFormulaSupercuspidalInert2} and apply the Postnikov formula along the same lines as in the proof of Lemma \ref{lemma:HhatInitialEvaluation}.  Whereas in that lemma we chose $n=
\lfloor k/2\rfloor$, we now choose  $n = \lceil k/2 \rceil +1$.  This choice of $n$ renders the computations simpler since $2^{-1}p^{2n} \equiv 0 \pmod {p^k}$, allowing us to drop the quadratic terms. Set also $m=k-n = \lfloor k/2 \rfloor -1$  and note that $m, n = k/2 + O(1)$.  

For $i=1,2,3$, we replace $x_i$ by $x_i(1 + p^n y_i)$ and write $t=t_0(1+p^nt_1)$, where $(x_1, x_2, x_3, t_0)$ runs over a fixed choice of coset representatives for $[(\Z/p^k\Z)/ (p^n\Z/p^k\Z)]^3 \times (\cO_L/p^k\cO_L)/ (p^n\cO_L/p^k\cO_L)$, and where 
 $y_i \in \Z/p^{m}\Z$ and $t_1 \in \cO_L/p^{m}\cO_L$.  We have
  \begin{align}
 \overline{\psi}(1+p^n y_i) &=e_{p^k}(- \ell_{\psi} \log(1+p^n y_i))= e_{p^m}(- \ell_{\psi} y_i), \\
 \psi(\mathrm{Nm}(1+p^n t_1)) & = \psi_L(1+p^n t_1) = e_{p^m}(\mathrm{Tr}(\ell_{\psi}  t_1)), \\
 \xi(1+p^n t_1) & = e_{p^m}(\mathrm{Tr}(\ell_{\xi}  t_1)), \\
 e_{p^k}(-\Tr(t)) & = e_{p^k}(- \Tr(t_0))e_{p^m}(\Tr(t_0t_1)), \\
 e_{p^k}(a_ix_i) & \to e_{p^k}(a_ix_i)e_{p^m}(a_iy_ix_i), 
 \end{align}
 and
 \begin{multline}
 -\frac{x_1(1 + p^n y_1) x_2(1 + p^n y_2) x_3(1 + p^n y_3) a_1 a_2 a_3}{\mathrm{Nm}(t_0(1 + p^n t_1))}
\\
=
-A \cdot
[1 
+ p^n (y_1 + y_2 + y_3 - \mathrm{Tr}(t_1)) ]+O(p^{2n-1}),
\end{multline}
where $A = a_1a_2a_3x_1x_2x_3 \Nm(t_0)^{-1}$, as before. 
Thus we obtain
\begin{equation*}
\widehat{H}(\psi) = 
\overline{\gamma}\frac{p^{-d/2}}{p^{2k}}
\sumstar_{x_1, x_2, x_3 \shortmod{p^n}}
\sum_{t_0 \in (\mathcal{O}/p^n \mathcal{O})^{\times}}
(\dots) 
S_1,
\end{equation*}
where $(\dots)$ has absolute value at most $1$, and where
\begin{multline}\label{eq:p=2section1}
S_1 = 
\sum_{y_1, y_2, y_3 \shortmod{p^m}}
\sum_{t_1 \in \mathcal{O}/p^m \mathcal{O}} 
e_{p^{m}}(\mathrm{Tr}(t_1( \ell_{\psi}+\ell_\xi - t_0)) + a_1 x_1 y_1 + a_2 x_2 y_2 + a_3 x_3 y_3 
- \ell_{\psi}(y_1 + y_2 + y_3)
\\
- A (y_1 + y_2 + y_3 - \mathrm{Tr}(t_1)).
\end{multline}
Hence by additive orthogonality of characters \eqref{eq:addorth}, the sum $S_1$ equals $p^{5m}$ times the indicator function of the system of equations 
\begin{equation}
\label{eq:congruenceSystemp=2}
\begin{split}
a_i x_i - \ell_{\psi} - A = 0 \in \mz/p^m\Z, \\
\ell_{\psi} - t_0 + A +\ell_\xi = 0 \in \mathcal{O}/\mathfrak{p}^{em-d}.
\end{split}
\end{equation}

Recall \eqref{eq:valuationNormR}, from which we obtain $v_L(\ell_\xi)\geq 100$. Hence $t_0 \equiv \ell_{\psi} + A \pmod{p\cO_L}$, and by assumption $(t_0, p) = 1$.  Therefore, $a_i x_i \equiv t_0 \pmod{p\cO_L}$, and $(a_i x_i, p)=1$. This gives item (1). 

Now let us assume that $a_1=a_2=a_3=1$. The system of equations \eqref{eq:congruenceSystemp=2} is the same as in the $p$ odd case, up to the modulus of the congruences.  Carrying through the same algebraic manipulations as in the first paragraph of Lemma \ref{lemma:HhatInitialEvaluation2}, i.e.\ replacing some instances of $n$ there by $m$, and using the condition $v(\Tr \ell_\xi)\geq k-d$ from Lemma \ref{lemma:traceRisZero}, we deduce
that $\gamma p^{2k-5m+d/2}\widehat{H}(\psi)$ equals
\begin{multline}
\label{eq:pequals2supercuspidal_s-eqn_fixed}
\sumstar_{\substack{s \shortmod{p^m}  \\ \ell_\psi s^2 -\Nm(\ell_\xi)s + \Nm(\ell_\xi)\ell_\psi\equiv 0 \shortmod{p^m}}}
\thinspace 
\sumstar_{\substack{x_1,x_2,x_3 \shortmod{p^n} \\ x_1\equiv x_2\equiv x_3 \equiv s\shortmod{p^m}}} \\ \sum_{\substack{t_0 \in (\mathcal{O}/p^n \mathcal{O})^{\times} \\ t_0 \equiv \ell_\xi + s \shortmod{\fp^{em-d}}}} \psi\Big(\frac{\mathrm{Nm}(t_0)}{x_1 x_2 x_3}\Big)
\xi(t_0) e_{p^k}(-\mathrm{Tr}(t_0)+
  x_1 +  x_2 + x_3 
 - 
 A).
\end{multline}
where the summand does not depend on the choice of coset representative for $x_i, t_0 \pmod{p^k}$.
Focusing on the condition on the $s$ sum, $\widehat{H}(\psi)$ vanishes unless $\min(v( \mathrm{Nm}(\ell_\xi)),m) = \min(v(\ell_{\psi}),m)$, as in the proof of Proposition \ref{prop:HhatBoundkbiggish}.  If $v( \mathrm{Nm}(\ell_\xi)) \geq m$, then every value of $s$ is a solution, and we get the bound
\begin{equation}
\label{eq:Hhatp=2boundInProof}
\widehat{H}(\psi) \ll p^{5m-2k} p^{m} = p^{k+O(1)} \ll \frac{p^{3k/2 + O(1)}}{C(\psi)},
\end{equation}
since $v(\ell_{\psi}) \geq m$ implies $C(\psi) \ll p^{m} = p^{k/2+O(1)}$.  

Next suppose $v( \mathrm{Nm}(\ell_\xi)) < m$.  Then dividing through by $v( \mathrm{Nm}(\ell_\xi))$, 
and letting $U = p^{-v(\mathrm{Nm}(\ell_\xi))} \mathrm{Nm}(\ell_\xi)$,
we get
\begin{equation*}
s^2 \ell_{\psi}' - s U + \mathrm{Nm}(\ell_\xi) \ell_{\psi}' \equiv 0 \pmod{p^{m-v( \mathrm{Nm}(\ell_\xi))}}.
\end{equation*}
Since $v( \mathrm{Nm}(\ell_\xi)) \geq 1$,  this implies $s(s \ell_{\psi}' - U) \equiv 0 \pmod{p}$.  Since $(s,p) = 1$, then $s$ is uniquely determined modulo $p$.  By Hensel, it lifts to a unique solution modulo $p^{m-v( \mathrm{Nm}(\ell_\xi))}$, using that the derivative of the polynomial is nonzero modulo $p$ (it is congruent to $-U$ modulo $p$, since $2s \ell_{\psi}' \equiv 0 \pmod{p}$).  
Hence this unique solution  $s$ modulo $p^{m-v( \mathrm{Nm}(\ell_\xi))}$ gives rise to $p^{v( \mathrm{Nm}(\ell_\xi)) }$ solutions to the $s$ congruence in \eqref{eq:pequals2supercuspidal_s-eqn_fixed}, so
\begin{equation*}
\widehat{H}(\psi) \ll p^{5m-2k} p^{v( \mathrm{Nm}(\ell_\xi))} = \frac{p^{5m-k}}{p^{k-v( \mathrm{Nm}(\ell_\xi))}} = \frac{p^{3k/2 + O(1)}}{C(\psi)}.
\end{equation*}
Thus we have shown item (2).

For item (3) we pick up the calculation at \eqref{eq:pequals2supercuspidal_s-eqn_fixed} and use that $\ell_{\psi_0}= 0$. The outer sum there is only non-empty if $\Nm(\ell_\xi) \equiv 0 \pmod{p^m}$, which we now assume. For use in the following calculations, recall \eqref{eq:valuationNormR} that $v_E(\ell_\xi)\geq e(k-c_0)-d$, $v(\Nm(\ell_\xi))\geq 2(k-c_0)-d$,  Lemma \ref{lemma:traceRisZero} that $\Tr(\ell_\xi) = O(p^{k-d})$ and \eqref{eq:TraceImage} that $\Tr(\fp^{em-d}) = (p)^m$.

Replace $t_0$ in \eqref{eq:pequals2supercuspidal_s-eqn_fixed} with $ s+ \ell_\xi  + t_0\pi^{em-d}$, where $\pi$ is any uniformizer for $L$ and the new $t_0$ runs modulo $\fp^{e(n-m)-d}$. We have by the Postnikov formula that 
$$\xi(s+\ell_\xi + t_0 \pi^{em-d}) = \xi(s) e_{p^k}\Big(\Tr\Big(\frac{\ell_\xi^2}{s} + \frac{\ell_\xi t_0\pi^{em-d}}{s} - \frac{ \ell_\xi^2 t_0\pi^{em-d}}{s^2} + \frac{1}{3} \frac{\ell_\xi^4}{s^3} \Big)\Big).$$
Similarly, replace $x_i$ in \eqref{eq:pequals2supercuspidal_s-eqn_fixed} with $s+x_i p^m$ with the new $x_i$ running modulo $p^{n-m}$. 
To facilitate the following computation, let 
\begin{align*} L &  = x_1+x_2+x_3, \quad Q = x_1x_2+x_2x_3+x_1x_3, \\
T & = \Tr (\ell_\xi+ t_0 \pi^{em-d}), \quad N = \Nm(\ell_\xi+t_0\pi^{em-d}),\end{align*}
and note that $v(T)\geq m$ and $v(N) \geq m$.
We have
$$e_{p^k}(-\mathrm{Tr}(t_0)+  x_1 +  x_2 + x_3) \to e_{p^k}( -2s -T +3s+p^mL),$$
and 
\begin{multline*} 
A = \frac{(s+x_{1} p^m)(s+x_{2} p^m)(s+x_{3} p^m)}{\Nm(\ell_\xi + s + t_0 \pi^{em-d})} \\
=s-T-s^{-1} N+p^mL + s^{-1} \left( (T+s^{-1} N)^2 -  p^m T L-  s^{-1}p^m NL + p^{2m}Q\right) + O(p^k).
\end{multline*}
Gathering the above computations, after the indicated change of variables as well as the change $s \to s^{-1}$, the expression in \eqref{eq:pequals2supercuspidal_s-eqn_fixed} becomes
\begin{multline}\label{p=2gathered}
\sumstar_{s \shortmod{p^m}} \sum_{x_1,x_2,x_3 \shortmod{p^{n-m}}}\sum_{t_0 \in \cO/\fp^{e(n-m)-d}} \xi(s) e_{p^k}( -s\Nm(\ell_\xi) + s\Nm( t_0 \pi^{em-d}) \\ -s(T+Ns)^2 +  sp^m T L+ s^{2}p^m NL
- sp^{2m}Q - s^2\Tr ( \ell_\xi^2 t_0\pi^{em-d}) + s^3\Tr( \tfrac{1}{3} \ell_\xi^4)).
 \end{multline}
 Note that the argument of $e_{p^k}$ in \eqref{p=2gathered} is $\equiv -s\Nm(\ell_\xi) \pmod{p^{2m-d}}$. Moreover, since $\pi_\xi$ has trivial central character $\xi\vert_{\Q_p^\times}$ is a quadratic character of conductor $p^d$ by Lemma \ref{lemma:thetarestrictedtoZp}. So, we write $s=s_0+p^6s'$, where $s_0$ runs modulo $p^6$ and $s'$ runs modulo $p^{m-6}$, and obtain that $\gamma p^{2k-5m+d/2}\widehat{H}({\psi_0}) $ equals
$$\sumstar_{s_0 \shortmod{p^6}} \sum_{x_1,x_2,x_3 \shortmod{p^{n-m}}}\sum_{t_0 \in \cO/\fp^{e(n-m)-d}} (\cdots) \sum_{s' \shortmod{2^{m-6}}} e_{p^k}\left(  - \Nm(\ell_\xi) p^6s'\right), $$
where $(\cdots)$ has absolute value at most $1$ and does not depend on $s'$. The interior sum vanishes unless $2(k-c_0)-d\geq k-6,$ leading to the bound
$$  |\widehat{H}({\psi_0})|\ll p^{5m-2k+m}\delta(k\geq 2c_0 + d-6) \ll p^k\delta(k\geq 2c_0+O(1)). \qedhere $$  
\end{proof}

\section{Proof of Lemma \ref{lemma:Zlemma}}
In this section we prove Lemma \ref{lemma:Zlemma}, drawing on the estimates from Section \ref{section:HhatCalculations}.
Recall that $Z$ is given by \eqref{eq:Zformula2} and $Z_{\mathrm{fin}, p}$ is given by \eqref{eq:ZfinFormula2}.

Our first goal is to reduce the problem to estimating $Z$ when $Q$ is divisible by a single prime.
The overarching strategy of the proof is similar to  the proof of  \cite[Lem.\ 4.2]{PetrowYoungCoset}, which we summarize here.  
We decompose $Z = Z_0 + Z_1$, where $Z_0$ corresponds to the sub-sum with $\psi=\psi_0$ trivial.   A little simplification and using the multiplicativity of $|\widehat{H}|$ shows that
\begin{equation}
\label{eq:Z0productbound}
 Z_0 \ll Q^{\varepsilon}
 \prod_{p|Q} 
 \sum_{k \geq v_p(q')} \frac{p^{v_p(q')}}{p^k} 
 \sum_{a_1, a_2, a_3 | p^{\infty}}
 \frac{|\widehat{H}(\psi_0, a_{1}, a_{2}, a_{3})|}{a_{1}^{\sigma_1} a_{2}^{\sigma_2} a_{3}^{\sigma_3} (p^k/p^{v_p(q')})^{\sigma_4}}.
\end{equation}
where $\sigma_i > 1$ for all $i$, and where $\psi_0$ is the trivial character modulo $p^k$.  
Note that our desired upper bound on $Z_0$, namely \eqref{eq:Z0bound}, is multiplicative, and so is the right hand side of \eqref{eq:Z0productbound}, so it suffices to prove the bound one prime at a time. 

Meanwhile, for $\sigma_i>1/2$ we have
\begin{equation}
\label{eq:Z1multi}
Z_1 \ll Q^\eps \sum_{\substack{c_Q | Q^{\infty} \\ q' | c_Q}}
\frac{q'}{c_Q} 
\sum_{\substack{m_{i} | Q^{\infty} \\ (i=1,2,3)}}
\sum_{\psi \shortmod{c_Q}}  \frac{|\widehat{H}(\psi, m_{1}, m_{2}, m_{3})|}{m_{1}^{\sigma_1} m_{2}^{\sigma_2} m_{3}^{\sigma_3} (c_Q/q')^{\sigma_4}}
|\mathscr{L}(\psi)|,
 \end{equation}
where here and below we employ the shorthand notation
\begin{equation}
\mathscr{L}(\psi) = L(s_1, \psi) L(s_2, \psi) L(s_3, \psi) L(s_4, \overline{\psi}).
\end{equation}
By positivity, we have extended the sum in \eqref{eq:Z1multi} to include the trivial character.
We also remind the reader that for the purposes of estimating $Z_1$ in Lemma \ref{lemma:Zlemma}, we will take $s_1 = s_2 = s_3 = \sigma+it$, and $s_4 = \sigma - it$, but we generally leave this substitution implicit for brevity.
The fourth moment of Dirichlet $L$-functions is not itself multiplicative, but the bound in Theorem \ref{thm:fourthcoset} is multiplicative. So, with the goal of reducing to a multiplicative situation, we decompose the sums in \eqref{eq:Z1multi} as follows. 

Given $c_Q$, for odd primes $p$ with $v_p(c_Q)= c_0\geq 2$, recall the quantity $\Delta_p \in \Z/p^{v_p(c_Q)-1}\Z$ defined in \cite[Thms.\ 3.3, 3.4]{PetrowYoungCoset} when $\sigma_p$ is principal series and by equation \eqref{eq:DeltaDef} when $\sigma_p$ is supercuspidal with $L/\Q_p$ unramified. 
Using Lemmas \ref{lemma:HhatEvalPScase}, \ref{lemma:HhatVanishesUnlessCoprimeSupercuspidalInert}, \ref{lemma:HhatpsiTrivialInertkLarger}, \ref{lemma:HhatSupercuspidalPrimeTrivialPsi}, \ref{lemma:HhatpsiTrivialInertkSmallest}, \ref{lemma:HhatBoundSupercuspidalInertkequalsi0IntermediateStep}, \ref{lemma:HhatBoundInertRhoBound}, \ref{lemma:HhatSupercuspidalInertDegenerate}, \ref{lemma:HhatBoundSupercuspidalInertDegenerate}, \ref{lemma:HhatSupercuspidalPrime}, \ref{lemma:HhatSupercuspidalPrime_triv}, \ref{lemma:HhatPropertiesp=2}, and Propositions \ref{prop:HhatBoundkbiggish}, \ref{AGprop}, and \ref{AGprop_triv}, we will bound $|\widehat{H}(\psi)|$ in terms of 
 $c_Q$, $C(\psi)$, $v_p(\Delta_p)$, and $v_p(m_i)$ for $p | Q$ and $i=1,2,3$ only (and $(\sigma_p)_{p \in S}$, of course). We now decompose \eqref{eq:Z1multi} accordingly. 
Given $c_Q$,  let $$c_Q' = \prod_{\substack{p | c_Q, \,p \neq 2 \\ v_p(c_Q) = c_0\geq 2}} p^{v_p(c_Q)}.$$ 
For a positive integer $n$, let $\widetilde{n}$ be its square-free radical. Let $\Delta = \Delta(\psi) \in [1, c_Q'/\widetilde{c_Q'}]\cap \Z$ be such that $\Delta \equiv \Delta_p \pmod {p^{v_p(c_Q)-1}}$ for all  $ p|c_Q'$. For $a| c_Q'/\widetilde{c_Q'}$ write $a \| \Delta$ as shorthand for the condition that $v_p(\Delta_p) = v_p(a)$  for all $ p|c_Q'$.  

For $p | Q$, $m_1, m_2,m_3 | p^\infty$,  
$0\leq \beta \leq v_p(c_Q)$, and $0 \leq \alpha \leq \max(v_p(c_Q') -1, 0)$ 
let 
\begin{equation}\label{eq:sec6Mdef}
M(p^\alpha, p^\beta , m_1,m_2,m_3) = \begin{cases} \max_{\substack{C(\psi) = p^\beta \\ v_p(\Delta_p) = \alpha}} |\widehat{H}(\psi, m_1,m_2,m_3)| & \text{ if } v_p(c_Q') >0, \\ 
\max_{C(\psi) = p^\beta} |\widehat{H}(\psi, m_1,m_2,m_3)| & \text{ if } v_p(c_Q') =0 .
\end{cases}
\end{equation}
Extending the definition of $M$ multiplicatively, we obtain by Lemma \ref{lemma:multiplicativity} that for each $c_Q$, $b | c_Q$, $a | c_Q'/\widetilde{c_Q'}$, and $m_i | Q^\infty$   
\begin{equation}
\label{eq:Bdef} 
|\widehat{H}(\psi, m_1,m_2,m_3)|\leq M(a, b, m_1,m_2,m_3),
\end{equation}
 for all $ \psi$ with $ C(\psi) = b$  and $a\| \Delta$.

Implementing this decomposition of \eqref{eq:Z1multi}, we have
\begin{equation}
\label{eq:Z1multi_ian2} 
Z_1 \ll Q^\eps 
\sum_{\substack{ c_Q | Q^\infty \\ q'|c_Q }} \frac{q'}{c_Q} 
\sum_{b | c_Q} 
\sum_{a | c_Q'/\widetilde{c_Q'}} 
\sum_{\substack{m_i | Q^\infty \\ (i=1,2,3)}} 
\frac{M(a,b,m_1,m_2,m_3)}{m_1^{\sigma_1}m_2^{\sigma_2}m_3^{\sigma_3}(c_Q/q')^{\sigma_4}}
\sum_{\substack{\psi \shortmod{c_Q} \\ C(\psi) =b \\ a \| \Delta(\psi)}} |\mathscr{L}(\psi)|.
\end{equation}
We argue that \eqref{eq:Z1multi_ian2} may be restricted by $a|b$.  If $p \nmid a$, then clearly $v_p(a)\leq v_p(C(\psi))$, so suppose $p|a$, that is, $v_p(\Delta) > 0$.
If $\sigma_p$ is supercuspidal with $L/\Q_p$ unramified, then recall from Lemma \ref{lemma:vpNormOfEllTheta}(2) that $v_p(\mathrm{Nm}(\ell_{\xi})) = 0$, which implies from \eqref{eq:DeltaDef} that $v_p(\ell_{\psi}) = 0$, i.e., $v_p(C(\psi)) = v_p(c_Q) $.  If $\sigma_p$ is principal series, then in \cite[\S 3.2]{PetrowYoungCoset}, $\Delta_p$ is undefined unless $v_p(C(\psi)) =  v_p(c_Q) $.  In either case, $v_p(a)\leq v_p(c_Q'/\widetilde{c_Q'})\leq v_p(c_Q') \leq v_p(c_Q) = v_p(C(\psi))$. Hence we may assume $a|b$.

By positivity, and by trivially estimating finitely many Euler factors, we have 
\begin{equation}
\label{scrLsum}
\sum_{\substack{\psi \shortmod{c_Q} \\ C(\psi) =b \\ a \| \Delta(\psi)}} |\mathscr{L}(\psi)| \leq \sum_{\substack{\psi \shortmod{c_Q} \\ C(\psi) =b \\ \Delta(\psi) \equiv 0 \shortmod{a}}} |\mathscr{L}(\psi)|\ll Q^\eps  \sum_{\substack{\psi \shortmod{b} \\ \Delta(\psi) \equiv 0 \shortmod{a}}} |\mathscr{L}(\psi)|.
\end{equation}

We would like to apply Theorem \ref{thm:fourthcoset} to \eqref{scrLsum}, so we further break up the sum on the right over cosets. Let $X(b)=\{\psi \pmod b\}$ be the group of Dirichlet characters modulo $b$, and consider the subgroup $\psi_0X(b/a)$ of $X(b)$, where $\psi_0$ is the trivial character modulo $b$. Then $\psi, \psi' \in X(b)$ are in the same $\psi_0X(b/a)$-coset if and only if $\ell_\psi \equiv \ell_{\psi'} \pmod{a}$. Thus, if $\psi,\psi'$ are in the same $\psi_0X(b/a)$-coset, then $\Delta(\psi) \equiv \Delta(\psi') \pmod a$. So,
\begin{equation}\label{decomp_into_cosets} \sum_{\substack{\psi \shortmod{b} \\ \Delta(\psi) \equiv 0 \shortmod{a}}} |\mathscr{L}(\psi)| = \sum_{\substack{ \theta \in X(b)/ \psi_0X(b/a) \\ \Delta(\theta) \equiv 0 \shortmod a}} \sum_{ \eta \in  \psi_0X(b/a)} | \mathscr{L}(\eta . \theta)|.\end{equation}
We insert \eqref{decomp_into_cosets} back into \eqref{eq:Z1multi_ian2}, integrate both sides, and finally apply Theorem \ref{thm:fourthcoset} to obtain %\my{I moved $Q^{\eps}$ out of the definition of $Y_1$ to save a little space}
\begin{equation*}
\frac{1}{X^{1+\eps}} \int_{-X}^X |Z_1|\,dt \ll Q^{\varepsilon} Y_1,
\end{equation*}
where
\begin{equation}
\label{eq:Z1multi_ian3}
 Y_1 :=  \sum_{\substack{c_Q | Q^\infty \\ q' | c_Q }} \frac{q'}{c_Q} \sum_{b | c_Q} \sum_{\substack{a | c_Q'/\widetilde{c_Q'} \\ a \mid b}} \sum_{\substack{m_i | Q^\infty \\ (i=1,2,3)}} \frac{M(a,b,m_1,m_2,m_3)}{m_1^{\sigma_1}m_2^{\sigma_2}m_3^{\sigma_3}(c_Q/q')^{\sigma_4}} \sum_{\substack{ \theta \in X(b)/ \psi_0X(b/a) \\ \Delta(\theta) \equiv 0 \shortmod a}} \lcm(b/a, b_0),
\end{equation}
and where $b_0$ is the least positive integer such that $b_0 | b$ and $b^2 | b_0^3$.
Now $Y_1$ is multiplicative and so is the desired upper bound \eqref{eq:Z1bound}. Therefore, to prove the second part of Lemma \ref{lemma:Zlemma} it suffices to bound $Y_1$ for prime powers. We summarize the above discussion as follows. 
\begin{mylemma}
Given $k\geq v_p(q')$, let $k' = k$ if $p\neq 2$, $k=c_0$, and $c_0\geq 2$, and $k'=0$ otherwise. If the $p$-part of $Y_1$, i.e.\ 
$$\sum_{k = v_p(q')}^\infty \frac{p^{v_p(q')}}{p^k} \sum_{\beta =0}^k \sum_{\alpha = 0}^{\max(k'-1,0)} \sum_{\substack{m_i | p^\infty \\ (i=1,2,3)}} \frac{M(p^\alpha,p^\beta,m_1,m_2,m_3)}{m_1^{\sigma_1}m_2^{\sigma_2}m_3^{\sigma_3}p^{\sigma_4(k-v_p(q'))}} \sum_{\substack{ \theta \in X(p^\beta)/ \psi_0X(p^{\beta-\alpha}) \\ \Delta(\theta) \equiv 0 \shortmod {p^\alpha}}} p^{\max(\beta-\alpha, \lceil2 \beta/3\rceil)}$$ 
is $\ll_{\sigma, \eps} p^{(3/2+\eps)v_p(q')}$ for all $p \mid Q$, then the second part of Lemma \ref{lemma:Zlemma} holds. 
\end{mylemma}
 Recall that  the results from Section \ref{section:HhatCalculations}  typically assume that $p^k$ is large, however the contribution to $Z_1$ or $Z_0$ from terms with $p^k = O(1)$ may be handled directly with a trivial bound.  Based on this simple observation, in the forthcoming work we will implicitly assume that $p^k$ is large enough so that the bounds from Section \ref{section:HhatCalculations} may be applied.

\begin{mylemma}[Principal series]\label{lemma:ZpropertiesPS}
 Suppose $\sigma_p$ is trivial central character principal series or special with $c(\sigma_p)\geq 2$.  Then the properties of $Z$ from Lemma \ref{lemma:Zlemma} hold.
\end{mylemma}
\begin{proof}
Recall that the geometric conductor $v_p(q')= \lceil \frac{c(\sigma_p)}{2}\rceil = {c_0}$ in this situation. Let $Z_{\mathrm{fin}}^{\mathrm{PY}}(\sigma_1, \sigma_2, \sigma_3, \sigma_4)$ be defined as in \cite[\S 7]{PetrowYoungWeyl},
and let $\|Z_{\mathrm{fin}, p}^{\mathrm{PY}}\| = \|Z_{\mathrm{fin}, p}^{\mathrm{PY}}(\psi)\|$ cf.\ \eqref{eq:ZfinL1normDef} be defined by
\begin{equation}
\label{eq:ZfinPYnormDef}
\|Z_{\mathrm{fin}, p}^{\mathrm{PY}}(\psi)\| 
= \sum_{\substack{a_1, a_2, a_3, d |p^{\infty} }} \frac{|\widehat{H}_{\mathrm{PY}}(\psi, \chi, a_1, a_2, a_3, d)|}{a_1^{\sigma_1} a_2^{\sigma_2} a_3^{\sigma_3} d^{\sigma_4}}.
\end{equation}
Technical remark: in \cite[\S 7]{PetrowYoungWeyl}, the sum defining $Z_{\mathrm{fin},p}$ 
 had an extra condition $(a_1, d) = 1$ which holds automatically since $\widehat{H}_{\mathrm{PY}}$ vanishes otherwise, so we are free to omit it in \eqref{eq:ZfinPYnormDef}.

Applying Lemma \ref{lemma:HhatEvalPScase}  to \eqref{eq:Z0productbound},  the $p$-part of the bound on $Z_0$ for $\real(s_j)\geq \sigma>1$ is 
\begin{equation*}
\ll \frac{p^{\eps c(\sigma_p)}}{p^{c_0}}\sum_{k=c_0}^\infty \sum_{a_1,a_2,a_3 | p^\infty} \frac{|\widehat{H}_{\rm PY}(\psi_0, \chi, a_1,a_2,a_3,p^{k-c_0})|}{a_1^{\sigma_1} a_2^{\sigma_2}a_3^{\sigma_3}(p^{k-c_0})^{\sigma_4}} \\
= \frac{p^{\eps c(\sigma_p)}}{p^{c_0}}\|Z_{\mathrm{fin}, p}^{\mathrm{PY}}(\psi_0)\| \ll p^{\eps c(\sigma_p)} 
\end{equation*} 
by \cite[(4.3)]{PetrowYoungCoset}, which is acceptable. 

Turning to $Z_1$, when $\real(s_j)\geq \sigma>1/2$, we have by Lemma \ref{lemma:HhatEvalPScase}(1) that the $p$-part of $Y_1$ is bounded by 
\begin{equation}\label{ian_PS1}
  \sum_{\beta=0}^{c_0} \sum_{\alpha=0}^{\max(k'-1,0)} \sum_{k=c_0}^\infty \frac{p^{c_0}}{p^k}\sum_{m_1,m_2,m_3 | p^\infty}  \frac{M(p^\alpha, p^\beta, m_1,m_2,m_3)}{ m_1^{\sigma_1}m_2^{\sigma_2} m_3^{\sigma_3}p^{(k-c_0)\sigma_4}}  \sum_{\substack{ \theta \in X(p^\beta)/ \psi_0X(p^{\beta-\alpha}) \\ \Delta(\theta) \equiv 0 \shortmod {p^\alpha}}} p^{\max(\beta-\alpha, \lceil 2\beta/3\rceil)}.
\end{equation} 
By Lemma \ref{lemma:HhatEvalPScase}(2) we have $$|\widehat{H}(\psi, a_1, a_2, a_3)|
= p^{k-2c_0}|\widehat{H}_{\mathrm{PY}}(\psi, \chi, a_1, a_2, a_3, p^{k-c_0})|.$$ 
Therefore, the sub-sum of \eqref{ian_PS1} over $k, m_1,m_2,m_3$ is essentially equal to $p^{-c_0} \|Z_{\mathrm{fin}, p}^{\mathrm{PY}}(\psi)\|$, i.e.\ they are equal except in that the sub-sum of \eqref{ian_PS1} has extra maximum over $\psi$ inside the sums, see \eqref{eq:sec6Mdef}. So, the expression \eqref{ian_PS1} is the decomposition of formula (4.6) within the proof of Lemma 4.2 of \cite{PetrowYoungCoset} according to the possible values of $C(\psi)$ and $\Delta(\psi)$, locally at $p$. Regardless of these minor differences between \eqref{ian_PS1} and (4.6) of \cite{PetrowYoungCoset}, applying the same steps as in the proof of Lemma 4.2 of loc.\ cit., one obtains the desired properties of $Z_1$. 
\end{proof}

\begin{mylemma}[Supercuspidal, ramified]
\label{lemma:ZpropertiesRamified}
Suppose that $p$ is odd and $\sigma_p$ is a trivial central character supercuspidal representation corresponding to an admissible pair $(L/\Q_p,\xi)$ with $L/\Q_p$ ramified.
 Then the properties of $Z$ from Lemma \ref{lemma:Zlemma} hold.
\end{mylemma}
\begin{proof}
Recall 
from \eqref{eq:supercuspidalTable} 
that $c(\sigma_p)= 2c_0+1$ with  geometric conductor  $q' = p^{c_0 + 1}$, so that $k'=0$.  Lemma \ref{lemma:HhatVanishesUnlessCoprimeSupercuspidalInert} and Proposition \ref{prop:HhatBoundkbiggish} assert that
$$M(1, p^\beta, m_1,m_2,m_3) \ll \delta(m_1m_2m_3=1)  p^{\lceil\frac{3k}{2}\rceil - \beta} .$$
Thus, the $p$-part of $Y_1$ is bounded by 
\begin{equation*}
\sum_{k = c_0+1}^{\infty} \frac{p^{v_p(q')}}{p^k} 
 \sum_{0 \leq j \leq k}
  \frac{ p^{3k/2}}{ p^{(k-c_0-1)\sigma_4}}
  \ll   \sum_{k = c_0+1}^{\infty} p^{v_p(q')} 
  \frac{ p^{k/2 + k\varepsilon}}{ p^{(k-c_0-1)\sigma_4}}
  \ll p^{(3/2+\varepsilon)v_p(q')}.
\end{equation*}
This is the desired bound on $Z_1$. 

To estimate $Z_0$, we have $\sigma_4 > 1$, $\psi = \psi_0$, $a_1 = a_2 = a_3 = 1$, so \eqref{eq:Z0productbound} becomes
\begin{equation*}
 Z_0 \ll 
 \sum_{k = c_0 + 1}^{\infty} \frac{p^{v_p(q')}}{p^k} \frac{|\widehat{H}(\psi_0, 1,1,1)|}{p^{(k-c_0-1)\sigma_4}}.
\end{equation*}
The bound on $|\widehat{H}|$ is given by Lemma \ref{lemma:HhatpsiTrivialInertkLarger}. 
Recall from Lemma \ref{lemma:vpNormOfEllTheta} that $v_p(\mathrm{Nm}(\ell_{\xi})) = 2 (k-c_0) - 1$.  It is well-known that for $k \geq 2$,  the Ramanujan sum
$S(p^j,0;p^k) = 0$ unless $j \geq k-1$, which means that $|\widehat{H}(\psi_0, 1, 1, 1)| = 0$ unless $2(k-c_0) - 1 \geq k-1$, that is, $k \geq 2c_0$.  Therefore,
\begin{equation*}
 Z_0 \ll 
 \sum_{k = 2c_0}^{\infty} \frac{p^{v_p(q')}}{p^k} \frac{(p^k, p^{2(k-c_0) - 1})}{p^{(k-c_0-1)\sigma_4}}
 \ll \frac{p^{v_p(q')}}{p^{c_0}} = p.
\end{equation*}
On the other hand, the target bound is $p^{2v_p(q')} /Q^{3/4} = p^{2(c_0+1) - \frac34 (2c_0+1)} = p^{\frac{c_0}{2} + \frac{5}{4}} \geq p^{7/4}$, so our bound is more than satisfactory.
\end{proof}

\begin{mylemma}[Supercuspidal, unramified]\label{lemma:ZpropertiesUnramified}
 Suppose that $p$ is odd and $\sigma_p$ is a trivial central character supercuspidal representation corresponding to an admissible pair $(L/\Q_p,\xi)$ with $L/\Q_p$ unramified. 
 Then the properties of $Z$ from Lemma \ref{lemma:Zlemma} hold.
\end{mylemma}
Recall that $c(\sigma_p)=2c_0$ and $q' = p^{c_0}$ when $L/\Q_p$ is unramified.
\begin{proof}
First consider the bound on $Z_1$, with $\sigma_i > 1/2$.  We have that the $p$-part of $Y_1$ equals
\begin{multline}
 \sum_{k=c_0}^\infty \frac{p^{c_0}}{p^{k}} \sum_{\beta=0}^k \sum_{\alpha=0}^{\max(k'-1,0)} \sum_{\substack{m_i | p^\infty \\ (i=1,2,3)}} \frac{M(p^\alpha,p^\beta,m_1,m_2,m_3)}{m_1^{\sigma_1}m_2^{\sigma_2}m_3^{\sigma_3}(p^{k-c_0})^{\sigma_4}} \sum_{\substack{ \theta \in X(p^\beta)/ \psi_0X(p^{\beta-\alpha}) \\ \Delta(\theta) \equiv 0 \shortmod {p^\alpha}}} p^{\max(\beta-\alpha, \lceil 2\beta/3\rceil)}.
\end{multline}
Write $Y_1 = Y_1^{(c_0)} + Y_1^{(>c_0)}$, 
where $Y_1^{(c_0)}$ corresponds to the terms with $k=c_0$, and $Y_1^{(>c_0)}$ corresponds to the terms with $k>c_0$.
We claim $Y_1^{(>c_0)} \ll (q')^{3/2+\varepsilon}$.  To see this, recall that for $k \geq c_0 + 1$ we may assume $a_1 = a_2 = a_3 = 1$ by Lemma \ref{lemma:HhatVanishesUnlessCoprimeSupercuspidalInert}.  
In addition, Proposition \ref{prop:HhatBoundkbiggish} gives the same bound for ramified and unramified $L$.  Therefore, the same bound on $Z_1$ from Lemma \ref{lemma:ZpropertiesRamified} carries over to $Y_1^{(>c_0)}$.

We turn to $Y_1^{(c_0)}$.  If $c_0 = 1$, then we apply Lemmas \ref{lemma:HhatSupercuspidalPrime} and \ref{lemma:HhatSupercuspidalPrime_triv}, and Propositions \ref{AGprop} and \ref{AGprop_triv} to give
\begin{equation}
Y_1^{(c_0)} \ll p^{\varepsilon} \frac{p^{v_p(q')}}{p} p^{1/2} p = p^{3/2+\varepsilon} =  p^{(3/2+\varepsilon)v_p(q')}.
\end{equation}

Now suppose $c_0 \geq 2$.  Consider first the $\beta=0$ (i.e.\ $\psi$ trivial) term of $Y^{(c_0)}$, which is 
\begin{equation}\label{beta0c0geq2_term}  
\sum_{\substack{m_i | p^\infty \\ (i=1,2,3)}} \frac{M(1,1,m_1,m_2,m_3)}{m_1^{\sigma_1}m_2^{\sigma_2}m_3^{\sigma_3}}.\end{equation}
By Lemma \ref{lemma:HhatBoundSupercuspidalInertkequalsi0IntermediateStep}, \eqref{eq:valuationNormR}, and  Lemma
\ref{lemma:HhatpsiTrivialInertkSmallest}, we have 
$$ M(1,1,m_1,m_2,m_3) \ll \begin{cases} p^{k/2} & \text{ if } m_1m_2m_3=1 \\ p^{-k} (m_1,p^k)(m_2,p^k)(m_3,p^k) & \text{ if } p | m_1m_2m_3.\end{cases}$$
Inserting this into  \eqref{beta0c0geq2_term}, we have for
 $\sigma_1, \sigma_2, \sigma_3 \geq \sigma \geq 1/2$ that the $\beta=0$ term of $Y^{(c_0)}$ is 
\begin{equation*}
 \ll p^{k/2} 
+ p^{-k} \Big(\sum_{\alpha=0}^{\infty} \frac{(p^\alpha, p^k)}{p^{\alpha \sigma}}\Big)^3 
\ll p^{k/2} = p^{\frac{1}{2}v_p(q')}.
\end{equation*}
The contribution of $\psi$ trivial to $Z_1$ is hence bounded satisfactorily.

Now consider the $\beta>0$ (i.e.\ $\psi$ non-trivial) terms of $Y_1^{(c_0)}$, namely
\begin{equation}\label{beta>0c0geq2_terms}
\sum_{\beta=1}^{k} \sum_{\alpha =0 }^{k-1}  \sum_{\substack{m_i | p^\infty \\ (i=1,2,3)}}  \frac{M(p^\alpha,p^\beta,m_1,m_2,m_3)}{m_1^{\sigma_1}m_2^{\sigma_2}m_3^{\sigma_3}}  \sum_{\substack{ \theta \in X(p^\beta)/ \psi_0X(p^{\beta-\alpha}) \\ \Delta(\theta) \equiv 0 \shortmod {p^\alpha}}} p^{\max(\beta-\alpha, \lceil 2\beta/3\rceil)}.
\end{equation}
Lemma \ref{lemma:HhatSupercuspidalInertDegenerate} now implies
\begin{equation*}
\sum_{\substack{m_i | p^\infty \\ (i=1,2,3)}} \frac{M(p^\alpha,p^\beta,m_1,m_2,m_3)}{m_1^{\sigma_1}m_2^{\sigma_2}m_3^{\sigma_3}}
\leq
M(p^\alpha,p^\beta,1,1,1)
 + \frac{M(p^\alpha,p^\beta, p^{k-\beta}, p^{k-\beta}, p^{k-\beta})}{p^{\frac{3}{2}(k-\beta)}} .
\end{equation*}
 Using Lemma \ref{lemma:HhatBoundSupercuspidalInertDegenerate} to bound the second term, we obtain
\begin{equation}
\label{Using_lemma:HhatBoundSupercuspidalInertDegenerate}
\sum_{\substack{m_i | p^\infty \\ (i=1,2,3)}}\frac{M(p^\alpha,p^\beta,m_1,m_2,m_3)}{m_1^{\sigma_1}m_2^{\sigma_2}m_3^{\sigma_3}}
  \leq
M(p^\alpha,p^\beta,1,1,1)
 +  p^{k/2}.
\end{equation}
The contribution from the $p^{k/2}$ term  to $Y_1^{(c_0)}$ is acceptable, as inserting it back into \eqref{beta>0c0geq2_terms}  leads to
\begin{equation*}
p^{k/2} \sum_{\beta=1}^{k} \sum_{\alpha =0 }^{k-1} \sum_{\substack{ \theta \in X(p^\beta)/ \psi_0X(p^{\beta-\alpha}) \\ \Delta(\theta) \equiv 0 \shortmod {p^\alpha}}} p^{\max(\beta-\alpha, \lceil 2\beta/3\rceil)} 
\ll p^{(\frac{3}{2} + \varepsilon)k}
= p^{(\frac{3}{2} + \varepsilon)v_p(q')}.
\end{equation*}
The key remaining issue is to estimate
$$\sum_{\beta=1}^{k} \sum_{\alpha =0 }^{k-1}  M(p^\alpha,p^\beta,1,1,1)\sum_{\substack{ \theta \in X(p^\beta)/ \psi_0X(p^{\beta-\alpha}) \\ \Delta(\theta) \equiv 0 \shortmod {p^\alpha}}} p^{\max(\beta-\alpha, \lceil 2\beta/3\rceil)}.$$

We use Lemma \ref{lemma:HhatBoundInertRhoBound} to bound $M(p^\alpha,p^\beta,1,1,1)$ in terms of the number of solutions $\rho(\Delta, p^m)$ to $x^2 \equiv \Delta \pmod {p^m}$. By Hensel's Lemma and clearing common factors of $p$, we have
\begin{equation}
\rho(\Delta,p^m) = \begin{cases} 1+\left( \frac{\Delta}{p}\right) & \text{ if } m=1 \\ \rho(\Delta,p) & \text{ if } m\geq 2 \text{ and } v(\Delta)=0 \\ 0 & \text{ if } m\geq 2 \text{ and } v(\Delta)=1 \\ p\rho(\Delta p^{-2}, p^{m-2}) &   \text{ if } m\geq 2 \text{ and } v(\Delta)\geq 2. \end{cases} 
\end{equation}
Thus, we have $\rho(\Delta, p^m) \leq R(v(\Delta), m)$ with $$R(\alpha, m) := \begin{cases} 
  2 & \text{ if } m \leq 1 \\ 
  0 & \text{ if } m \geq 2 \text{ and } \alpha \text{ odd } \\ 
  2 p^{\lfloor \min(\alpha, m)/2\rfloor} &  \text{ if } m \geq 2 \text{ and } \alpha \text{ even.}\end{cases} $$  
  Now if $k=2n$ then we have by Lemma \ref{lemma:HhatBoundInertRhoBound} that 
$$M(p^\alpha, p^\beta, 1,1,1) \leq p^{k/2}R(\alpha, n),$$
while if $k=2n+1$ we have 
$$M(p^\alpha, p^\beta, 1,1,1) \leq p^{k/2}R(\alpha, n) + p^{\frac{k+1}{2}} \delta(\alpha \geq 2) R(\alpha-2, n-1).$$
Thus, 
both $k=2n$ and $k=2n+1$ require the estimation of
\begin{equation}\label{eq:Upsilondef}
\Upsilon := p^{k/2}  \sum_{\alpha =0 }^{k-1}  R(\alpha,n) \sum_{\beta=1}^{k}  \sum_{\substack{ \theta \in X(p^\beta)/ \psi_0X(p^{\beta-\alpha}) \\ \Delta(\theta) \equiv 0 \shortmod {p^\alpha}}} p^{\max(\beta-\alpha, \lceil 2\beta/3\rceil)}, 
\end{equation}
while the case $k=2n+1$ additionally requires
\begin{equation}
\Upsilon' := p^{\frac{k+1}{2}}  \sum_{\alpha =2 }^{k-1}  R(\alpha-2,n-1) \sum_{\beta=1}^{k}\sum_{\substack{ \theta \in X(p^\beta)/ \psi_0X(p^{\beta-\alpha}) \\ \Delta(\theta) \equiv 0 \shortmod {p^\alpha}}} p^{\max(\beta-\alpha, \lceil 2\beta/3\rceil)}.
\end{equation}

We begin with $\Upsilon$ and split the $\alpha$ sum at $n$, i.e.\ 
we write
$\Upsilon = \sum_{0 \leq \alpha \leq n} \Upsilon_{\alpha}$, where $\Upsilon_{\alpha}$ is the $\alpha$ term of $\Upsilon$ for $0\leq \alpha < n$ and 
$$\Upsilon_n = p^{k/2}  \sum_{\alpha =n }^{k-1}  R(\alpha,n)\sum_{\beta=1}^{k} \sum_{\substack{ \theta \in X(p^\beta)/ \psi_0X(p^{\beta-\alpha}) \\ \Delta(\theta) \equiv 0 \shortmod {p^\alpha}}} p^{\max(\beta-\alpha, \lceil 2\beta/3\rceil)}.$$

If $\alpha = 0$, then $R(\alpha, n) \leq 2$, and we quickly obtain $\Upsilon_0 \ll p^{3k/2} = p^{3v_p(q')/2}$.
At the other extreme, if $\alpha \geq n$, then $R(\alpha, n) \leq 2 p^{\lfloor n/2 \rfloor}$.  The condition
$\Delta(\theta) \equiv 0 \pmod{p^n}$ then means that $4 \ell_{\psi}^2 \equiv   \mathrm{Nm}(\ell_\xi) \pmod{p^n}$ for all $\psi$ in the coset $\theta$.  Since $(4^{-1}\mathrm{Nm}(\ell_\xi), p) = 1$, there are at most two solutions $\ell_{\psi} \pmod{p^n}$ to this congruence.  Therefore,
\begin{equation}
\Upsilon_n
\\
\ll  p^{\frac{k}{2} + \lfloor \frac{n}{2} \rfloor + \max(k-n,\lceil \frac{2k}{3} \rceil) +  k \varepsilon}.
\end{equation}
Note $\lfloor n/2 \rfloor = \lfloor \lfloor k/2 \rfloor /2 \rfloor \leq \lfloor k/4 \rfloor$, and 
$k-n = \lceil k/2 \rceil \leq \lceil 2k/3 \rceil$.
One can check by brute force computation that 
\begin{equation}
\label{eq:floork}
\lfloor k/4 \rfloor + \lceil 2k/3 \rceil \leq k
\end{equation}
 for  $k \geq 0$ (one only needs to check $0 \leq k \leq 11$, by breaking into arithmetic progressions modulo $12$).
Hence $\Upsilon_n \ll  p^{(\frac{3}{2} + \varepsilon)v_p(q')}$.
 
Next we deal with the intermediate terms, with $1 \leq \alpha< n$. Since $n\geq 2$, the bound $R(\alpha,n)$  vanishes unless  $\alpha$ is even, in which case we write $\alpha = 2j$ and have $R(2j, n) = 2p^j$.
 Next, as in the case that $\alpha \geq n$ above, we observe that the condition $\Delta(\theta) \equiv 0 \pmod{p^{2j}}$ specifies at most two cosets $\theta$.
Then
\begin{multline}
  \sum_{0 < \alpha < n}  \Upsilon_{\alpha} =
2p^{k/2}
\sum_{0 < 2j < n} p^j
\sum_{\beta=1}^{k}\sum_{\substack{ \theta \in X(p^\beta)/ \psi_0X(p^{\beta-2j}) \\ \Delta(\theta) \equiv 0 \shortmod {p^{2j}}}} p^{\max(\beta-2j, \lceil 2\beta/3\rceil)} 
\\
\ll 
p^{k(\frac12 + \varepsilon)} 
\sum_{0 < 2j < n} p^j (p^{k-2j} + p^{\lceil 2k/3 \rceil})
\ll p^{k\varepsilon} (p^{3k/2} + p^{k/2 + \lfloor k/4 \rfloor + \lceil 2k/3 \rceil}).
\end{multline}
By \eqref{eq:floork}, this meets the desired bound.

We also need to estimate $\Upsilon'$.   This is similar to $\Upsilon$, though the exponents are slightly modified. We write $\Upsilon' = \sum_{0<2j\leq n+1} \Upsilon'_{2j}$, where $\Upsilon'_{\alpha}$ for $\alpha\leq n$ corresponds to the $\alpha$ term of $\Upsilon'$ and 
$$ \Upsilon'_{n+1} = p^{\frac{k+1}{2}}  \sum_{\alpha =n+1 }^{k-1}  R(\alpha-2,n-1)\sum_{\beta=1}^{k} \sum_{\substack{ \theta \in X(p^\beta)/ \psi_0X(p^{\beta-\alpha}) \\ \Delta(\theta) \equiv 0 \shortmod {p^\alpha}}} p^{\max(\beta-\alpha, \lceil 2\beta/3\rceil)}.$$
Recall that $$R(2j-2, n-1) = \begin{cases} 2p^{j-1} & \text{ if } 0<2j< n+1 \\ 2p^{\lfloor \frac{n-1}{2}\rfloor} & \text{ if } 2j \geq n+1.\end{cases}$$

For the first type of terms we get
\begin{multline}
\sum_{0<2j<n+1} \Upsilon'_{2j} =2 p^{\frac{k+1}{2}}
 \sum_{0<2j< n+1}p^{j-1}  \sum_{\beta=1}^{k}\sum_{\substack{ \theta \in X(p^\beta)/ \psi_0X(p^{\beta-2j}) \\ \Delta(\theta) \equiv 0 \shortmod {p^{2j}}}} p^{\max(\beta-2j, \lceil 2\beta/3\rceil)} 
\\
\ll 
p^{k(\frac12 + \varepsilon)} p^{-1/2}
\sum_{0 < 2j \leq n} p^j (p^{k-2j} + p^{\lceil 2k/3 \rceil})
\ll p^{k\varepsilon}  (p^{\frac{3k-3}{2}} + p^{\frac{k-1}{2} + \lfloor n/2 \rfloor + \lceil 2k/3 \rceil}).
\end{multline}
  Here $ \lfloor n/2 \rfloor  = \lfloor (k-1)/4 \rfloor$.  One can check that $\lfloor (k-1)/4 \rfloor + \lceil 2k/3 \rceil\leq k$ by brute force. For the case $\alpha \geq n+1$ we get
  \begin{multline}
\Upsilon_{n+1}'= 2 p^{\frac{k+1}{2}} p^{\lfloor \frac{n-1}{2}\rfloor } 
\sum_{\alpha = n+1}^{k-1} \sum_{\beta=1}^{k} 
\sum_{\substack{ \theta \in X(p^\beta)/ \psi_0 X(p^{\beta-\alpha}) \\ \Delta(\theta) \equiv 0 \shortmod {p^{\alpha}}}} p^{\max(\beta-\alpha, \lceil 2\beta/3\rceil)}
\\
\ll p^{k\eps} p^{\frac{k+1}{2} + \lfloor \frac{n-1}{2}\rfloor } ( p^{k-n-1}+ p^{\lceil 2k/3 \rceil})
\ll  p^{k\eps}( p^{k+\lfloor \frac{k-3}{4}\rfloor} + p^{\frac{k+1}{2} + \lfloor \frac{k-3}{4}\rfloor  + \lceil \frac{2k}{3} \rceil} ).
  \end{multline}
   For this bound to be satisfactory, it suffices to check that
 \begin{equation}
 1 + \lfloor (k-3)/4 \rfloor + \lceil 2k/3 \rceil \leq k,
 \end{equation}
for $k \geq 3$ odd.  This is easily checkable by brute force, again.
In summary, we have shown the desired bound \eqref{eq:Z1bound} on $Z_1$.

 Next consider the bound on $Z_0$, corresponding to $\psi = \psi_0$.  
 By Lemmas \ref{lemma:HhatpsiTrivialInertkSmallest}, \ref{lemma:HhatVanishesUnlessCoprimeSupercuspidalInert} and \ref{lemma:HhatpsiTrivialInertkLarger}, we have for $c_0 \geq 2$ that
\begin{equation*}
 |Z_0| \ll
 \sum_{k=c_0} 
 \frac{p^{c_0}}{p^{2k}} 
\Big(\sum_{a|p^{\infty}} \frac{|S(a,0;p^k)|}{a^{\sigma}} \Big)^3
+
\sum_{k \geq c_0+1} \frac{p^{c_0}}{p^k} \frac{|S(\mathrm{Nm}(\ell_{\xi}), 0;p^k)|}{p^{(k-c_0) \sigma_4}},
\end{equation*}
where $\sigma, \sigma_4 > 1$.  
The sum over $k=c_0$ is $\ll p^{-c_0}$.  
For the sum over $k \geq c_0 +1$, recall from Lemma \ref{lemma:vpNormOfEllTheta} that $v_p(\mathrm{Nm}(\ell_{\xi})) = 2(k-c_0)$, which means the sum over $k$ here can be further restricted to $k \geq 2c_0 - 1$.
The term with $k = 2c_0-1$ is $\ll \frac{p^{c_0}}{p^{2c_0-1}} \frac{p^{2c_0-2}}{p^{c_0-1}} =  1$.  Similarly, the sum over $k \geq 2c_0$ is $\ll \sum_{k \geq 2c_0} \frac{p^{c_0}}{p^k} \frac{p^k}{p^{k-c_0}} \ll 1$.
This is stronger than the claimed bound of
$p^{2v_p(q')}/Q^{3/4}$, since $Q = p^{2v_p(q')} = p^{2c_0}$.

Next suppose $c_0 = 1$.  As in the case $c_0 = 2$, the sum splits into $k=c_0$ and $k \geq c_0+1$.  
By Lemma \ref{lemma:HhatSupercuspidalPrimeTrivialPsi} and Proposition \ref{AGprop}, the contribution from $k=c_0$ is at most
\begin{equation*}
\frac{p^{c_0}}{p^{c_0}} \Big(p^{1/2} + \frac{1}{p} \Big(\sum_{a | p^{\infty}} \frac{(a, p)}{a^{\sigma}} \Big)^3 \Big)
\ll p^{1/2},
\end{equation*}
which is consistent with $p^{2v_p(q')}/Q^{3/4} = p^{1/2}$.  For $k \geq c_0 + 1$, we may assume $a_1 = a_2 = a_3 = 1$ by Lemma \ref{lemma:HhatVanishesUnlessCoprimeSupercuspidalInert}, giving
\begin{equation*}
\sum_{k \geq c_0 + 1} \frac{p^{c_0}}{p^k}
\sum_{\substack{\psi_0 \shortmod{p^k} \\ \text{trivial}}} \frac{|\widehat{H}(\psi_0, 1,1,1)|}{p^{(k-c_0) \sigma_4}}.
\end{equation*}
Lemma \ref{lemma:HhatpsiTrivialInertkLarger} shows this is $O(1)$, which is better than what is required.
\end{proof}

\begin{mylemma}[Supercuspidal, $p=2$]
\label{lemma:Zpropertiesp=2}
 Suppose $p=2$ and $\sigma_2$ is a trivial central character supercuspidal representation.
 Then the properties of $Z$ from Lemma \ref{lemma:Zlemma} hold.
\end{mylemma}
\begin{proof}
Recall that when $p=2$ we have $v_2(q')\geq c_0+100$ due to our choice of enlarged family $\sigma_2[200]$ when $p=2$. Then  a minor modification of the proof of Lemmas \ref{lemma:ZpropertiesRamified} and \ref{lemma:ZpropertiesUnramified} carries over here, using Lemma \ref{lemma:HhatPropertiesp=2} in place of the similar bounds that were used for $p$ odd.
\end{proof}


\begin{thebibliography}{GHLN26}

\bibitem[BFW25]{BFW}
Olga Balkanova, Dmitry Frolenkov, and Han Wu.
\newblock On {W}eyl's subconvex bound for cube-free {H}ecke characters: Totally
  real case.
\newblock {\em arXiv preprint}, arXiv:2108.12283, 2025.

\bibitem[BH06]{BushnellHenniart:06a}
C.~Bushnell and G.~Henniart.
\newblock {\em The {Local} {Langlands} {Conjecture} for $\rm{GL}(2)$}.
\newblock Springer-Verlag, Berlin, 2006.

\bibitem[BK19]{BlomerKhan}
Valentin Blomer and Rizwanur Khan.
\newblock Twisted moments of {$L$}-functions and spectral reciprocity.
\newblock {\em Duke Math. J.}, 168(6):1109--1177, 2019.

\bibitem[CI00]{ConreyIwaniec}
J.~B. Conrey and H.~Iwaniec.
\newblock The cubic moment of central values of automorphic {$L$}-functions.
\newblock {\em Ann. of Math. (2)}, 151(3):1175--1216, 2000.

\bibitem[Del77]{SGA4.5}
P.~Deligne.
\newblock {\em Cohomologie \'etale}, volume 569 of {\em Lecture Notes in
  Mathematics}.
\newblock Springer-Verlag, Berlin, 1977.
\newblock S\'eminaire de g\'eom\'etrie alg\'ebrique du Bois-Marie SGA
  $4\frac{1}{2}$.

\bibitem[Del80]{DeligneWeil2}
Pierre Deligne.
\newblock La conjecture de {W}eil. {II}.
\newblock {\em Inst. Hautes \'Etudes Sci. Publ. Math.}, (52):137--252, 1980.

\bibitem[FKMS19]{FKMS}
\'Etienne Fouvry, Emmanuel Kowalski, Philippe Michel, and Will Sawin.
\newblock Lectures on applied {$\ell$}-adic cohomology.
\newblock In {\em Analytic methods in arithmetic geometry}, volume 740 of {\em
  Contemp. Math.}, pages 113--195. Amer. Math. Soc., [Providence], RI, [2019]
  \copyright2019.

\bibitem[Fro20]{Frolenkov}
Dmitry Frolenkov.
\newblock The cubic moment of automorphic {$L$}-functions in the weight aspect.
\newblock {\em J. Number Theory}, 207:247--281, 2020.

\bibitem[GHLN26]{GHLN}
Soumendra Ganguly, Peter Humphries, Yongxiao Lin, and Ramon Nunes.
\newblock Strong hybrid subconvexity for twisted selfdual {$\mathrm{GL}_3$}
  {$L$}-functions.
\newblock {\em Math. Ann.}, 2026.
\newblock To appear.

\bibitem[Guo96]{Guo}
Jiandong Guo.
\newblock On the positivity of the central critical values of automorphic
  {$L$}-functions for {${\rm GL}(2)$}.
\newblock {\em Duke Math. J.}, 83(1):157--190, 1996.

\bibitem[HPY26]{HPY}
Yueke Hu, Ian Petrow, and Matthew~P. Young.
\newblock A generalized {$\mathrm{PGL}(2)$} {P}etersson/{B}ruggeman-{K}uznetsov
  formula for analytic applications.
\newblock {\em Forum Math. Sigma}, 14:Paper No. e27, 2026.

\bibitem[Hu24]{Hu}
Yueke Hu.
\newblock The {P}etersson/{K}uznetsov trace formula with prescribed local
  ramifications.
\newblock {\em Amer. J. Math.}, 146(5):1193--1252, 2024.

\bibitem[IK04]{IK}
Henryk Iwaniec and Emmanuel Kowalski.
\newblock {\em Analytic number theory}, volume~53 of {\em American Mathematical
  Society Colloquium Publications}.
\newblock American Mathematical Society, Providence, RI, 2004.

\bibitem[Ivi01]{Ivic}
Aleksandar Ivi\'{c}.
\newblock On sums of {H}ecke series in short intervals.
\newblock {\em J. Th\'{e}or. Nombres Bordeaux}, 13(2):453--468, 2001.

\bibitem[JL70]{JacquetLanglands}
H.~Jacquet and R.~P. Langlands.
\newblock {\em Automorphic forms on {${\rm GL}(2)$}}.
\newblock Lecture Notes in Mathematics, Vol. 114. Springer-Verlag, Berlin-New
  York, 1970.

\bibitem[Kat80]{Katz1980}
Nicholas~M. Katz.
\newblock {\em Sommes exponentielles}, volume~79 of {\em Ast\'erisque}.
\newblock Soci\'et\'e{} Math\'ematique de France, Paris, 1980.
\newblock Course taught at the University of Paris, Orsay, Fall 1979, With a
  preface by Luc Illusie, Notes written by G\'erard Laumon, With an English
  summary.

\bibitem[Kat88]{Katz1988}
Nicholas~M. Katz.
\newblock {\em Gauss sums, {K}loosterman sums, and monodromy groups}, volume
  116 of {\em Annals of Mathematics Studies}.
\newblock Princeton University Press, Princeton, NJ, 1988.

\bibitem[Kat90]{Katz1990}
Nicholas~M. Katz.
\newblock {\em Exponential sums and differential equations}, volume 124 of {\em
  Annals of Mathematics Studies}.
\newblock Princeton University Press, Princeton, NJ, 1990.

\bibitem[Kat96]{KatzRigidLocalSystems}
Nicholas~M. Katz.
\newblock {\em Rigid local systems}, volume 139 of {\em Annals of Mathematics
  Studies}.
\newblock Princeton University Press, Princeton, NJ, 1996.

\bibitem[KL13]{knightly_kuznetsovs_2013}
A.~Knightly and C.~Li.
\newblock Kuznetsov's trace formula and the {H}ecke eigenvalues of {M}aass
  forms.
\newblock {\em Mem. Amer. Math. Soc.}, 224(1055):vi+132, 2013.

\bibitem[KMV00]{KMVdK}
E.~Kowalski, P.~Michel, and J.~VanderKam.
\newblock Mollification of the fourth moment of automorphic {$L$}-functions and
  arithmetic applications.
\newblock {\em Invent. Math.}, 142(1):95--151, 2000.

\bibitem[Kni25]{Knightly}
Andrew Knightly.
\newblock Counting locally supercuspidal newforms.
\newblock {\em Essent. Number Theory}, 4(2):349--438, 2025.

\bibitem[Kwa24]{KwanI}
Chung-Hang Kwan.
\newblock Spectral moment formulae for {${\rm GL}(3) \times{\rm GL}(2)$}
  {$L$}-functions {I}: {T}he cuspidal case.
\newblock {\em Algebra Number Theory}, 18(10):1817--1862, 2024.

\bibitem[Kwa25a]{KwanEisenstein}
Chung-Hang Kwan.
\newblock Spectral moment formulae for {$GL(3)\times GL(2)$} {$L$}-functions
  {II}: The {E}isenstein case.
\newblock {\em arXiv preprint}, arXiv:2310.09419, 2025.

\bibitem[Kwa25b]{KwanIII}
Chung-Hang Kwan.
\newblock Spectral moment formulae for {$GL(3)\times GL(2)$} {$L$}-functions
  {III}: the twisted case.
\newblock {\em Math. Ann.}, 391(1):363--398, 2025.

\bibitem[KY21]{KiralYoung}
Eren~Mehmet K\i{}ral and Matthew Young.
\newblock The fifth moment of modular {$L$}-functions.
\newblock {\em J. Eur. Math. Soc. (JEMS)}, 23(1):237--314, 2021.

\bibitem[Lau87]{LaumonFT}
G.~Laumon.
\newblock Transformation de {F}ourier, constantes d'\'equations fonctionnelles
  et conjecture de {W}eil.
\newblock {\em Inst. Hautes \'Etudes Sci. Publ. Math.}, (65):131--210, 1987.

\bibitem[MV10]{MichelVenkateshGL2}
Philippe Michel and Akshay Venkatesh.
\newblock The subconvexity problem for {${\rm GL}_2$}.
\newblock {\em Publ. Math. Inst. Hautes \'{E}tudes Sci.}, (111):171--271, 2010.

\bibitem[Nel20]{NelsonCubic}
Paul~D. Nelson.
\newblock Eisenstein series and the cubic moment for {PGL}(2).
\newblock {\em arXiv preprint}, arXiv:1911.06310, 2020.

\bibitem[Neu99]{Neukirch}
J\"{u}rgen Neukirch.
\newblock {\em Algebraic number theory}, volume 322 of {\em Grundlehren der
  Mathematischen Wissenschaften [Fundamental Principles of Mathematical
  Sciences]}.
\newblock Springer-Verlag, Berlin, 1999.
\newblock Translated from the 1992 German original and with a note by Norbert
  Schappacher, With a foreword by G. Harder.

\bibitem[Pal12]{Palm}
Marc~R. Palm.
\newblock {\em Explicit {${\rm GL}(2)$} trace formulas and uniform, mixed
  {W}eyl laws}.
\newblock PhD thesis, Georg-August-Universit\"{a}t G\"{o}ttingen, 2012.
\newblock arXiv:1212.4282.

\bibitem[Pet15]{PetrowTwistedMotohashi}
Ian~N. Petrow.
\newblock A twisted {M}otohashi formula and {W}eyl-subconvexity for
  {$L$}-functions of weight two cusp forms.
\newblock {\em Math. Ann.}, 363(1-2):175--216, 2015.

\bibitem[PY19]{PetrowYounghybrid}
Ian Petrow and Matthew~P. Young.
\newblock A generalized cubic moment and the {P}etersson formula for newforms.
\newblock {\em Math. Ann.}, 373(1-2):287--353, 2019.

\bibitem[PY20]{PetrowYoungWeyl}
Ian Petrow and Matthew~P. Young.
\newblock The {W}eyl bound for {D}irichlet {$L$}-functions of cube-free
  conductor.
\newblock {\em Ann. of Math. (2)}, 192(2):437--486, 2020.

\bibitem[PY23]{PetrowYoungCoset}
Ian Petrow and Matthew~P. Young.
\newblock The fourth moment of {D}irichlet {$L$}-functions along a coset and
  the {W}eyl bound.
\newblock {\em Duke Math. J.}, 172(10):1879--1960, 2023.

\bibitem[Rio06]{Rio}
Anna Rio.
\newblock Dyadic exercises for octahedral extensions. {II}.
\newblock {\em J. Number Theory}, 118(2):172--188, 2006.

\bibitem[Sch02]{Schmidt:02a}
Ralf Schmidt.
\newblock Some remarks on local newforms for $\rm{GL}(2) $.
\newblock {\em Journal of the Ramanujan Mathematical Society}, 17(2):115--147,
  2002.

\bibitem[Ser79]{SerreLocalFields}
Jean-Pierre Serre.
\newblock {\em Local fields}, volume~67 of {\em Graduate Texts in Mathematics}.
\newblock Springer-Verlag, New York-Berlin, 1979.
\newblock Translated from the French by Marvin Jay Greenberg.

\bibitem[SFFK23]{FFKS_conductor}
Will Sawin, A.~Forey, J.~Fres\'an, and E.~Kowalski.
\newblock Quantitative sheaf theory.
\newblock {\em J. Amer. Math. Soc.}, 36(3):653--726, 2023.

\bibitem[Tun78]{TunnellLLC}
Jerrold~B. Tunnell.
\newblock On the local {L}anglands conjecture for {$GL(2)$}.
\newblock {\em Invent. Math.}, 46(2):179--200, 1978.

\bibitem[Wal85]{Waldspurger}
J.-L. Waldspurger.
\newblock Quelques propri\'{e}t\'{e}s arithm\'{e}tiques de certaines formes
  automorphes sur {${\rm GL}(2)$}.
\newblock {\em Compositio Math.}, 54(2):121--171, 1985.

\bibitem[WX23]{WuXi}
Han Wu and Ping Xi.
\newblock A uniform {W}eyl bound for {$L$}-functions of {H}ilbert modular
  forms.
\newblock {\em arXiv preprint}, arXiv:2302.14652, 2023.

\bibitem[You17]{YoungHybrid}
Matthew~P. Young.
\newblock Weyl-type hybrid subconvexity bounds for twisted {$L$}-functions and
  {H}eegner points on shrinking sets.
\newblock {\em J. Eur. Math. Soc. (JEMS)}, 19(5):1545--1576, 2017.

\end{thebibliography}
\end{document}